\documentclass[journal]{IEEEtran}

\usepackage[ruled]{algorithm2e}
\usepackage{algorithmic}

\usepackage[colorlinks=true,linkcolor=blue,citecolor=blue,urlcolor=blue]{hyperref}
\usepackage{graphicx}
\graphicspath{{figures/}}

\usepackage{amsmath}
\usepackage{cases}
\usepackage{amssymb}
\usepackage{amsfonts}
\usepackage{mathtools}

\usepackage{amsthm}

\theoremstyle{definition}
\newtheorem{lemma}{Lemma}
\newtheorem{theorem}{Theorem}
\newtheorem{proposition}{Proposition}

\newtheorem{corollary}{Corollary}
\theoremstyle{definition}
\newtheorem{definition}{Definition}
\newtheorem{assumption}{Assumption}
\newtheorem*{remark}{Remark}
\newtheorem{example}{Example}

\newcommand\norm[2][1]{\left\|#2\right\|_{#1}}
\newcommand\abs[1]{\left|#1\right|}
\newcommand\where[2]{\left.#1\right|_{#2}}
\newcommand\difffrac[3][1]{
	\ifnum #1=1
		\frac{\mathrm{d} #2}{\mathrm{d} #3}
	\else
	\frac{{\mathrm{d}}^{#1} #2}{\mathrm{d} #3^{#1}}
	\fi
}

\newcommand\R{{\mathbb{R}}}
\newcommand\N{{\mathbb{N}}}

\newcommand{\Setminus}[2]{{\left.#1\middle\backslash #2\right.}}
\newcommand\hamilton{{\mathcal{H}}}
\newcommand\st{{\mathrm{s.t.}}}

\renewcommand\vector[1]{\boldsymbol{#1}}
\newcommand\va{{\vector{a}}}
\newcommand\vb{{\vector{b}}}
\newcommand\vc{{\vector{c}}}
\newcommand\vd{{\vector{d}}}
\newcommand\ve{{\vector{e}}}
\newcommand\vf{{\vector{f}}}
\newcommand\vg{{\vector{g}}}

\newcommand\vt{{\vector{t}}}
\newcommand\vx{{\vector{x}}}

\newcommand\vA{{\vector{A}}}
\newcommand\vB{{\vector{B}}}
\newcommand\vC{{\vector{C}}}

\newcommand\vF{{\vector{F}}}

\newcommand\vH{{\vector{H}}}
\newcommand\vI{{\vector{I}}}
\newcommand\vJ{{\vector{J}}}

\newcommand\vP{{\vector{P}}}

\newcommand\vphi{{\vector{\phi}}}
\newcommand\vlambda{{\vector{\lambda}}}
\newcommand\veta{{\vector{\eta}}}

\newcommand\vzero{{\vector{0}}}

\newcommand\tf{{t_\mathrm{f}}}
\newcommand\um{{u_\mathrm{m}}}
\newcommand\e{{\mathrm{e}}}
\newcommand\m{{\mathrm{m}}}
\newcommand\f{{\mathrm{f}}}
\newcommand\rank{{\mathrm{rank}}}
\newcommand\vxf{{\vx_\mathrm{f}}}
\newcommand\sgn{{\mathrm{sgn}}}
\newcommand\const{{\mathrm{const}}}

\newcommand\systembehavior{{\mathcal{S}}}
\newcommand\tangentmarker{{\mathcal{T}}}
\newcommand\switchinglaw{{\mathcal{L}}}
\newcommand\addendconstraint{{\mathcal{E}}}

\usepackage{xcolor}

\newcommand\RomanNum[1]{\uppercase\expandafter{\romannumeral #1}}

\begin{document}

\title{A Novel State-Centric Necessary Condition for Time-Optimal Control of Controllable Linear Systems Based on Augmented Switching Laws (Extended Version)}

\author{Yunan~Wang,~
        Chuxiong~Hu,~\IEEEmembership{Senior~Member,~IEEE,}
        Yujie~Lin,
        Zeyang~Li,
        Shize~Lin,
        and~Suqin~He
\thanks{Corresponding author: Chuxiong Hu (e-mail: cxhu@tsinghua.edu.cn).
}}

\markboth{IEEE TRANSACTIONS ON AUTOMATIC CONTROL}{}

\maketitle

\begin{abstract}
    Most existing necessary conditions for optimal control based on adjoining methods require both state and costate information, yet the unobservability of costates for a given feasible trajectory impedes the determination of optimality in practice. This paper establishes a novel theoretical framework for time-optimal control of controllable linear systems with a single input, proposing the augmented switching law (ASL) that represents the input control and the feasibility in a compact form. Given a feasible trajectory, the perturbed trajectory under the constraints of ASL is guaranteed to be feasible, resulting in a novel state-centric necessary condition without dependence on costate information. A first-order necessary condition is proposed that the Jacobian matrix of the ASL is not of full row rank, which also results in a potential approach to optimizing a given feasible trajectory with the preservation of arc structures. The proposed necessary condition is applied to high-order chain-of-integrator systems with full box constraints, contributing to some theoretical results challenging to reason by costate-based conditions.

\end{abstract}

\begin{IEEEkeywords}
    Optimal control, linear systems, variational methods, switched systems, necessary condition.

\end{IEEEkeywords}

\IEEEpeerreviewmaketitle

\section{Introduction}\label{sec:Introduction}

\IEEEPARstart{T}{ime-optimal} control for controllable linear systems achieves universal applications in aerospace \cite{trelat2012optimal}, manufacturing \cite{wang2023optimization,wang2025consistency}, robotic control \cite{zhao2020pareto,wang2022learning}, and autonomous driving \cite{guler2016adaptive}. Numerous works have been conducted on the behaviors of the optimal controls, resulting in some well-known necessary or sufficient conditions \cite{hartl1995survey}. However, it remains difficult to fully solve a high-order problem in practice, especially when state inequality constraints are introduced \cite{maurer1977optimal,jacobson1971new}. Specifically, when given a feasible trajectory, it is challenging to determine the trajectory's optimality based on state information since most existing necessary conditions require costate information, especially when the trajectory is not planned by costate-based methods.

In the domain of optimal control theory, numerous necessary conditions are developed based on the variational method \cite{ma2021optimal} and Pontryagin's maximum principle (PMP) \cite{pontryagin1987mathematical}. The former one provides a necessary condition on the extreme of a functional in an open feasible set, while PMP has the potential to deal with inequality constraints on states or controls \cite{makowski1974optimal}. Among them, the state-control inequality constraints pose a challenge in solution since they involve the connection of different arcs. Chang \cite{chang1963optimal}, Dreyfus \cite{dreyfus1960dynamic}, and other scholars developed the direct adjoining method for the connection of unconstrained arcs and constrained arcs, i.e., the junction conditions of costates. Jacobson et al. \cite{jacobson1971new} pointed out that constrained arcs are not allowed when the state constraint is of odd order $p>1$, and instead, the unconstrained arc is tangent to the constraint's boundary at a single point. Makowski and Neustadt \cite{makowski1974optimal} generalized the PMP-based necessary condition to the mixed constraints of state and control. Bryson et al. \cite{bryson1963optimal} developed the necessary condition based on the indirect adjoining method, further supplemented by Kreindler \cite{kreindler1982additional}. The necessary conditions are usually applied to rule out some forms of non-optimal controls in practice. However, the above necessary conditions require information on both states and costates, while costates are usually unobservable in practice, resulting in a challenge in determining the optimality of a given feasible trajectory by costate-based conditions.

\begin{figure*}[!t]
    \centering
    \includegraphics[width=0.85\textwidth]{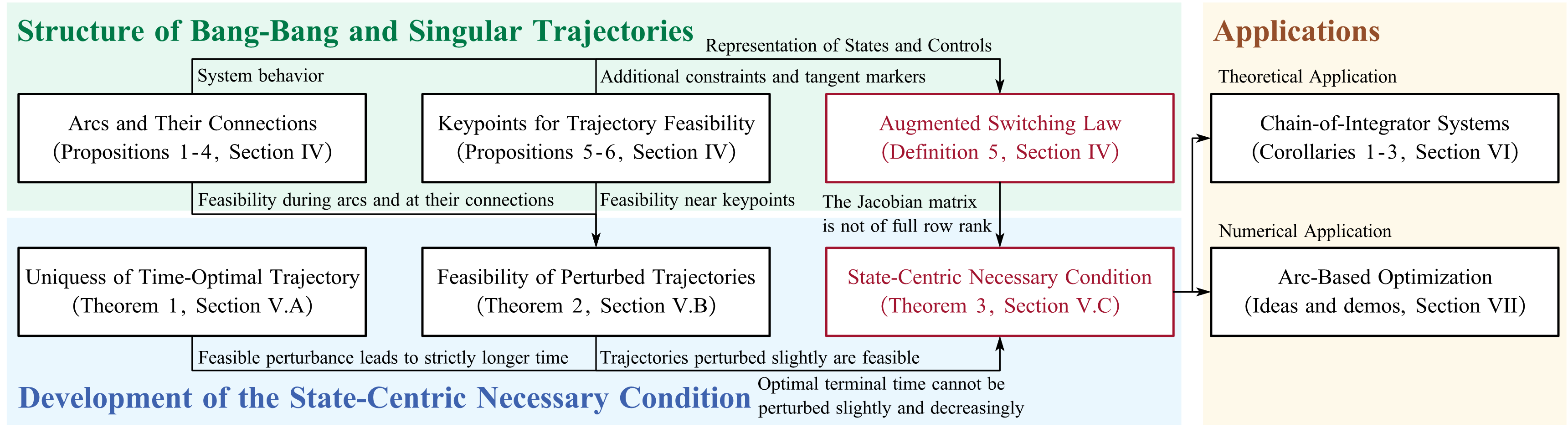}
    \vspace{-2mm}
    \caption{Graphical roadmap of the proposed theoretical framework. Main contributions regarding representations and optimal conditions are highlighted in red.}
    \vspace{-2mm}
    \label{fig:RoadMap_Theorems}
\end{figure*}

Based on PMP, one can get rid of dependence on costates by deriving the optimal trajectory's structure. Switching surfaces can be constructed in the state space for time-optimal control of linear systems. Walther et al. \cite{walther2001computation} computed switching surfaces for some systems using a Gr\"oebner basis. Patil et al. \cite{patil2015computation} represented switching surfaces using state  equations for single-input diagonal linear systems without state constraints. However, for systems with state constraints, it remains challenging to completely construct switching surfaces due to the lack of optimal switching laws. For example, He et al. \cite{he2020time} provided explicit equations for switching surfaces of 3rd-order chain-of-integrator systems with full state constraints, limited to zero terminal states. Wang et al. \cite{wang2025time} extended \cite{he2020time} to non-zero terminal conditions, failing to achieve time-optimality for higher-order systems. Numerical methods \cite{yury2016quasi} face the curse of dimensionality. Hence, there is a lack of efficient and general optimality conditions without dependence on costates for high-order systems with state constraints.

Another type of necessary condition leverages Bellman's principle of optimality \cite{bellman1952theory}, asserting that a sub-arc of an optimal trajectory is itself optimal for the induced sub-problem, without explicitly involving costates. With wide applications in optimal control \cite{tedone2020hamilton,li2024safe}, the Hamilton-Jacobi-Bellman (HJB) equation \cite{bardi1997optimal} serves as a fundamental necessary condition which is also sufficient under suitable regularity conditions. Evans and James \cite{evans1989hamiltonian} investigated the HJB equation for time-optimal control. Wolenski and Zhuang \cite{wolenski1998proximal} showed that the minimal time function is the unique proximal solution to the HJB equation. However, most numerical HJB-based methods require high computational cost \cite{sun2015convergence}. From a theoretical perspective, HJB-based methods also fail to preserve arc structures, which hinders the analytical reasoning.

This paper sets out to establish a theoretical framework for time-optimal control of controllable single-input linear systems and develops a novel \textit{state-centric} necessary condition that requires state information without dependence on costate information. A graphical roadmap is shown in Fig. \ref{fig:RoadMap_Theorems}. Sections \ref{sec:ArcAnalysis} and \ref{sec:KeypointsFeasibilityOptimalTrajectory} derive trajectories' structures on arcs and keypoints for feasibility, respectively. Guided by the knowledge of trajectory structures, Section \ref{sec:StateCentricNecessaryCondition} develops the state-centric necessary condition which is applied from theoretical and numerical perspectives in Sections \ref{sec:Applications} and \ref{sec:simulation} as examples, respectively. Proofs of propositions and theorems are provided in Appendix \ref{sec:proofs}. The contributions of this paper are as follows:
\begin{enumerate}
    \item This paper establishes a novel theoretical framework for time-optimal control of controllable single-input linear systems, proposing the augmented switching law (ASL) which represents structures and feasibility of bang-bang and singular (BBS) trajectories in a compact form. Compared to existing representations of switching surfaces, the ASL can generally deal with high-order systems with state constraints. The developed framework provides a novel variational approach with feasibility assurance, i.e., perturbing the motion time of each arc with a fixed ASL and resulting in feasible perturbed BBS trajectories.

    \item Based on the developed framework, this paper proposes a novel state-centric necessary condition for time-optimality, asserting that the Jacobian matrix induced by the ASL is not of full row rank. The condition is derived from the fact that if a BBS trajectory is optimal, then all perturbed feasible trajectories have a strictly longer terminal time than the original one. Compared to existing costate-based conditions, the proposed state-centric necessary condition only requires state information without dependence on costates. Note that costates are usually unobservable in practice, especially when the trajectory is not planned by costate-based optimization.

    \item This paper proposes a potential optimization idea based on the developed necessary condition. Given a feasible BBS trajectory as the initial value, the motion time is optimized with a fixed ASL until the Jacobian matrix is not of full row rank. During iteration, the feasibility is guaranteed based on the developed ASL framework. Compared to discretized methods, the proposed approach can preserve the BBS arc structure and has the potential to significantly reduce control oscillations and open-loop integration errors of terminal states.
    \item The proposed state-centric condition is applied to high-order chain-of-integrator systems with full state constraints which remains an open problem in the optimal control domain. An upper bound on the number of arcs is determined in a general sense, while case-by-case analysis is necessary in traditional methods. A recursive equation for the junction time when the state is tangent to the constraints' boundaries is derived in chattering arcs induced by 2nd-order constraints, which is challenging to prove by costate analysis. As a result, one could determine the existence of chattering induced by 2nd-order constraints in chain-of-integrator systems in future works. The above results demonstrate the theoretical effectiveness of the proposed necessary condition.
\end{enumerate}

\section{Problem Formulation}\label{sec:ProblemFormulation}

Some notations should be introduced. For an integer $n\in\N$, the set $\left[n\right]$ refers to $\left\{1,2,\dots,n\right\}$. $\left(x_i\right)_{i\in\mathcal{I}}\in\R^{\abs{\mathcal{I}}}$ is a column vector and $\left(x_{ij}\right)_{i\in\mathcal{I},j\in\mathcal{J}}\in\R^{\abs{\mathcal{I}}\times\abs{\mathcal{J}}}$ is a matrix, where $\mathcal{I}$ and $\mathcal{J}$ are index sets. Specifically, $\left(x_i\right)_{i=1}^n\in\R^n$ refers to $\left(x_i\right)_{i\in\left[n\right]}$.

This section formulates the time-optimal control problem for controllable linear systems and introduces some necessary assumptions. The optimal control problem is as follows:
\begin{IEEEeqnarray}{rl}\label{eq:optimalproblem}
    \min_{u}\quad& \tf=\int_{0}^{\tf}\mathrm{d}t,\label{eq:optimalproblem_objective}\IEEEyesnumber\IEEEyessubnumber*\\
    \st\quad&\dot{\vx}\left(t\right)=\vA x\left(t\right)+\vb u\left(t\right),\,\forall t\in\left[0,t_\f\right],\label{eq:optimalproblem_dynamics}\\
    &\vC\vx\left(t\right)+\vd\leq\vzero,\,\forall t\in\left[0,t_\f\right],\label{eq:optimalproblem_inequality_x}\\
    &\vx\left(0\right)=\vx_0,\,\vx\left(\tf\right)=\vx_\tf,\\
    &\abs{u\left(t\right)}\leq \um,\,\forall t\in\left[0,t_\f\right],\label{eq:optimalproblem_inputconstraint}
\end{IEEEeqnarray}
where $\vA\in\R^{n\times n}$, $\vb\in\R^{n}$, $\vC=\left(\vc_p\right)_{p=1}^P\in\R^{P\times n}$, and $\vd=\left(d_p\right)_{p=1}^P\in\R^{P}$, where $\forall p\in\left[P\right]$, $\vc_{p}\not=\vzero$. In problem \eqref{eq:optimalproblem}, the terminal time $\tf$ is free. $\vx=\left(x_k\right)_{k=1}^n\in\R^n$ is the state vector. The state constraint \eqref{eq:optimalproblem_inequality_x} requires that $\forall p\in\left[P\right]$, $\vc_p^\top\vx+d_p\leq0$. The input control $u$ is subject to a box constraint \eqref{eq:optimalproblem_inputconstraint}.

Assume that the system \eqref{eq:optimalproblem_dynamics} is controllable, i.e.,
\begin{equation}\label{eq:controllable}
    \rank\left[\vb,\vA\vb,\vA^2\vb,\dots,\vA^{n-1}\vb\right]=n.
\end{equation}

\begin{remark}
    Generally, an equality constraint $\vF\vx\left(t\right)+\vg=\vzero$ can be incorporated into problem \eqref{eq:optimalproblem}. However, it is not considered in this paper since the equality constraint can be equivalently eliminated by a linear transformation.
\end{remark}

In order to solve \eqref{eq:optimalproblem}, the Hamiltonian is constructed as
\begin{equation}\label{eq:hamilton}
    \begin{aligned}
        &\hamilton\left(\vx\left(t\right),u\left(t\right),\lambda_0,\vlambda\left(t\right),\veta\left(t\right),t\right)\\
        =&\lambda_0+\vlambda^\top\left(\vA\vx+\vb u\right)+\veta^\top\left(\vC\vx+\vd\right).
    \end{aligned}
\end{equation}
Among them, $\lambda_0\geq0$ is a constant. $\vlambda=\left(\lambda_k\right)_{k=1}^n$ is the costate vector. $\forall t\in\left[0,\tf\right]$, $\left(\lambda_0,\vlambda\left(t\right)\right)\in\R^{n+1}$ is not $\vzero$ \cite{maurer1977optimal}. The Euler-Lagrange equations \cite{maurer1977optimal} holds that
\begin{equation}\label{eq:dotlambda}
    \dot\vlambda=-\frac{\partial\hamilton}{\partial\vx}=-\vA^\top\vlambda-\vC^\top\veta.
\end{equation}
In \eqref{eq:hamilton}, $\veta=\left(\eta_p\right)_{p=1}^P$ is the multiplier vector induced by the state inequality constraint \eqref{eq:optimalproblem_inequality_x}, satisfying
\begin{equation}\label{eq:eta_constraint_zero}
    \eta_p\geq0,\,\eta_{p}\left(\vc_p^\top\vx+d_p\right)=0,\,\forall p\in\left[P\right].
\end{equation}
Therefore, $\forall t\in\left[0,t_\f\right]$, $\eta_p\left(t\right)>0$ only if $\vc_p^\top\vx\left(t\right)+d_p=0$.

Pontryagin's maximum principle (PMP) \cite{hartl1995survey} states that the optimal control $u\left(t\right)$ minimizes the Hamiltonian $\hamilton$, i.e.,
\begin{equation}
    u\left(t\right)\in\mathop{\arg\min}\limits_{\abs{U}\leq\um}\hamilton\left(\vx\left(t\right),U,\lambda_0,\vlambda\left(t\right),\veta\left(t\right),t\right).
\end{equation}
By \eqref{eq:hamilton}, $\hamilton$ is affine in $u$; hence,
\begin{equation}\label{eq:bang_singular_bang_law}
    u\left(t\right)=\begin{dcases}
        \um,&\vb^\top\vlambda\left(t\right)<0,\\
        *,&\vb^\top\vlambda\left(t\right)=0,\\
        -\um,&\vb^\top\vlambda\left(t\right)>0.
    \end{dcases}
\end{equation}
$u\left(t\right)\in\left[-\um,\um\right]$ is undetermined during $\vb^\top\vlambda\equiv0$, which is called a singular arc. The notation ``$\equiv$'' means that the equality holds for a continuous period. \eqref{eq:bang_singular_bang_law} is the well-known bang-bang and singular (BBS) control law \cite{pontryagin1987mathematical,silva2010smooth}. In other words, the optimal control is always $\pm\um$, except during singular arcs.

Note that the objective function $\tf=\int_{0}^{\tf}\mathrm{d}t$ is in a Lagrangian form; hence, the continuity of the system is guaranteed by the following equality, i.e.,
\begin{equation}\label{eq:hamilton_equiv_0}
    \forall t\in\left[0,t_\f\right],\,\hamilton\left(\vx\left(t\right),u\left(t\right),\lambda_0,\vlambda\left(t\right),\veta\left(t\right),t\right)\equiv0.
\end{equation}

\begin{assumption}\label{assumption:feasible_optimal}
    Assume that problem \eqref{eq:optimalproblem} is feasible, and the optimal control exists. Furthermore, assume that the chattering phenomenon does not occur in the optimal profile.
\end{assumption}

\begin{remark}
    In Assumption \ref{assumption:feasible_optimal}, the chattering phenomenon \cite{zelikin2012theory} means that the optimal control switches infinitely many times in a finite period. At a one-sided neighborhood of a chattering limit time point, an infinite number of constrained arcs are joined at the boundary of inequality state constraints, where chattering limit time points are usually isolated \cite{zelikin2012theory}. If a chattering phenomenon occurs, then, by Bellman's principle of optimality \cite{bellman1952theory}, a sub-arc of the optimal trajectory is also optimal in the sub-problem. Hence, a sub-arc that contains no chattering limit time points can be considered in this paper.
\end{remark}

In the following of this paper, Assumption \ref{assumption:feasible_optimal} is considered to hold throughout unless otherwise specified.

Given a feasible trajectory $\left(\vx\left(t\right),u\left(t\right),\lambda_0,\vlambda\left(t\right),\veta\left(t\right)\right)$, one can determine the optimality through PMP-based necessary conditions like \eqref{eq:dotlambda}--\eqref{eq:hamilton_equiv_0}. However, in practice, only $\left(\vx\left(t\right),u\left(t\right)\right)$ is observable, while $\left(\lambda_0,\vlambda\left(t\right),\veta\left(t\right)\right)$ is difficult to fully solve based on provided $\left(\vx\left(t\right),u\left(t\right)\right)$. Consequently, the focus of this paper is to determine the optimality of a feasible trajectory in problem \eqref{eq:optimalproblem} using the provided $\left(\vx\left(t\right),u\left(t\right)\right)$.

\section{Arc Representation of the Optimal Trajectory}\label{sec:ArcAnalysis}

An arc of a trajectory is defined as a continuous segment conditioned by a fixed dynamic equation, i.e., $\dot\vx=\widehat{\vA}\vx+\widehat{\vb}$ where $\widehat{\vA}$ and $\widehat{\vb}$ are constant. Two kinds of arcs, i.e., unconstrained arcs and constrained arcs, are discussed in Section \ref{subsec:UnconstrainedArc}. Through discussion on the two arcs, the system behavior representing a single arc is defined in Section \ref{subsec:SystemBehavior}. Finally, an optimal non-chattering trajectory can be represented by a switching law, i.e., a sequence of system behaviors, as discussed in Section \ref{subsec:SwitchingLaw}.

\subsection{Unconstrained and Constrained Arcs in Problem \eqref{eq:optimalproblem}}\label{subsec:UnconstrainedArc}
    In an unconstrained arc, $\vC\vx+\vd<\vzero$ holds almost everywhere (a.e.), and the control does not switch. Proposition \ref{prop:unconstrainedarc_nosingular} states that $u\equiv\pm\um$ holds in an unconstrained arc.

    \begin{proposition}\label{prop:unconstrainedarc_nosingular}
        In an unconstrained arc, it holds a.e. that $\vb^\top\vlambda\not=0$ and $u\equiv-\um\sgn\left(\vb^\top\vlambda\right)$.
    \end{proposition}

    \begin{remark}
        In an unconstrained arc, $\dot\vlambda=-\vA^\top\vlambda$ holds since $\veta\equiv\vzero$. Two cases could occur: (a) $\vb^\top\vlambda\equiv0$ holds for a period; (b) $\vb^\top\vlambda$ crosses 0 for finite times. Case (a) is ruled out in an unconstrained arc by Proposition \ref{prop:unconstrainedarc_nosingular}. In Case (b), the control switches between $u\equiv\um$ and $u\equiv-\um$ when $\vb^\top\vlambda$ crosses 0, where two unconstrained arcs are connected.
    \end{remark}

    Compared to unconstrained arcs, constrained arcs exhibit more complex behavior. In a constrained arc, $\exists p\in\left[P\right]$, s.t. $\vc_p^\top\vx+d_p\equiv0$. Proposition \ref{prop:constrainedarc_control} characterizes a constrained arc. Proposition \ref{prop:consistent_constraint} points out that multiple inequality constraints are allowed to be active in a constrained arc.

    \begin{proposition}\label{prop:constrainedarc_control}
        If $\vc_p^\top\vx+d_p\equiv0$ holds in a constrained arc for $t\in\left[t_1,t_2\right]$, then $\exists r_p\in\left[n\right]$, s.t. $\forall r\in\left[r_p-1\right]$, $\vc_p^\top\vA^{r-1}\vb=0$, and $\vc_p^\top\vA^{r_p-1}\vb\not=0$. During the constrained arc, it holds that
        \begin{equation}\label{eq:constrainedarc_control}
            u=\widehat{\va}_p^\top\vx,\,\dot\vx=\left(\vA+\vb\widehat{\va}_p^\top\right)\vx,\,\widehat{\va}_p\triangleq-\frac{\left(\vc_p^\top\vA^{r_p}\right)^\top}{\vc_p^\top\vA^{r_p-1}\vb},
        \end{equation}
        and $\forall t\in\left[t_1,t_2\right]$,
        \begin{equation}\label{eq:constrainedarc_highorder_constraint}
            \begin{dcases}
                \vc_p^\top\vx+d_p=0,\\
                \forall r\in\left[r_p-1\right],\,\vc_p^\top\vA^{r}\vx=0.
            \end{dcases}
        \end{equation}
        $r_p$ is called the \textit{order} of the active constraint $\vc_p^\top\vx+d_p\equiv0$.

        Conversely, if $\exists t_0\in\left[t_1,t_2\right]$, s.t. $\vx\left(t_{0}\right)$ satisfies \eqref{eq:constrainedarc_highorder_constraint}, then under the driving of \eqref{eq:constrainedarc_control}, $\vc_p^\top\vx+d_p\equiv0$ holds in $\left[t_1,t_2\right]$.
    \end{proposition}

    \begin{lemma}\label{lemma:solution_linearODEs}
        $\vA\in\R^{n\times n}$ and $\vb\in\R^n$. The initial value problem where $\dot\vx=\vA\vx+\vb$, $\vx\left(t_0\right)=\vx_0$ has a unique solution:
        \begin{equation}\label{eq:solution_linearODEs}
            \forall t\in\R,\,\vx\left(t\right)=\e^{\vA\left(t-t_0\right)}\vx_0+\left(\int_{t_0}^{t}\e^{\vA\left(t-\tau\right)}\mathrm{d}\tau\right)\vb.
        \end{equation}
        Furthermore, if $\vc\in\R^n$ satisfies $\forall r\in\left[n\right]$, $\frac{\mathrm{d}^{r}\left(\vc^\top\vx\right)}{\mathrm{d}t^{r}}\left(t_0\right)=0$ holds, then $\vc^\top\vx\left(t\right)\equiv\const$ on $\R$.
    \end{lemma}

    \begin{proposition}\label{prop:consistent_constraint}
        Consider a constrained arc for $t\in\left[t_1,t_2\right]$. $\mathcal{P}\subset\left[P\right]$. Assume that during the arc, $\forall p\in\mathcal{P}$, $\vc_p^\top\vx+d_p\equiv0$ holds. Then, $\forall p\in\mathcal{P}$, $t\in\left[t_1,t_2\right]$, it holds that
        \begin{equation}\label{eq:constrainedarc_highorder_constraint_multiple}
            \begin{dcases}
                \vc_p^\top\vx+d_p=0,\\
                \forall q\in\mathcal{P},\,r\in\left[n\right],\,\vc_q^\top\left(\vA+\vb\widehat{\va}_p^\top\right)^r\vx=0.
            \end{dcases}
        \end{equation}
        Active constraints $\vc_p^\top\vx+d_p\equiv0$ for $p\in\mathcal{P}$ are called \textit{consistent} if linear equations \eqref{eq:constrainedarc_highorder_constraint_multiple} have a feasible solution $\vx$.

        Conversely, assume that $\exists t_0\in\left[t_1,t_2\right]$, s.t. $\vx\left(t_0\right)$ satisfies \eqref{eq:constrainedarc_highorder_constraint_multiple}. Then, under the driving of \eqref{eq:constrainedarc_control} with an arbitrary $p\in\mathcal{P}$, it holds that $\forall q\in\mathcal{P}$, $\vc_q^\top\vx+d_q\equiv0$ in $\left[t_1,t_2\right]$.
    \end{proposition}

    \begin{example}
        Consider a chain-of-integrator system of order $n$:
        \begin{equation}\label{eq:dynamics_chainofintegrators}
            \begin{dcases}
                \dot x_{k}\left(t\right)=x_{k-1}\left(t\right),&t\in\left[0,\tf\right],\, 1<k\leq n,\\
                \dot x_1\left(t\right)=u\left(t\right),&t\in\left[0,\tf\right].
            \end{dcases}
        \end{equation}
        Two inequality constraints are introduced as $x_1+x_3\leq1$ and $x_2\leq0$. $x_1+x_3\equiv1$ and $x_2\equiv0$ are consistent, i.e., a constrained arc is allowed where both $x_1+x_3\equiv0$ and $x_2\equiv0$ hold. In this case, $x_2\equiv0$ implies $x_1\equiv0$ and $u\equiv0$; hence, $x_1+x_3\leq1$ implies $x_3\equiv1$. Similarly, $x_1+x_3\equiv1$ and $x_2\equiv1$ are inconsistent, i.e., in a constrained arc, if $x_1+x_3\equiv1$ holds, then $x_2\equiv1$ cannot hold.
    \end{example}

\subsection{System Behavior}\label{subsec:SystemBehavior}
    \begin{definition}\label{def:systembehavior}
        A \textit{system behavior} in problem \eqref{eq:optimalproblem} is an arc with constant system dynamics, denoted as $\systembehavior=\left(\widehat{\vA},\widehat{\vb},\vF,\vg,\mathcal{P}\right)$. During the arc, $\dot\vx=\widehat{\vA}\vx+\widehat{\vb}$ holds. Induced by Proposition \ref{prop:consistent_constraint}, $\vF\vx+\vg\equiv\vzero$ provides full equality constraints, where $\vF$ has full row rank. $\mathcal{P}$ is the set of consistent active state constraints.
    \end{definition}

    In Definition \ref{def:systembehavior}, an unconstrained arc can be denoted as $\systembehavior=\left(\vA,u_0\vb,\sim,\sim,\varnothing\right)$, where $\sim$ means that no equality constraints are introduced. $u_0\in\left\{\pm\um\right\}$ is a constant. By Lemma \ref{lemma:solution_linearODEs},
    \begin{equation}
        \vx\left(t\right)=\e^{\vA\left(t-t_0\right)}\vx\left(t_0\right)+u_0\left(\int_{t_0}^{t}\e^{\vA\left(t-\tau\right)}\mathrm{d}\tau\right)\vb.
    \end{equation}

    For a constrained arc $\systembehavior=\left(\widehat{\vA},\vzero,\vF,\vg,\mathcal{P}\right)$, $\mathcal{P}\not=\varnothing$. \eqref{eq:constrainedarc_control} implies that $\forall p\in\mathcal{P}$, $\dot\vx=\left(\vA+\vb\widehat{\va}_p^\top\right)\vx$, where $\widehat{\va}_p$ is defined in \eqref{eq:constrainedarc_control}. Although it allows that $\exists p,q\in\mathcal{P}$, s.t. $\widehat{\va}_{p}\not=\widehat{\va}_{q}$, it is guaranteed by the consistency of $\mathcal{P}$ that $\widehat{\va}_p^\top\vx\equiv\widehat{\va}_q^\top\vx$ during the arc. Denote by $\widehat{\vA}\triangleq\widehat{\vA}_{\min\left(\mathcal{P}\right)}$. Then, $\systembehavior$ in Definition \ref{def:systembehavior} is well-defined. By Lemma \ref{lemma:solution_linearODEs}, it holds during $\systembehavior$ that
    \begin{equation}
        \vx\left(t\right)=\e^{\widehat{\vA}\left(t-t_0\right)}\vx\left(t_0\right).
    \end{equation}

    \begin{example}\label{example:chainofintegrators_boxconstraint}
        Consider a chain-of-integrator system \eqref{eq:dynamics_chainofintegrators} with full box state constraints, i.e., $\forall k\in\left[n\right]$, $t\in\left[0,\tf\right]$, $-x_{\mathrm{m}k}\leq x_k\left(t\right)\leq x_{\mathrm{m}k}$. The $\left(2k-1\right)$-th constraint is $-x_k\left(t\right)\leq x_{\mathrm{m}k}$, and the $\left(2k\right)$-th constraint is $x_k\left(t\right)\leq x_{\mathrm{m}k}$. A constrained arc $x_k\equiv x_{\mathrm{m}k}$ is of $k$th-order, where $\forall j\in\left[k-1\right]$, $x_j\equiv0$, and $u\equiv0$. Then, $x_k\equiv x_{\mathrm{m}k}$ can be represented as $\systembehavior=\left(\widehat{\vA},\vzero,\vI_{k\times n},x_{\mathrm{m}k}\ve_k,\left\{2k\right\}\right)$ where $\widehat{\vA}=\left(\hat{a}_{ij}\right)_{i\in\left[n\right],j\in\left[n\right]}$. $\hat{a}_{ij}=1$ if $k<i=j+1\leq n$; otherwise, $\hat{a}_{ij}=0$.
    \end{example}

\begin{figure}[!t]
    \centering
    \includegraphics[width=\columnwidth]{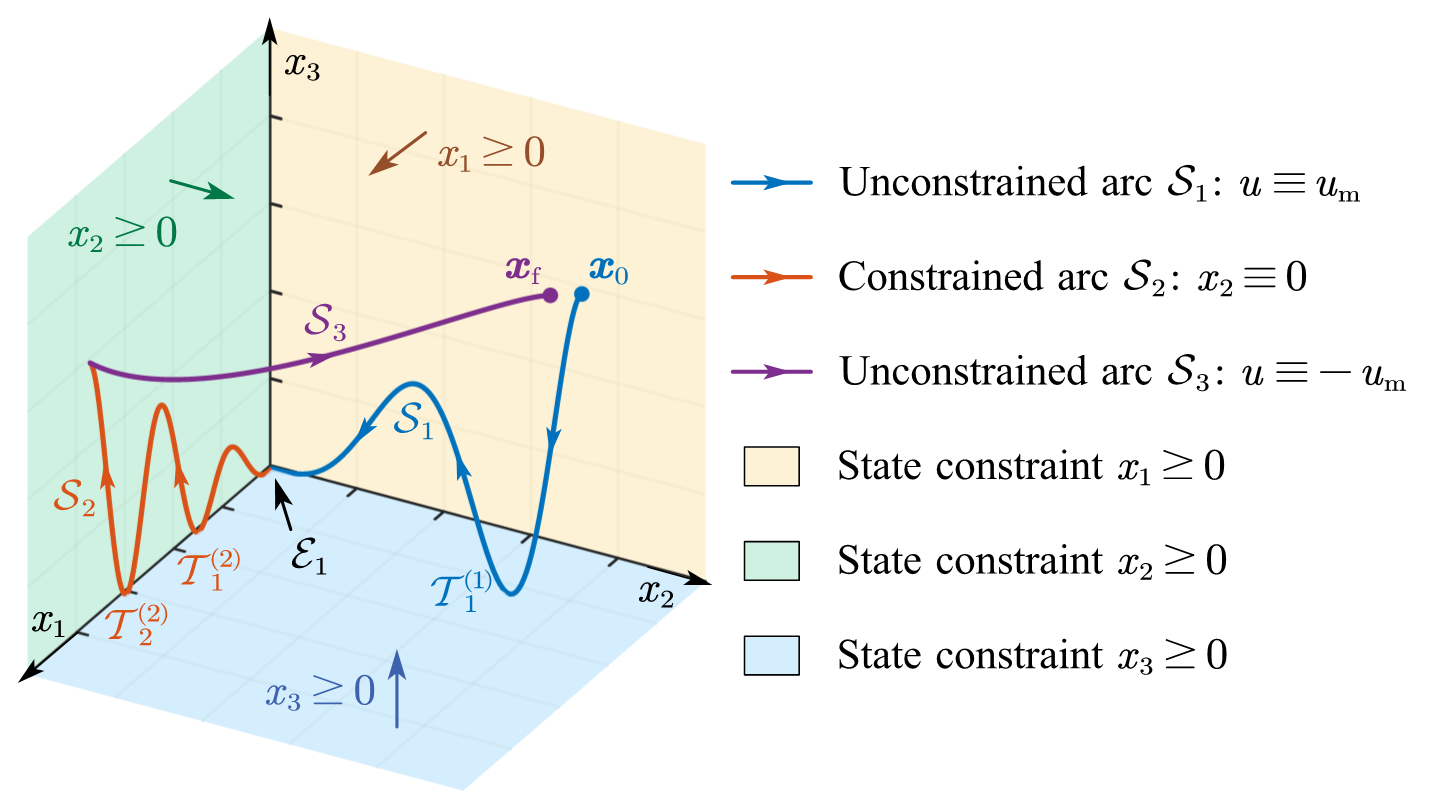}
    \caption{A feasible BBS trajectory under the constraint  $\vx\geq\vzero$. The trajectory consists of 3 arcs, i.e., $\systembehavior_1$, $\systembehavior_2$, and $\systembehavior_3$. $\systembehavior_1$ and $\systembehavior_2$ is connected at $\vzero$, denoted as $\mathcal{E}_1$. The trajectory is tangent to $\left\{x_3=0\right\}$ at $\mathcal{T}_1^{(1)}$, $\mathcal{T}_1^{(2)}$, and $\mathcal{T}_2^{(2)}$.}
    \label{fig:ASL_demo}
\end{figure}

\subsection{Switching Law}\label{subsec:SwitchingLaw}
    Based on the analysis in Section \ref{subsec:SystemBehavior}, if the optimal trajectory consists of a finite number of arcs, then the optimal trajectory can be represented by a sequence of system behaviors. The switching law is defined as follows:

    \begin{definition}\label{def:switchinglaw}
        Assume that a BBS trajectory in problem \eqref{eq:optimalproblem} consists of a finite number of arcs $\systembehavior_1,\systembehavior_2,\dots,\systembehavior_N$ successively. The \textit{switching law} is denoted as $\switchinglaw=\systembehavior_1\systembehavior_2\dots \systembehavior_N$.
    \end{definition}

    The switching law focuses on the arcs that compose the optimal trajectory, while the motion time of each arc is not included in a switching law. For the example in Fig. \ref{fig:ASL_demo}, the switching law is $\systembehavior_1\systembehavior_2\systembehavior_3$. $\systembehavior_1$ and $\systembehavior_3$ are unconstrained arcs where $\vx>\vzero$ holds a.e. $\systembehavior_2$ is a constrained arc where $x_2\equiv0$.

    The following proposition discusses connections of two adjacent arcs, including unconstrained and constrained arcs.

    \begin{proposition}\label{prop:connection}
        Assume that $\systembehavior_i=\left(\widehat{\vA}_i,\widehat{\vb}_i,\vF_i,\vg_i,\mathcal{P}_i\right)$ occurs in $\left[t_{i-1},t_i\right]$, $i=1,2$, where $\left[\widehat{\vA}_1,\widehat{\vb}_1\right]\not=\left[\widehat{\vA}_2,\widehat{\vb}_2\right]$. The following conclusions hold at the connection of $\systembehavior_1$ and $\systembehavior_2$.
        \begin{enumerate}
            \item\label{subprop:connection_unconstrained} If $\systembehavior_1$ and $\systembehavior_2$ are unconstrained arcs, then $\vb^\top\vlambda$ crosses 0 at $t_1$. $\widehat{\vb}_1=u_1\vb$ and $\widehat{\vb}_2=-u_1\vb$, where $u_1\in\left\{\pm\um\right\}$.
            \item\label{subprop:connection_constrained} If $\systembehavior_1$ and $\systembehavior_2$ are constrained arcs, then $\mathcal{P}_1\cap\mathcal{P}_2=\varnothing$. Furthermore, $\forall p\in\mathcal{P}_1$, $\exists \hat{r}_p\in\left[r_p,n\right]\cap\N$, s.t.
            \begin{equation}\label{eq:connection_constrained_s1tos2}
                \begin{dcases}
                    \vc_p^\top\widehat{\vA}_2^{\hat{r}_p}\vx\left(t_1\right)<0,\\
                    \forall r\in\left[\hat{r}_p-1\right],\,\vc_p^\top\widehat{\vA}_2^r\vx\left(t_1\right)=0.
                \end{dcases}
            \end{equation}
            $\forall p\in\mathcal{P}_2$, $\exists \hat{r}_p\in\left[r_p,n\right]\cap\N$, s.t.
            \begin{equation}\label{eq:connection_constrained_s2tos1}
                \begin{dcases}
                    \left(-1\right)^{\hat{r}_p}\vc_p^\top\widehat{\vA}_1^{\hat{r}_p}\vx\left(t_1\right)<0,\\
                    \forall r\in\left[\hat{r}_p-1\right],\,\vc_p^\top\widehat{\vA}_1^r\vx\left(t_1\right)=0.
                \end{dcases}
            \end{equation}
            Conversely, assume that $\dot\vx=\widehat{\vA}_1\vx$ for $\left(t_0,t_1\right)$ and $\dot\vx=\widehat{\vA}_2\vx$ for $\left(t_1,t_2\right)$. If $\vx\left(t_1\right)$ satisfies \eqref{eq:connection_constrained_s1tos2} for $p\in\mathcal{P}_1$, \eqref{eq:connection_constrained_s2tos1} for $p\in\mathcal{P}_2$, and \eqref{eq:constrainedarc_highorder_constraint} for $p\in\mathcal{P}_1\cup\mathcal{P}_2$, then $\exists \varepsilon>0$, $\systembehavior_1$ occurs for $\left(t_1-\varepsilon,t_1\right)$ and $\systembehavior_2$ occurs for $\left(t_1,t_1+\varepsilon\right)$.
            \item\label{subprop:connection_constrained_unconstrained} If $\systembehavior_1$ is a constrained arc and $\systembehavior_2$ is an unconstrained arc, then $\forall p,q\in\mathcal{P}_1$, $\sgn\left(\vc_p^\top\vA^{r_p-1}\vb\right)=\sgn\left(\vc_q^\top\vA^{r_q-1}\vb\right)$. On $\systembehavior_2$, $\forall p\in\mathcal{P}_1$, it holds that
            \begin{equation}\label{eq:connection_unconstrained_constrained_control_s2unconstrained}
                u\equiv u_2=-\sgn\left(\vc_p^\top\vA^{r_p-1}\vb\right)\um.
            \end{equation}
            Furthermore, $\forall p\in\mathcal{P}_1$, $\exists \hat{r}_p\in\left[r_p,n\right]\cap\N$, s.t.
            \begin{equation}\label{eq:connection_unconstrained_constrained_s1tos2}
                \begin{dcases}
                    \vc_p^\top\vA^{\hat{r}_p-1}\left(\vA\vx\left(t_1\right)+u_2\vb\right)<0,\\
                    \forall r\in\left[\hat{r}_p-1\right],\,\vc_p^\top\vA^{r-1}\left(\vA\vx\left(t_1\right)+u_2\vb\right)=0.
                \end{dcases}
            \end{equation}
            Conversely, assume that $\dot\vx=\widehat{\vA}_1\vx$ for $\left(t_0,t_1\right)$ and $\dot\vx=\vA\vx+u_2\vb$ for $\left(t_1,t_2\right)$ where $u_2$ is given in \eqref{eq:connection_unconstrained_constrained_control_s2unconstrained}. If $\vx\left(t_1\right)$ satisfies \eqref{eq:connection_unconstrained_constrained_s1tos2} and \eqref{eq:constrainedarc_highorder_constraint} for $p\in\mathcal{P}_1$, then $\exists \varepsilon>0$, $\systembehavior_1$ occurs for $\left(t_1-\varepsilon,t_1\right)$ and $\systembehavior_2$ occurs for $\left(t_1,t_1+\varepsilon\right)$.
            \item\label{subprop:connection_unconstrained_constrained} If $\systembehavior_1$ is an unconstrained arc and $\systembehavior_2$ is a constrained arc, then $\forall p,q\in\mathcal{P}_2$, $\sgn\left(\vc_p^\top\vA^{r_p-1}\vb\right)=\sgn\left(\vc_q^\top\vA^{r_q-1}\vb\right)$. On $\systembehavior_1$, $\forall p\in\mathcal{P}_2$, it holds that
            \begin{equation}\label{eq:connection_constrained_unconstrained_control_s1unconstrained}
                u\equiv u_1=\left(-1\right)^{r_p-1}\sgn\left(\vc_p^\top\vA^{r_p-1}\vb\right)\um.
            \end{equation}
            Furthermore, $\forall p\in\mathcal{P}_2$, $\exists \hat{r}_p\in\left[r_p,n\right]\cap\N$, s.t.
            \begin{equation}\label{eq:connection_constrained_unconstrained_s1tos2}
                \begin{dcases}
                    \left(-1\right)^{\hat{r}_p}\vc_p^\top\vA^{\hat{r}_p-1}\left(\vA\vx\left(t_1\right)+u_1\vb\right)<0,\\
                    \forall r\in\left[\hat{r}_p-1\right],\,\vc_p^\top\vA^{r-1}\left(\vA\vx\left(t_1\right)+u_1\vb\right)=0.
                \end{dcases}
            \end{equation}
            Conversely, assume that $\dot\vx=\vA\vx+u_1\vb$ for $\left(t_0,t_1\right)$ and $\dot\vx=\widehat{\vA}_1\vx$ for $\left(t_1,t_2\right)$ where $u_1$ is given in \eqref{eq:connection_constrained_unconstrained_control_s1unconstrained}. If $\vx\left(t_1\right)$ satisfies \eqref{eq:connection_constrained_unconstrained_s1tos2} and \eqref{eq:constrainedarc_highorder_constraint} for $p\in\mathcal{P}_2$, then $\exists \varepsilon>0$, $\systembehavior_1$ occurs for $\left(t_1-\varepsilon,t_1\right)$ and $\systembehavior_2$ occurs for $\left(t_1,t_1+\varepsilon\right)$.
        \end{enumerate}
    \end{proposition}

    \begin{remark}
        Proposition \ref{prop:connection} provides a necessary and sufficient condition for the feasibility of adjacent arcs near the connection. Note that the equality constraints at the connection time, i.e., \eqref{eq:constrainedarc_highorder_constraint}, \eqref{eq:connection_constrained_s2tos1}, \eqref{eq:connection_unconstrained_constrained_s1tos2}, and \eqref{eq:connection_constrained_unconstrained_s1tos2}, are allowed to be linearly dependent. When examining feasibility near the connection time, it is sufficient to consider inequality constraints and all linearly independent equality constraints.
    \end{remark}

\section{Keypoints for Trajectory Feasibility}\label{sec:KeypointsFeasibilityOptimalTrajectory}
The switching law proposed in Section \ref{sec:ArcAnalysis} can describe a given optimal trajectory and the feasibility near connections of arcs, while the feasibility during an arc remains inadequately investigated. To address this limitation, this section investigates the feasibility of the optimal trajectory on each arc, especially of \textit{keypoints} in an arc. As formally defined in Definition \ref{def:augmentedswitchinglaw}, a keypoint, where the trajectory touches state constraints' boundaries in isolation, directly influences the feasibility of the perturbed trajectory. For the example shown in Fig. \ref{fig:ASL_demo}, the time points at $\mathcal{T}_1^{(1)}$, $\mathcal{T}_1^{(2)}$, $\mathcal{T}_2^{(2)}$, and ends of arcs are keypoints.

Section \ref{subsec:AdditionalConstraintsEndArc} analyzes additional constraints of an arc's end besides Proposition \ref{prop:connection}. The tangent marker is proposed in Section \ref{subsec:TangentMarkersArc} to describe the feasibility of a BBS trajectory tangent to constraints' boundaries. Finally, Section \ref{subsec:AugmentedSwitchingLaw} proposes the augmented switching law and provides a sufficient and necessary condition for the feasibility near the keypoints.

\subsection{Additional Constraints at the End of An Arc}\label{subsec:AdditionalConstraintsEndArc}
    Proposition \ref{prop:connection} fully discusses the feasibility of the connection of two adjacent arcs w.r.t. constraints induced by $\mathcal{P}_1$ and $\mathcal{P}_2$. However, the constraints on $u$ and $\vx$, except for those in $\mathcal{P}_1\cup\mathcal{P}_2$ still lacks discussion. The following proposition provides a sufficient and necessary condition for the feasibility at the end of an arc.

    \begin{proposition}\label{prop:constrainedarc_end_feasibility}
        Assume that $\systembehavior=\left(\widehat{\vA},\widehat{\vb},\vF,\vg,\mathcal{P}\right)$ is a system behavior for $t\in\left[t_0,t_1\right]$.
        \begin{enumerate}
            \item\label{prop:constrainedarc_end_feasibility_state} $\forall p\not\in\mathcal{P}$, $\exists\varepsilon>0$, s.t. $\vc_p^\top\vx\left(t\right)+d_p\leq0$ on $\left[t_0,t_0+\varepsilon\right]$ if and only if one of the following conditions holds:
            \begin{enumerate}
                \item $\vc_p^\top\vx\left(t_0\right)+d_p<0$.
                \item $\vc_p^\top\vx\left(t_0\right)+d_p=0$. Furthermore, $\exists\hat{r}_p\in\left[n\right]$, s.t. $\forall r\in\left[\hat{r}_p-1\right]$, $\vc_p^\top\widehat{\vA}^{r-1}\left(\widehat{\vA}\vx\left(t_0\right)+\widehat{\vb}\right)=0$, and $\vc_p^\top\widehat{\vA}^{\hat{r}_p-1}\left(\widehat{\vA}\vx\left(t_0\right)+\widehat{\vb}\right)<0$.
            \end{enumerate}
            Similar conclusions hold for $\left[t_1-\varepsilon,t_1\right]$.
            \item\label{prop:constrainedarc_end_feasibility_control} Assume that $\systembehavior$ is a constrained arc. Then, $\exists\varepsilon>0$, s.t. $u\left(t\right)\leq\um$ on $\left[t_0,t_0+\varepsilon\right]$ if and only if one of the following conditions holds, where $p\in\mathcal{P}$:
            \begin{enumerate}
                \item\label{prop:constrainedarc_end_feasibility_control_small} $-\frac{\vc_p^\top\vA^{r_p}}{\vc_p^\top\vA^{r_p-1}\vb}\vx\left(t_0\right)<\um$.
                \item\label{prop:constrainedarc_end_feasibility_control_small_highorder} $-\frac{\vc_p^\top\vA^{r_p}}{\vc_p^\top\vA^{r_p-1}\vb}\vx\left(t_0\right)=\um$. Furthermore, $\exists\hat{r}_p \in\left[n\right]$, s.t. $\forall r\in\left[\hat{r}_p-1\right]$, $-\frac{\vc_p^\top\vA^{r_p}\widehat{\vA}^r}{\vc_p^\top\vA^{r_p-1}\vb}\vx\left(t_0\right)=0$, and $-\frac{\vc_p^\top\vA^{r_p}\widehat{\vA}^{\hat{r}_p}}{\vc_p^\top\vA^{r_p-1}\vb}\vx\left(t_0\right)<0$.
                \item\label{prop:constrainedarc_end_feasibility_control_equal} $-\frac{\vc_p^\top\vA^{r_p}}{\vc_p^\top\vA^{r_p-1}\vb}\vx\left(t_0\right)=\um$. Furthermore, $\forall r\in\left[n\right]$, $-\frac{\vc_p^\top\vA^{r_p}\widehat{\vA}^r}{\vc_p^\top\vA^{r_p-1}\vb}\vx\left(t_0\right)=0$.
            \end{enumerate}
            Similar conclusions hold for $u\geq-\um$ and $\left[t_1-\varepsilon,t_1\right]$.
        \end{enumerate}
    \end{proposition}

    Consider two connected arcs $\systembehavior_1$ and $\systembehavior_2$. Proposition \ref{prop:connection} guarantees the feasibility of constraints active along $\systembehavior_1$ or $\systembehavior_2$. If some constraints are active at the ends of arcs but are inactive in $\systembehavior_1$ and $\systembehavior_2$ a.e., then additional constraints at the connection, as defined in Definition \ref{def:additionalendconstraint}, should be introduced for feasibility based on Proposition \ref{prop:constrainedarc_end_feasibility}.

    \begin{definition}\label{def:additionalendconstraint}
        Consider two arcs $\systembehavior_i$ with active constraints $\mathcal{P}_i$ for $t\in\left[t_{i-1},t_i\right]$, $i\in\left\{1,2\right\}$. $\mathcal{P}_1$ and $\mathcal{P}_2$ induce the equality constraints $\vF\vx\left(t_{1}\right)+\vg=\vzero$ where $\vF$ has full row rank. By Proposition \ref{prop:constrainedarc_end_feasibility}, state constraints except $\mathcal{P}_1\cup\mathcal{P}_2$ and control constraints induce $\widehat{\vC}\vx\left(t_{1}\right)+\widehat{\vd}<\vzero$ and $\widehat{\vF}\vx\left(t_{1}\right)+\widehat{\vg}=\vzero$, where $\left[\begin{array}{c}
            \vF\\\widehat{\vF}
        \end{array}\right]$ has full row rank. Then, the \textit{additional end-constraint} of $\vx\left(t_{1}\right)$ is defined as $\addendconstraint=\left(\widehat{\vC},\widehat{\vd},\widehat{\vF},\widehat{\vg}\right)$. Furthermore, if $\widehat{\vC}$ and $\widehat{\vF}$ are empty matrices, then $\addendconstraint$ is called empty.
    \end{definition}

    \begin{remark}
        Given a feasible arc, the order $\hat{r}_p$ in Proposition \ref{prop:constrainedarc_end_feasibility} is directly determined. In Definition \ref{def:additionalendconstraint}, the additional end-constraint $\addendconstraint$ does not include the equality constraints induced by $\mathcal{P}_{1}\cup\mathcal{P}_{2}$. To avoid over-determination, the equality constraints in $\addendconstraint$ and $\mathcal{P}_{1}\cup\mathcal{P}_{2}$ are made linearly independent through elimination. The elimination process is guaranteed by the feasibility of the given trajectory.
    \end{remark}

    \begin{figure*}[!t]
        \begin{equation}\label{eq:augmentedswitchinglaw}
            \switchinglaw=\left(\systembehavior_1,\left\{\tangentmarker_k^{(1)}\right\}_{k=1}^{M_1}\right)\addendconstraint_1\left(\systembehavior_1,\left\{\tangentmarker_k^{(2)}\right\}_{k=1}^{M_1}\right)\addendconstraint_2\dots\addendconstraint_{N-1}\left(\systembehavior_N,\left\{\tangentmarker_k^{(N)}\right\}_{k=1}^{M_N}\right)
        \end{equation}
        \vspace{-5mm}
    \end{figure*}

\subsection{Tangent Markers in An Arc}\label{subsec:TangentMarkersArc}
    Proposition \ref{prop:connection} and Proposition \ref{prop:constrainedarc_end_feasibility} analyzed the feasibility of an arc near the end points. An arc might be tangent to some constrained boundary, where small perturbations may affect the feasibility of the arc. The following proposition provides a sufficient and necessary condition for the feasibility of the optimal trajectory tangent to constraints' boundaries.

    \begin{proposition}\label{prop:tangentcondition}
        Consider an arc $\systembehavior=\left(\widehat{\vA},\widehat{\vb},\vF,\vg,\mathcal{P}\right)$ for $\left[t_0,t_2\right]$. $\systembehavior$ is tangent to the boundary of the constraint $\vc^\top\vx+d\leq0$ at $t_1\in\left(t_0,t_2\right)$ if and only if $\exists \hat{r}\in\left[n\right]$ is even, s.t.
        \begin{equation}\label{eq:tangentcondition}
            \begin{dcases}
                \vc^\top\vx\left(t_1\right)+d=0,\\
                \forall r\in\left[\hat{r}-1\right],\,\vc^\top\widehat{\vA}^{r-1}\left(\widehat{\vA}\vx\left(t_1\right)+\widehat{\vb}\right)=0,\\
                \vc^\top\widehat{\vA}^{\hat{r}-1}\left(\widehat{\vA}\vx\left(t_1\right)+\widehat{\vb}\right)<0.
            \end{dcases}
        \end{equation}
    \end{proposition}

    Regardless of whether an arc $\systembehavior=\left(\widehat{\vA},\widehat{\vb},\vF,\vg,\mathcal{P}\right)$ is constrained, $\systembehavior$ can be tangent to the boundary of the state constraint $\vc_p^\top\vx+d_p\leq0$ where $p\not\in\mathcal{P}$. It is noteworthy that $\systembehavior$ can also be tangent to the boundary of the control constraints $u\leq\um$ and $-u\leq\um$. For this case, $\systembehavior$ should be constrained. Arbitrarily consider $p\in\mathcal{P}$. According to Proposition \ref{prop:constrainedarc_control}, $\pm u\leq\um$ is equivalent to $\mp\frac{\vc_p^\top\vA^{r_p}\vx}{\vc_p^\top\vA^{r_p-1}\vb}\leq\um$. In other words, the constraints on $u$ are equivalent to the constraints on $\vx$. Hence, Proposition \ref{prop:tangentcondition} can apply for both state constraints and control constraints. For convenience, $\forall p\in\left[P\right]$, the $p$-th constraint refers to $\vc_p^\top\vx+d_p\leq 0$. The $\left(P+1\right)$-th and $\left(P+2\right)$-th one refer to $u\leq\um$ and $-u\leq\um$, respectively.

    Another noteworthy fact is that an arc $\systembehavior$ can be tangent to multiple constraints $\forall p\in\mathcal{P}'\subset\left[P+2\right]$ at a time point $t_1$. Then, the necessary and sufficient condition for feasibility is $\forall p\in\mathcal{P}'$, \eqref{eq:tangentcondition} holds. The tangent marker is defined as follows:

    \begin{definition}\label{def:tangentmarker}
        Consider an arc $\systembehavior=\left(\widehat{\vA},\widehat{\vb},\vF,\vg,\mathcal{P}\right)$ for $t\in\left[t_0,t_2\right]$. Assume that at $t_1\in\left(t_0,t_2\right)$, the arc $\systembehavior$ is tangent to some constraints' boundaries whose indices constitute a nonempty set $\mathcal{P}'\subset\left[P+2\right]$. By Proposition \ref{prop:tangentcondition}, the tangent condition \eqref{eq:tangentcondition} is equivalent to $\widehat{\vC}\vx\left(t_1\right)+\widehat{\vd}<\vzero$ and $\widehat{\vF}\vx\left(t_1\right)+\widehat{\vg}=\vzero$, where $\left[\begin{array}{c}
            \vF\\\widehat{\vF}
        \end{array}\right]$ has full row rank. Then, the \textit{tangent marker} is defined as $\tangentmarker=\left(\widehat{\vC},\widehat{\vd},\widehat{\vF},\widehat{\vg},\mathcal{P}'\right)$.
    \end{definition}

    \begin{remark}
        In Definition \ref{def:tangentmarker}, the equality constraints induced by \eqref{eq:tangentcondition} can be linearly dependent. However, for a given feasible trajectory, the equalities should have a feasible solution; hence, equality constraints with full row rank, i.e., $\widehat{\vF}\vx\left(t_1\right)+\widehat{\vg}=\vzero$, can be constructed through elimination. Therefore, the tangent marker is well-defined.
    \end{remark}

\subsection{Augmented Switching Law}\label{subsec:AugmentedSwitchingLaw}
    Sections \ref{sec:ArcAnalysis}, \ref{subsec:AdditionalConstraintsEndArc}, and \ref{subsec:TangentMarkersArc} analyze the feasibility of the optimal trajectory at the end points of arcs and the tangent points to the constrained boundaries. To characterize the optimal trajectory, this section proposes the augmented switching law, which represents the input control and the feasibility of the trajectory in a compact form.

    \begin{definition}\label{def:augmentedswitchinglaw}
        Consider a feasible BBS trajectory of problem \eqref{eq:optimalproblem} with $N$ arcs, denoted as $\systembehavior_i$, $i\in\left[N\right]$. During $\systembehavior_i$, $M_i$ tangent markers occur, denoted as $\tangentmarker_k^{(i)}$, $k\in\left[M_i\right]$. At the connection of $\systembehavior_i$ and $\systembehavior_{i+1}$, the additional end-constraint is denoted as $\addendconstraint_i$. Then, time points at tangent markers and ends of arcs are called \textit{keypoints} for trajectory feasibility. The \textit{augmented switching law} (ASL) $\switchinglaw$ is defined in \eqref{eq:augmentedswitchinglaw}.
    \end{definition}

    \begin{remark}
        If a feasible trajectory with a finite number of arcs follows the BBS control law \eqref{eq:bang_singular_bang_law}, then Propositions \ref{prop:connection}, \ref{prop:constrainedarc_end_feasibility}, and \ref{prop:tangentcondition} hold for the trajectory. For this case, the ASL can also represent the feasible trajectory.
    \end{remark}

    \begin{remark}
        If a feasible non-chattering BBS trajectory is perturbed by a sufficiently small perturbation, then the perturbed trajectory can only exceed constraints' boundaries at the keypoints since the distance between the original trajectory and constraints' boundaries is 0 at the keypoints, while it is strictly positive at other points. Therefore, the ASL can not only characterize the input control but also describe the feasibility of the trajectory.
    \end{remark}

    \begin{remark}
        The switching law in Definition \ref{def:switchinglaw} corresponding to the ASL in Definition \ref{def:augmentedswitchinglaw} is $\systembehavior_1\systembehavior_2\dots\systembehavior_N$. It can be observed that the ASL provides more detailed information on the feasibility of the trajectory at keypoints than the switching law. Notably, neither the switching law nor the ASL provides the motion time of each arc.
    \end{remark}

    To simplify the notation, denote an arc without tangent markers $\left(\systembehavior,\varnothing\right)$ as $\systembehavior$ in an ASL. If $\addendconstraint$ is empty, then $\addendconstraint$ is omitted. The additional end-constraints before $\systembehavior_1$ and after $\systembehavior_{N}$ are omitted since the initial and terminal states are given, i.e., $\vx_0$ and $\vx_\f$. For the example shown in Fig. \ref{fig:ASL_demo}, the ASL is $\left(\systembehavior_1,\left\{\tangentmarker_1^{(1)}\right\}\right)\addendconstraint_1\left(\systembehavior_2,\left\{\tangentmarker_2^{(1)},\tangentmarker_2^{(2)}\right\}\right)\systembehavior_3$.

    For convenience of understanding, the chain-of-integrator system with full box state constraints in Example \ref{example:chainofintegrators_boxconstraint} is considered as an example that was discussed in our previous works \cite{wang2025time,wang2025chattering}. The system behaviors where $u\equiv\um$ and $u\equiv-\um$, i.e., unconstrained arcs, are denoted as $\overline{0}$ and $\underline{0}$, respectively. The system behaviors where $x_k\equiv x_{\m k}$ and $x_k\equiv-x_{\m k}$, i.e., constrained arcs, are denoted as $\overline{k}$ and $\underline{k}$, respectively. Evidently, during $\overline{k}$ or $\underline{k}$, $u\equiv0$ and $\forall j\in\left[k-1\right]$, $x_j\equiv0$. The tangent marker where $\vx$ is tangent to $x_{\m k}$ and $-x_{\m k}$ of even order $2l<k$ in Proposition \ref{prop:tangentcondition} is denoted as $\left(\overline{k},2l\right)$ and $\left(\underline{k},2l\right)$, respectively. The additional end-constraint $\addendconstraint=\left\{\left(s_k,r_k\right)\right\}_{k=1}^K$ means that the additional state constraints $s_k$ is active with order $r_k$ in Proposition \ref{prop:constrainedarc_end_feasibility}, where $s_k=\overline{j}$ and $s_k=\underline{j}$ refer to the constraints $x_j\leq x_{\m j}$ and $-x_j\leq x_{\m j}$, respectively. Note that the constraints on control are inactive at the ends of constrained arcs since $u\equiv0$ in a constrained arc. Based on the above simplified notations, an example is provided to illustrate the ASL.

    \begin{example}\label{example:MIM_4order}
        Consider a position-to-position problem for 4th-order chain-of-integrator systems with full box constraints shown in Fig. \ref{fig:example_augmentedswitchinglaw}. A feasible suboptimal trajectory with a BBS form, planned by the manifold-intercepted method (MIM) \cite{wang2025time}, is shown in Fig. \ref{fig:example_augmentedswitchinglaw}(a). It can be observed that the ASL is $\switchinglaw=\overline{01}\underline{0}\overline{2}\underline{01}\overline{03}\underline{01}\overline{0}\underline{2}\overline{01}\underline{0}$.
    \end{example}

    \begin{figure}[!t]
        \centering
        \includegraphics[width=\columnwidth]{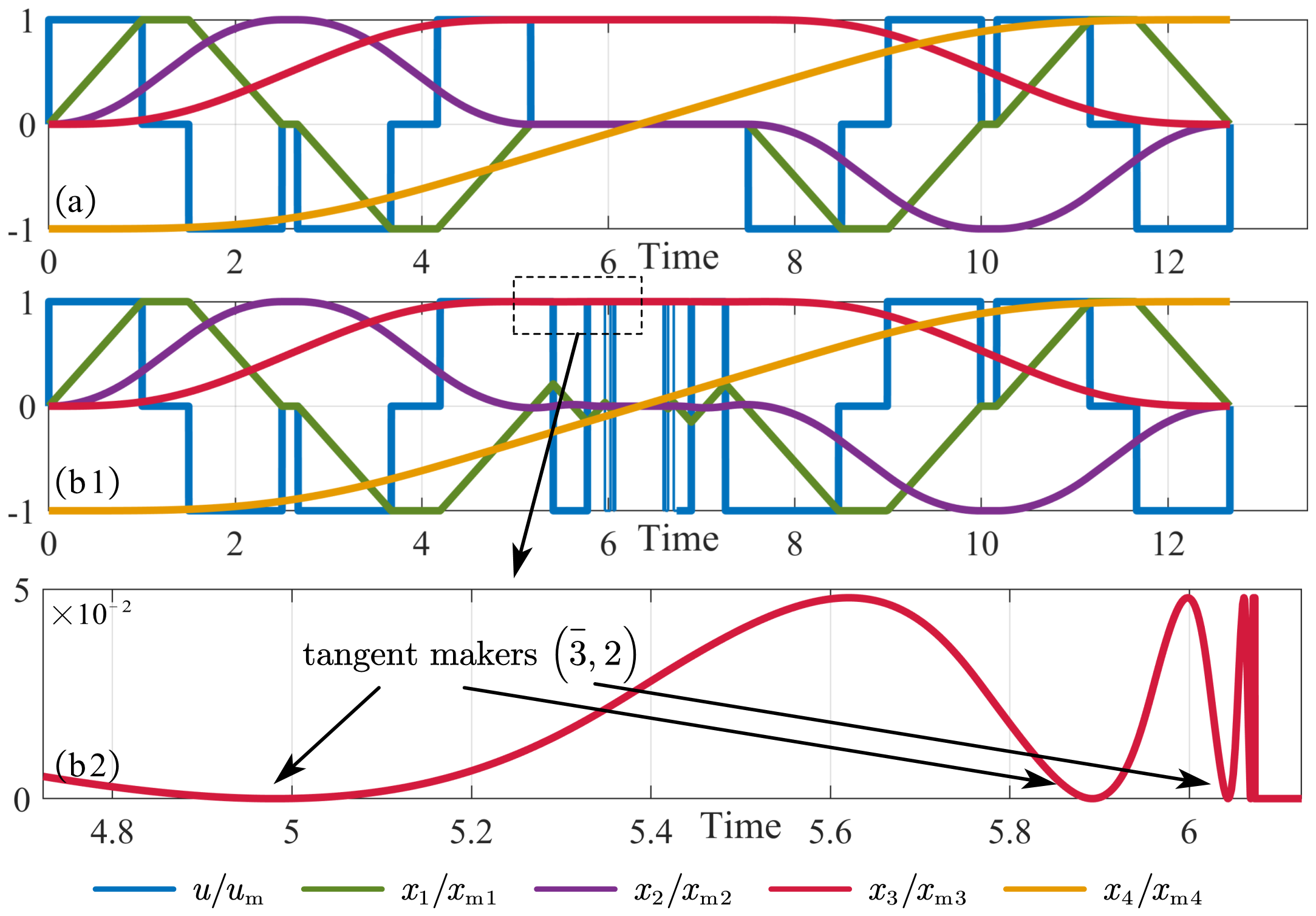}
        \caption{Two examples of ASLs. For a 4th-order chain-of-integrator system with full box state constraints, let $\um=1$, $x_{\m1}=1$, $x_{\m2}=1.5$, $x_{\m3}=4$, $x_{\m4}=15$, $\vx_0=-x_{\m4}\ve_4$, and $\vxf=x_{\m4}\ve_4$. (a) A non-chattering feasible trajectory planned by \cite{wang2025time}. (b) A chattering optimal trajectory planned by \cite{wang2025chattering}. (b2) is the plot of $\hat{x}_3\left(t\right)=\left(x_{\m3}-x_3\left(t\right)\right)\left(t_\infty-t\right)^{-3}$ where $t_\infty\approx6.0732$ is a chattering limit time.}
        \label{fig:example_augmentedswitchinglaw}
    \end{figure}

    \begin{example}\label{example:chattering_4order}
        Our previous work \cite{wang2025chattering} pointed out that the chattering phenomenon occurs in the optimal trajectory for the problem in Example \ref{example:MIM_4order}. The chattering optimal trajectory is shown in Fig. \ref{fig:example_augmentedswitchinglaw}(b1). Fig. \ref{fig:example_augmentedswitchinglaw}(b2) plots the trajectory of $x_3$ during the chattering period with amplitude compensation applied. It can be observed that the trajectory during the left chattering period consists of a cycle of $\left(\overline{0},\left\{\left(\overline{3},2\right)\right\}\right)\underline{0}$. According to Bellman's principle of optimality \cite{bellman1952theory}, the trajectory in $\left[0,t_\infty-\varepsilon\right]$ is also time-optimal, where $\varepsilon>0$ is sufficiently small. The augmented switching of the above sub-arc begins with $\overline{01}\underline{0}\overline{2}\underline{01}\left(\overline{0},\left\{\left(\overline{3},2\right)\right\}\right)\underline{0}\left(\overline{0},\left\{\left(\overline{3},2\right)\right\}\right)\underline{0}\dots$
    \end{example}

\section{A State-Centric Necessary Condition}\label{sec:StateCentricNecessaryCondition}
The theoretical framework for the ASL of time-optimal control for linear systems is established in Section \ref{sec:KeypointsFeasibilityOptimalTrajectory}. Based on the established framework, this section provides a state-centric necessary condition for the optimal control of problem \eqref{eq:optimalproblem}. Section \ref{subsec:UniquenessTimeOptimalControl} discusses the uniqueness of the time-optimal control. Then, given a feasible BBS trajectory, Section \ref{subsec:FeasibilityDisturbedTrajectory} analyzes the feasibility of the perturbed BBS trajectory. Based on the conclusions in Sections \ref{subsec:UniquenessTimeOptimalControl} and \ref{subsec:FeasibilityDisturbedTrajectory}, Section \ref{subsec:NecessaryConditionOptimalTrajectory} provides a necessary condition of optimality.

\subsection{Uniqueness of Time-Optimal Control}\label{subsec:UniquenessTimeOptimalControl}
    The uniqueness of the time-optimal control for problem \eqref{eq:optimalproblem} is proved in Theorem \ref{thm:uniqueness_nochattering} as preliminaries.

    \begin{theorem}\label{thm:uniqueness_nochattering}
        Under Assumption \ref{assumption:feasible_optimal}, the time-optimal control of problem \eqref{eq:optimalproblem} is unique in an a.e. sense. In other words, if $u=u_1^*\left(t\right)$ and $u=u_2^*\left(t\right)$, $t\in\left[0,t_\f^*\right]$, are both optimal, then $u_1^*\left(t\right)=u_2^*\left(t\right)$ a.e.; hence, $\forall t\in\left[0,t_\f^*\right]$, $\vx_1^*\left(t\right)=\vx_2^*\left(t\right)$.
    \end{theorem}

    \begin{remark}
        Based on Theorem \ref{thm:uniqueness_nochattering}, it can be proved that the optimal control of problem \eqref{eq:optimalproblem} is unique in an a.e. sense when there exists at most one chattering limit point. The proof is omitted since it does not contribute to the proposed state-centric necessary condition in Section \ref{subsec:NecessaryConditionOptimalTrajectory}.
    \end{remark}

    According to Theorem \ref{thm:uniqueness_nochattering}, if the time-optimal trajectory is perturbed by a sufficiently and finitely small perturbation, then the resulting perturbed trajectory achieves a strictly longer terminal time than the optimal terminal time. The above observation will be applied in Section \ref{subsec:FeasibilityDisturbedTrajectory}.

\subsection{Feasibility of the Perturbed BBS Trajectory}\label{subsec:FeasibilityDisturbedTrajectory}
    Consider a feasible BBS trajectory $\vx=\vx\left(t\right)$ of problem \eqref{eq:optimalproblem} without chattering. The ASL is denoted as $\switchinglaw$. For convenience of discussion, the keypoints of $\switchinglaw$ occur at $t_i$, where $0=t_0<t_1<\dots<t_M=t_\f$. In other words, $\left\{t_i\right\}_{i=0}^M$ consists of time points at ends of arcs and tangent markers. Denote $\vt=\left(t_i\right)_{i=1}^M$, and $\dot\vx=\vA_i\vx+\vb_{i}$ in $\left(t_{i-1},t_i\right)$. This section discusses the feasibility of a perturbed trajectory $\vx=\vx'\left(t\right)$ represented by the ASL $\switchinglaw'$, where the time vector $\vt'=\left(t_i'\right)_{i=1}^M$ is close to $\vt$, i.e., $\norm[]{\vt'-\vt}$ is sufficiently small. For convenience, the notation $\norm[]{\bullet}$ refers to $\norm[\infty]{\bullet}$ in this section if not specified.

    The feasibility of the perturbed trajectory is proved by two steps. Firstly, in Proposition \ref{prop:disturbedtrajectory_error}, the state error between the perturbed trajectory and the original trajectory is bounded by $C\norm[]{\vt-\vt'}$ where $C$ is a constant dependent on the original trajectory. In other words, if the perturbation of keypoints' time is sufficiently small and the augmented switching law is fixed, then the states of the perturbed trajectory are close to those of the original trajectory. Secondly, Theorem \ref{thm:disturbedtrajectory_feasibility} guarantees the feasibility of the perturbed trajectory based on Propositions \ref{prop:connection}, \ref{prop:constrainedarc_end_feasibility}, \ref{prop:tangentcondition}, and \ref{prop:disturbedtrajectory_error}. In other words, if the original trajectory is feasible, the perturbation of keypoints' time is sufficiently small, and the augmented switching law is fixed, then the perturbed trajectory is also feasible. Specifically, Proposition \ref{prop:disturbedtrajectory_error} helps to provide an upper bound on the perturbation under the requirement of the perturbed trajectory's feasibility in Theorem \ref{thm:disturbedtrajectory_feasibility}.

    \begin{proposition}\label{prop:disturbedtrajectory_error}
        Consider a BBS trajectory $\vx\left(t\right)$ satisfying $\forall i\in\left[M\right]$, $\dot\vx=\vA_i\vx+\vb_i$ in $\left(t_{i-1},t_i\right)$. Denote $\vt=\left(t_i\right)_{i=1}^M$. $\exists C>0$ only dependent on $\vx=\vx\left(t\right)$, s.t. $\forall\varepsilon\in\left(0,\frac12\min_{i\in\left[M-1\right]}\left\{\abs{t_{i+1}-t_i}\right\}\right)$, $\forall \vx'\left(t\right)$ is a perturbed trajectory induced by $\vt'=\left(t_i'\right)_{i=1}^M$ with $\norm[]{\vt-\vt'}<\varepsilon$,
        \begin{equation}\label{eq:disturbedtrajectory_error_sametime}
            \norm[\infty]{\vx-\vx'}\triangleq\sup_{t\in\left[t_0,\min\left\{t_M,t_M'\right\}\right]}\norm[]{\vx\left(t\right)-\vx'\left(t\right)}\leq C\varepsilon
        \end{equation}
        where $\forall i\in\left[M\right]$, $\dot\vx'=\vA_i\vx'+\vb_i$ in $\left(t_{i-1}',t_i'\right)$ and $t_0'=t_0$.
    \end{proposition}

    \begin{theorem}\label{thm:disturbedtrajectory_feasibility}
        Consider a feasible non-chattering BBS trajectory $\vx=\vx\left(t\right)$ whose ASL is $\switchinglaw$ in problem \eqref{eq:optimalproblem}. Then, $\exists\varepsilon>0$ dependent on $\vx=\vx\left(t\right)$, s.t. $\forall \vx=\vx'\left(t\right)$ is a trajectory also represented by $\switchinglaw'=\switchinglaw$, if $\norm[]{\vt-\vt'}<\varepsilon$, then $\vx=\vx'\left(t\right)$ is feasible for problem \eqref{eq:optimalproblem}. $\vx\left(t\right)$ and $\vx'\left(t\right)$ are induced by $\vt$ and $\vt'$ in the same way as that in Proposition \ref{prop:disturbedtrajectory_error}.
    \end{theorem}

    Theorem \ref{thm:disturbedtrajectory_feasibility} points out that if the perturbed trajectory is represented by the same ASL to the original trajectory and the perturbance is small enough, then the perturbed trajectory is feasible. In other words, Theorem \ref{thm:disturbedtrajectory_feasibility} provides a novel variational approach under the fixed ASL with a feasibility guarantee near constraints' boundaries.

\subsection{State-Centric Necessary Condition of Optimality}\label{subsec:NecessaryConditionOptimalTrajectory}

    Based on Theorem \ref{thm:disturbedtrajectory_feasibility}, this section provides an approach to constructing a feasible perturbed trajectory with the same ASL. According to Theorem \ref{thm:uniqueness_nochattering}, this section shows that the existence of such a perturbed trajectory implies that the original trajectory is not optimal, resulting in a state-centric necessary for non-chattering optimal control of problem \eqref{eq:optimalproblem}.

    The Jacobian matrix of equality constraints induced by the ASL can be calculated through Proposition \ref{prop:Jacobian}.

    \begin{proposition}\label{prop:Jacobian}
        $\vx=\vx\left(t\right)$ is driven by $\forall i\in\left[M\right]$, $t\in\left(t_{i-1},t_i\right)$, $\dot\vx=\vA_i\vx+\vb_i$, and $\vx\left(t_0\right)=\vx_0$ is fixed. Consider the mapping $\vx_i:\vt=\left(t_i\right)_{i=1}^M\mapsto\vx\left(t_i\right)$. Then, $\forall i,j\in\left[M\right]$,
        \begin{equation}\label{eq:Jacobian}
            \frac{\partial\vx_i}{\partial t_j}=\begin{dcases}
                \vP_{j,i}\left(\left(\vA_{j}-\vA_{j+1}\right)\vx_j+\vb_{j}-\vb_{j+1}\right),\,j<i,\\
                \vA_i\vx_i+\vb_i,\,j=i,\\
                \vzero,\,j>i,
            \end{dcases}
        \end{equation}
        where $\vP_{j,i}=\e^{\vA_i\left(t_i-t_{i-1}\right)}\e^{\vA_{i-1}\left(t_{i-1}-t_{i-2}\right)}\dots\e^{\vA_{j+1}\left(t_{j+1}-t_j\right)}$.

    \end{proposition}

    Utilizing the notations in Section \ref{subsec:FeasibilityDisturbedTrajectory}, the ASL provides a system of equalities on every keypoints. Denote that $\forall i\in\left[M\right]$, $\vF_i\vx_i+\vg_i=\vzero$ and $\vC_i\vx_i+\vd_i<\vzero$, where $\vx_i=\vx\left(t_i\right)$ and $\vF_i$ has full row rank. For a constrained arc with active constraints $\mathcal{P}$, the equality constraints induced by $\mathcal{P}$ are added in only one end instead of both ends of the arcs since one end satisfies $\mathcal{P}$ implies that the whole arc satisfies $\mathcal{P}$. Note that the additional end-constraints at both ends are eliminated by equality constraints induced by $\mathcal{P}$. Similar processes are applied if $u\equiv\um$ or $u\equiv-\um$ during a constrained arc. $\vF_i\vx_i\left(\vt\right)+\vg_i$ for $i\in\left[M\right]$ induce a function $\vH\left(\vt\right)$. Then, Theorem \ref{thm:Jacobian_necessary_condition} provides a necessary condition of optimality.

    \begin{theorem}[State-Centric Necessary Condition]\label{thm:Jacobian_necessary_condition}
        Assume that the optimal trajectory $\vx=\vx^*\left(t\right)$ of problem \eqref{eq:optimalproblem} satisfies Assumption \ref{assumption:feasible_optimal}. Denote the equality constraints induced by the ASL as $\vH\left(\vt^*\right)=\vzero$, where $\vt^*=\left(t_i^*\right)_{i=1}^M$ is the arriving time of each keypoints and $t_0=0$. Then, $\frac{\partial\vH}{\partial\vt_{1:\left(M-1\right)}}\left(\vt^*\right)$ does not have full row rank.
    \end{theorem}

    \begin{remark}
        In contrast to existing costate-based necessary conditions \cite{jacobson1971new}, Theorem \ref{thm:Jacobian_necessary_condition} requires no information on costates, which significantly simplifies the computation. Compared to existing state-centric necessary conditions \cite{tedone2020hamilton}, Theorem \ref{thm:Jacobian_necessary_condition} is effective in theoretical reasoning and has a lower computational complexity in numerical computation. Some examples are provided in Sections \ref{subsec:chainofintegrators_onlysystembehaviors} and \ref{subsec:chainofintegrators_chattering}.
    \end{remark}

    \begin{remark}
        Theorem \ref{thm:Jacobian_necessary_condition} provides a first-order necessary condition of optimality of problem \eqref{eq:optimalproblem}. In fact, if $\frac{\partial\vH}{\partial\vt_{1:\left(M-1\right)}}\left(\vt^*\right)$ does not have full row rank but $\frac{\partial\vH}{\partial\vt}\left(\vt^*\right)$ has full row rank with $M'<M$, then the solution $\vt^*$ can be seen as a stationary point under the constraint of $\mathcal{L}$. For this case, some high-order necessary conditions can be derived. Consider the implicit function $\vf:\left(t_i\right)_{i\in\Setminus{\left[M\right]}{\mathcal{I}}}\mapsto\left(t_i\right)_{i\in\mathcal{I}}$ induced by $\vH\left(\vt\right)=\vzero$ where $M\in\mathcal{I}$. Then, the function $\left(t_i\right)_{i\in\Setminus{\left[M\right]}{\mathcal{I}}}\mapsto t_M$ should achieve a strictly local minimum at $\left(t_i^*\right)_{i\in\Setminus{\left[M\right]}{\mathcal{I}}}$ according to Theorem \ref{thm:uniqueness_nochattering}.
    \end{remark}

    \begin{remark}
        If a given non-chattering BBS trajectory does not satisfy the necessary condition in Theorem \ref{thm:Jacobian_necessary_condition}, then one can perform numerical optimization of the trajectory based on Theorem \ref{thm:Jacobian_necessary_condition}. Numerical examples are provided in Section \ref{sec:simulation}.

    \end{remark}

\section{Applications in Chain-of-Integrator Systems}\label{sec:Applications}
This section provides some applications of the state-centric necessary condition, i.e., Theorem \ref{thm:Jacobian_necessary_condition}, for the optimal control of problem \eqref{eq:optimalproblem}. For convenience, the chain-of-integrator system with box constraints is considered in this section, where the notations have been introduced in Section \ref{subsec:AugmentedSwitchingLaw}.

Time-optimal control for chain-of-integrator systems with box constraints is an open and challenging problem in the optimal control domain, yet to be resolved. The problem can be summarized as follows:
\begin{IEEEeqnarray}{rl}\label{eq:optimalproblem_chainofintegrators}
    \min\quad& \tf=\int_{0}^{\tf}\mathrm{d}t,\IEEEyesnumber\IEEEyessubnumber*\\
    \st\quad&\dot{x}_k\left(t\right)=x_{k-1}\left(t\right),\,1<k\leq n,\,\forall t\in\left[0,t_\f\right],\\
    &\dot{x}_1\left(t\right)=u\left(t\right),\,\forall t\in\left[0,t_\f\right],\\
    &\vx\left(0\right)=\vx_0,\,\vx\left(\tf\right)=\vx_\f,\\
    &\abs{x_k\left(t\right)}\leq x_{\m k},\,1\leq k\leq n,\,\forall t\in\left[0,t_\f\right],\\
    &\abs{u\left(t\right)}\leq\um,\,\forall t\in\left[0,t_\f\right].
\end{IEEEeqnarray}

This section provides some applications of Theorem \ref{thm:Jacobian_necessary_condition}. Section \ref{subsec:chainofintegrators_purecontrolconstriant} proves a trivial conclusion from a state-centric perspective. Section \ref{subsec:chainofintegrators_onlysystembehaviors} and Section \ref{subsec:chainofintegrators_chattering} prove two corollaries that are challenging to prove by traditional costate-based necessary conditions.

\subsection{The Case where $\forall k\in\left[n\right]$, $x_{\m k}=\infty$}\label{subsec:chainofintegrators_purecontrolconstriant}
    The case where $\forall k\in\left[n\right]$, $x_{\m k}=\infty$ has a well-known conclusion that the optimal control switches no more than $\left(n-1\right)$ times \cite{lee1967foundations}. This section proves the above conclusion based on the proposed Theorem \ref{thm:Jacobian_necessary_condition}.

    \begin{corollary}\label{cor:chainofintegrators_purecontrolconstriant}
        In problem \eqref{eq:optimalproblem_chainofintegrators}, assume that $\forall k\in\left[n\right]$, $x_{\m k}=\infty$. Then, the optimal control switches no more than $\left(n-1\right)$ times.
    \end{corollary}

    Corollary \ref{cor:chainofintegrators_purecontrolconstriant} implies that the switching law is in the form of $\overline{0}\underline{0}\overline{0}\underline{0}\dots$ or $\underline{0}\overline{0}\underline{0}\overline{0}\dots$ with no more than $n$ arcs. The case in Section \ref{subsec:chainofintegrators_purecontrolconstriant} is trivial since no state constraints exist and the costate vector follows a fixed equation that $\forall k\in\left[n-1\right]$, $\dot\lambda_i=-\lambda_{i+1}$ and $\lambda_n\equiv\const$.

    However, when the state constraints are induced, the behavior of the optimal control can be complex. On the one hand, the multiplier $\veta$ in \eqref{eq:dotlambda} can be non-zero during a constrained arc. On the other hand, the junction of $\vlambda$ can occur when a state constraint switches between active and inactive \cite{hartl1995survey}. For this case, the costate analysis can be complex, while the proposed state-centric necessary condition provides a simple and effective approach for analyzing the optimal control.

\subsection{The Case where Only System Behaviors Occur}\label{subsec:chainofintegrators_onlysystembehaviors}
    The case where only system behaviors occur, i.e., the optimal trajectory consists of a finite number of unconstrained arcs and constrained arcs that are not tangent to constraints' boundaries, is widely applied in existing works on trajectory planning \cite{he2020time,berscheid2021jerk,wang2025time}.

    Denote the ASL of the trajectory as $\switchinglaw=\systembehavior_1\systembehavior_2\dots\systembehavior_N\addendconstraint_N$, where $\addendconstraint_N$ is induced by $\vx\left(t_N\right)=\vxf$. $\forall k\in\N$, denote $\abs{\overline{k}}=\abs{\underline{k}}=k$, $\sgn\left(\overline{k}\right)=+1$, and $\sgn\left(\underline{k}\right)=-1$. The non-existence of tangent markers and additional end-constraints in $\switchinglaw$ means that $\forall i\in\left[N\right]$, the arc $\systembehavior_i$ achieves strictly feasibility. An arc $\systembehavior=\left(\widehat{\vA},\widehat{\vb},\vF,\vg,\mathcal{P}\right)$ in $\left[t_0,t_1\right]$ is strictly feasible if $\forall p\not\in\mathcal{P}$, $\vc_p^\top\vx+d_p<0$ holds in $\left[t_0,t_1\right]$.

    \begin{corollary}\label{cor:chainofintegrators_onlysystembehaviors}
        Assume that the ASL consists of system behaviors without tangent markers and additional end-constraints except at $\tf$. Assume that:
        \begin{enumerate}
            \item\label{condition:3_2_switchinglaw_1ji} $\forall 1\leq j<i$, if $\abs{\systembehavior_j}\geq\abs{\systembehavior_i}$, then $\sum_{k=j+1}^{i}\abs{\systembehavior_k}<i-j$.
            \item\label{condition:3_2_switchinglaw_ijN} $\forall i<j\leq N$, if $\abs{\systembehavior_j}\geq\abs{\systembehavior_i}$, then $\sum_{k=i}^{j-1}\abs{\systembehavior_k}<j-i$.
            \item\label{condition:3_2_switchinglaw_ijN_} $\forall i\in\left[N\right]$, if $\abs{\systembehavior_i}>0$ and $\forall i<j\leq N$, $\abs{\systembehavior_j}<\abs{\systembehavior_i}$, then $\sum_{k=i}^{N}\abs{\systembehavior_k}\leq N-i$.
        \end{enumerate}
        Then, it holds that
        \begin{equation}\label{eq:chainofintegrators_onlysystembehaviors}
            N-\sum_{i=1}^{N}\abs{\systembehavior_i}\leq n.
        \end{equation}
    \end{corollary}

    \begin{remark}
        In our previous work \cite{wang2025time}, a similar conclusion pointed out that when fixing $\vxf$, then the set of $\left(\vx_0,\vt\right)$ that satisfies the ASL $\switchinglaw=\systembehavior_1\systembehavior_2\dots\systembehavior_N$ locally forms a submanifold of dimension $r$. Among them, $r=N-\sum_{i=1}^{N}\abs{\systembehavior_i}$. However, the local optimality is not discussed in \cite{wang2025time}. Corollary \ref{cor:chainofintegrators_onlysystembehaviors} proves that the ASL with $r>n$ fails to achieve optimality.
    \end{remark}

    It is challenging to reason Corollary \ref{cor:chainofintegrators_onlysystembehaviors} by existing costate-based necessary conditions \cite{jacobson1971new} due to complex behaviors of the costates. In contrast to Corollary \ref{cor:chainofintegrators_onlysystembehaviors}, PMP-based approaches need to derive the costates for a given trajectory case-by-case.

    For example, in a 4th-order problem, $\mathcal{L}_1=\underline{01}\overline{01}\underline{0}\overline{2}\underline{01}\overline{01}\underline{0}$ can be feasible, but it cannot be optimal since $N-\sum_{i=1}^{N}\abs{\systembehavior_i}=5>4$. $\mathcal{L}_2=\overline{01}\underline{0}\overline{2}\underline{01}\overline{0}\underline{2}\overline{01}\underline{0}$ can be a candidate for optimal control since $N-\sum_{i=1}^{N}\abs{\systembehavior_i}=4$. Though Corollary \ref{cor:chainofintegrators_onlysystembehaviors} is not a sufficient condition for optimal control, it provides a simple and effective approach to analyzing optimal control of problem \eqref{eq:optimalproblem_chainofintegrators}. Numerical examples are provided in Section \ref{sec:simulation}.

\subsection{The Chattering Phenomenon induced by $\abs{x_2}\leq x_{\m2}$}\label{subsec:chainofintegrators_chattering}
    Our previous work \cite{wang2025chattering} points out that if the chattering phenomenon occurs in problem \eqref{eq:optimalproblem_chainofintegrators}, then there exists one and only one inequality state constraint that switches between active and inactive during the chattering phenomenon. An infinite number of unconstrained arcs are joined by constraints' boundaries, i.e., the unconstrained arcs are tangent to the boundary for infinite times. In contrast, constrained arcs do not occur. This section discusses the chattering phenomenon induced by $x_2\leq x_{\m2}$ based on the proposed state-centric necessary condition. \cite{wang2025chattering} provides the following lemma.
    \begin{lemma}\label{lemma:chattering_s2}\cite{wang2025chattering}
        Assume that the chattering phenomenon is induced by $x_2\leq x_{\m2}$ in problem \eqref{eq:optimalproblem_chainofintegrators} in a left neighborhood of $t_\infty$, where $t_\infty$ is the chattering limit point. Then, $\exists\left\{t_{3k}\right\}_{k=0}^\infty\subset\left[t_0,t_\infty\right]$ increasing strictly monotonically and converging to $t_\infty$, s.t. $x_2$ is tangent to $x_{\m2}$ at $t_{3k}$. $\forall k\in\N^*$, $u$ switches at most 2 times during $\left(t_{3k-3},t_{3k}\right)$.
    \end{lemma}

    \begin{corollary}\label{cor:chainofintegrators_chattering_s2}
        Assume that the chattering phenomenon is induced by $x_2\leq x_{\m2}$ in problem \eqref{eq:optimalproblem} in a left neighborhood of $t_\infty$. Let $\tau_k=t_\infty-t_{3k}$, $\forall k\in\N$. Denote $f_m\left(a,b,c\right)$ as
        \begin{equation}
            \left(b+3a\right)^m-3\left(3b+a\right)^m+3\left(c+3b\right)^m-\left(3c+b\right)^m.
        \end{equation}
        Then, $\forall N\in\N^*$, $\vJ'$ does not have full row rank, where
        \begin{equation}
            \vJ'=\left(f_{i+1}\left(\tau_{j-1},\tau_{j},\tau_{j+1}\right)\right)_{i\in\left[n-2\right],j\in\left[N\right]}.
        \end{equation}
        Furthermore, $\forall k\geq n-2$, $\tau_k$ can be calculated by the following recursive equation:
        \begin{equation}\label{eq:recursive_chattering_s2}
            \det\left(f_{i+1}\left(\tau_{k-j-1},\tau_{k-j},\tau_{k-j+1}\right)\right)_{i\in\left[n-2\right],j\in\left[n-2\right]}=0.
        \end{equation}
    \end{corollary}

    \begin{remark}
        According to Corollary \ref{cor:chainofintegrators_chattering_s2}, the chattering phenomenon can be induced by $x_2\leq x_{\m2}$ only if $\lim_{k\to\infty}\tau_k=0$ and $\tau_k>0$ under the recursive equation \eqref{eq:recursive_chattering_s2}. It is evident that chattering does not occur when $n=3$ since $\tau_k=2\tau_{k-1}-\tau_{k-2}$ converges to $\infty$. Through more refined derivation, it can be rigorously proved that $n=4$ does not hold. In future works, one can try to determine the existence of chattering in problem \eqref{eq:optimalproblem_chainofintegrators} of order $n\geq5$ under constraints on $x_2$, where Corollary \ref{cor:chainofintegrators_chattering_s2} and Theorem \ref{thm:Jacobian_necessary_condition} would serve as useful mathematical tools.

        However, it is challenging to obtain the recursive equation \eqref{eq:recursive_chattering_s2} for problem \eqref{eq:optimalproblem_chainofintegrators} based on the traditional PMP-based necessary condition \cite{jacobson1971new} since the behaviors of costates are significantly complex for high-order problems, let alone determining the existence of chattering. A comparison between state-centric and PMP-based approaches is provided in Appendix \ref{app:comparison}.
    \end{remark}

    Although Theorem \ref{thm:Jacobian_necessary_condition} holds under the assumption of the non-existence of the chattering phenomenon, Corollary \ref{cor:chainofintegrators_chattering_s2} provides an example to apply the proposed necessary condition to a chattering trajectory, where a sub-arc without chattering of the optimal trajectory is investigated.

\section{Numerical Experiments}\label{sec:simulation}

Numerical experiments for high-order chain-of-integrator systems \eqref{eq:optimalproblem_chainofintegrators} are conducted to provide a potential way to apply the proposed state-centric necessary condition to optimizing trajectories. Since this paper focuses on the theoretical necessary condition instead of optimization algorithms, metrics like computational time are not considered. Note that this paper does not provide a manual optimization algorithm.

\subsection{Representing a Feasible BBS Trajectory}\label{subsec:simulation_3_order}

Consider problem \eqref{eq:optimalproblem} whose system matrix is non-diagonalizable with multiple Jordan blocks as follows:
\begin{equation}\label{eq:complex_example}
    \vA=\left[\begin{array}{cccc}
        -1&1&0&0\\
        0&-1&1&0\\
        0&0&-1&0\\
        0&0&0&0
    \end{array}\right],\,\vb=\left[\begin{array}{c}
        0\\0\\2\\-1
    \end{array}\right].
\end{equation}
The constraints are $\abs{u}\leq1$, $x_1\geq-0.7$, and $x_3+x_4\leq0.5$. As defined in Proposition \ref{prop:constrainedarc_control}, $x_1\geq-0.7$ is a 3rd-order state constraint. As shown in Figs. \ref{fig:complex_example}(a-c), a feasible BBS trajectory $\vx_{\text{ASL}}$ consists of full features defined in this paper. The ASL is $\switchinglaw=\systembehavior_1\systembehavior_2\systembehavior_3\systembehavior_4\left(\systembehavior_5,\left\{\tangentmarker_5^{(1)}\right\}\right)\systembehavior_6\addendconstraint_6\systembehavior_7\systembehavior_8$. Specifically, $\systembehavior_1$, $\systembehavior_4$, and $\systembehavior_8$ are unconstrained arcs where $u\equiv-1$. $\systembehavior_2$, $\systembehavior_5$, and $\systembehavior_7$ are unconstrained arcs where $u\equiv+1$. $\systembehavior_3$ and $\systembehavior_6$ are constrained arcs where $x_3+x_4\equiv0.5$. $\tangentmarker_5^{(1)}$ is a tangent marker where $x_1$ is tangent to $-0.7$. $\addendconstraint_6$ is an additional end-constraint requiring $u\left(t_6^-\right)=1$. In other words, the control $u$ is continuous at the connection of the constrained arc $\systembehavior_6$ and the unconstrained arc $\systembehavior_7$. In contrast, $u$ is not continuous at the two ends of another constrained arc $\systembehavior_3$. One can check that Theorem \ref{thm:Jacobian_necessary_condition} holds for $\vx_{\text{ASL}}$.

A numerical optimal control toolbox, i.e., Yop \cite{leek2016optimal}, is applied to solving problem \eqref{eq:complex_example}. The number of intervals is set to 200. The resulting trajectory without a specified initial solution $\vx_{\text{Yop1}}$ is significantly far from optimal, as shown in Fig. \ref{fig:complex_example}. If setting $\vx_{\text{ASL}}$ as the initial solution, then a near-optimal solution $\vx_{\text{Yop2}}$ is obtained. The terminal times of $\vx_{\text{ASL}}$, $\vx_{\text{Yop1}}$, and $\vx_{\text{Yop2}}$ are 5.541, 6.896, and 5.624, respectively.

Furthermore, costates can be reconstructed from states $\vx_{\text{ASL}}$ through meticulous derivations, as shown in Fig. \ref{fig:complex_example}(d). The information on costates also verifies the proposed state-centric necessary condition in a way, noting that PMP \eqref{eq:bang_singular_bang_law} holds in this example. In contrast, the costates of $\vx_{\text{Yop1}}$ and $\vx_{\text{Yop2}}$ do not exist since Yop does not solve a problem based on costates.

\begin{figure}[!t]
    \centering
    \includegraphics[width=\columnwidth]{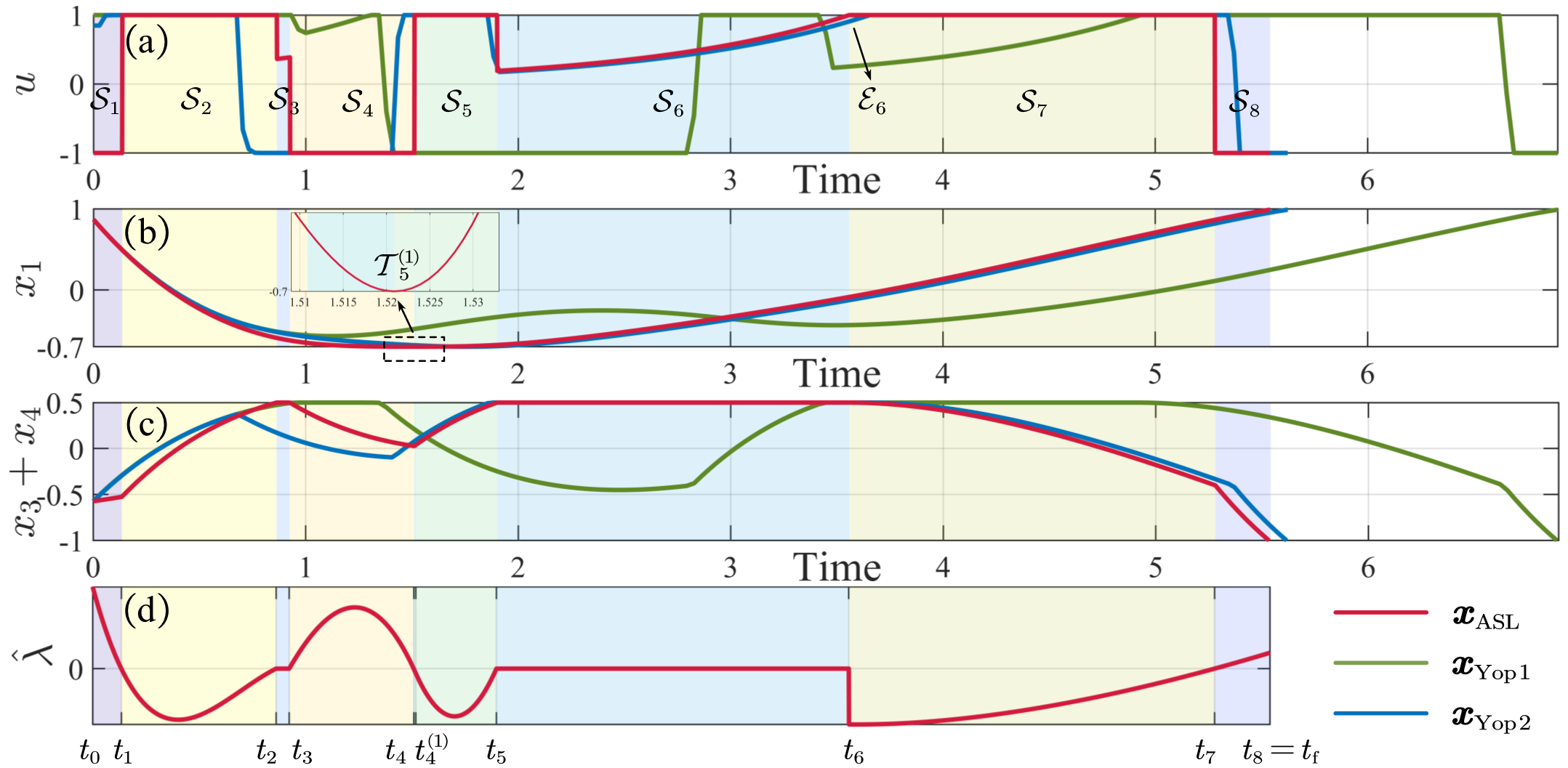}
    \caption{A feasible BBS trajectory with system behaviors, tangent markers, and additional end-constraints. In (d), $\hat{\lambda}\triangleq\left(2\lambda_3-\lambda_4\right)\xi\left(t\right)$ where $\xi\left(t\right)>0$ serves as a scaling factor to enhance the visibility of $\sgn\left(2\lambda_3-\lambda_4\right)$.}
    \label{fig:complex_example}
\end{figure}

\subsection{Determining the Optimality of a 4th-Order Trajectory}\label{subsec:simulation_4_order}

A near-optimal trajectory planning method is proposed in our previous work \cite{wang2025time}. Consider a 4th-order trajectory for high-order chain-of-integrator systems \eqref{eq:optimalproblem_chainofintegrators} provided in \cite{wang2025time} where $\vx_0=\left(0.75,-0.375,2,9\right)$, $\vxf=\left(0.25,0.5,-2,-5\right)$, $\vx_\m=\left(1,1.5,4,20\right)$, and $\um=1$, as shown in Fig. \ref{fig:simulation_disturb}(b). The ASL of the original trajectory $\vx\left(t\right)$ is $\switchinglaw=\underline{01}\overline{0}\underline{2}\overline{01}\underline{0}\overline{01}\underline{0}\overline{0}$. Note that $N-\sum_{i=1}^{N}\abs{\systembehavior_i}=6>4$; hence, Corollary \ref{cor:chainofintegrators_onlysystembehaviors} implies that the terminal time $\tf=9.8604$ is not optimal. Let $\vH$ be the equality constraints induced by $\mathcal{L}$. Note that $\rank\frac{\partial\vH}{\partial\vt}=9$ and $\frac{\partial\vH}{\partial\vt}\in\R^{9\times11}$ except the 4th and 11th columns has full row rank. Then, perturb $t_4$ to $t_4'$ and search for the minimum terminal time $t_\f'$ subject to the constraints of \eqref{eq:optimalproblem_chainofintegrators}. As shown in Fig. \ref{fig:simulation_disturb}(a), $t_\f'$ achieves its local minimum $t_\f-0.0061$ at $t_4'=t_4+0.0208$, resulting in the optimized trajectory $\vx'$. It can be observed in Fig. \ref{fig:simulation_disturb}(b2) that $x_3'$ is tangent to $-x_{\m3}$, while $x_3>-x_{\m3}$ holds in the original trajectory. Hence, the original trajectory fails to fully utilize the feasible set, and the extreme of $\tf'$ activates a new constraint. Specifically, $\vx'$ achieves a 0.06\% reduction in the terminal time compared to $\vx$. Therefore, the proposed state-centric necessary condition can sensitively detect the non-optimality of a given trajectory based on state information.

\begin{figure}[!t]
    \centering
    \includegraphics[width=\columnwidth]{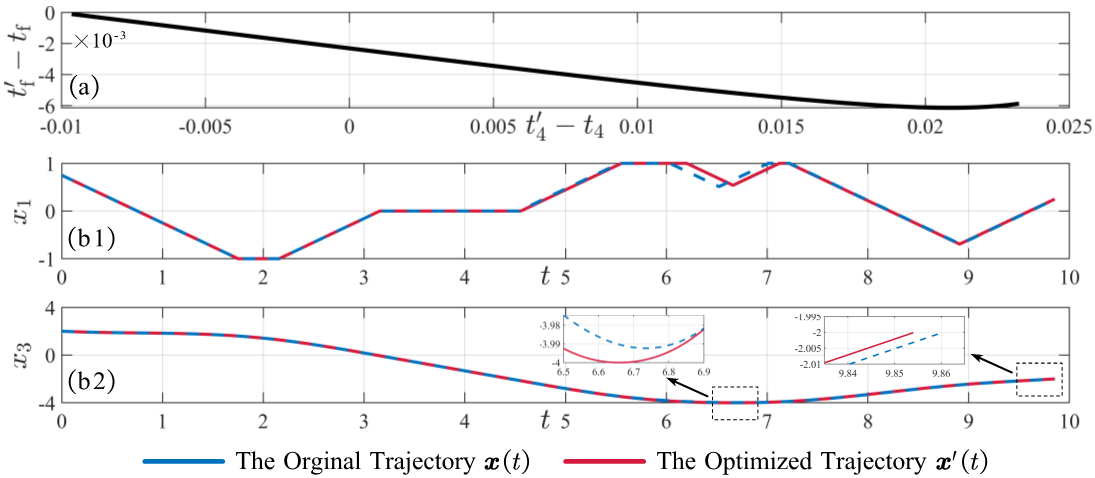}
    \caption{Optimizing a 4th-order near-optimal trajectory. (a) Plot of the locally minimal $t_\f'-t_\f$ and $t_4'-t_4$. (b) The original and the optimized trajectories.}
    \label{fig:simulation_disturb}
\end{figure}

\subsection{Optimizing a 5th-Order Feasible BBS Trajectory}\label{subsec:simulation_5_order}

Arc-based \cite{ezair2014planning} and discretized optimization \cite{leomanni2022time} methods are two main approaches to solving high-order problems. Consider a 5th-order problem for high-order chain-of-integrator systems \eqref{eq:optimalproblem_chainofintegrators} where $\vx_0=\vzero$, $\vxf=\ve_5$, $u_{\mathrm{m}}=1$, and $\vx_{\mathrm{m}}=\left(0.8,0.5,0.5,0.5,1\right)$. A feasible BBS trajectory $\vx_1$ is planned by \cite{ezair2014planning}, as shown in Fig. \ref{fig:ASL_iter_5order}(a). Denote $\mathrm{DOF}\triangleq N-\sum_{i=1}^{N}\abs{\systembehavior_i}-n$; hence, Corollary \ref{cor:chainofintegrators_onlysystembehaviors} implies that $\mathrm{DOF}\leq0$ holds for an optimal trajectory. As shown in Fig. \ref{fig:ASL_iter_5order}(b), $\mathrm{DOF}_1=6>0$ which does not satisfy the proposed state-centric necessary condition. Fix the first 5 arcs as well as the ASL, and vary the motion time of the 6th arc to search the minimal terminal time until reaching the boundary of the feasible set. If some arcs reach zero time, then delete the system behavior from the ASL.  Repeat the above process, resulting in trajectories $\vx_i$, $i\in\left[6\right]$. As shown in Figs. \ref{fig:ASL_iter_5order}(a) and \ref{fig:ASL_SOCPs_5order}(b), the terminal time decreases through iteration, and the DOF finally reaches 0 after 5 iterative steps. Finally, the optimized trajectory $\vx_6$ achieves a 17.7\% reduction in the terminal time compared to the original trajectory $\vx_1$. In fact, it remains challenging to plan near-optimal high-order trajectories using arc-based methods despite advantages in computational accuracy and analytical simplicity. The proposed state-centric condition has the potential to optimize BBS trajectories planned by arc-based methods, which will be investigated in future works.

\begin{figure}[!t]
    \centering
    \includegraphics[width=\columnwidth]{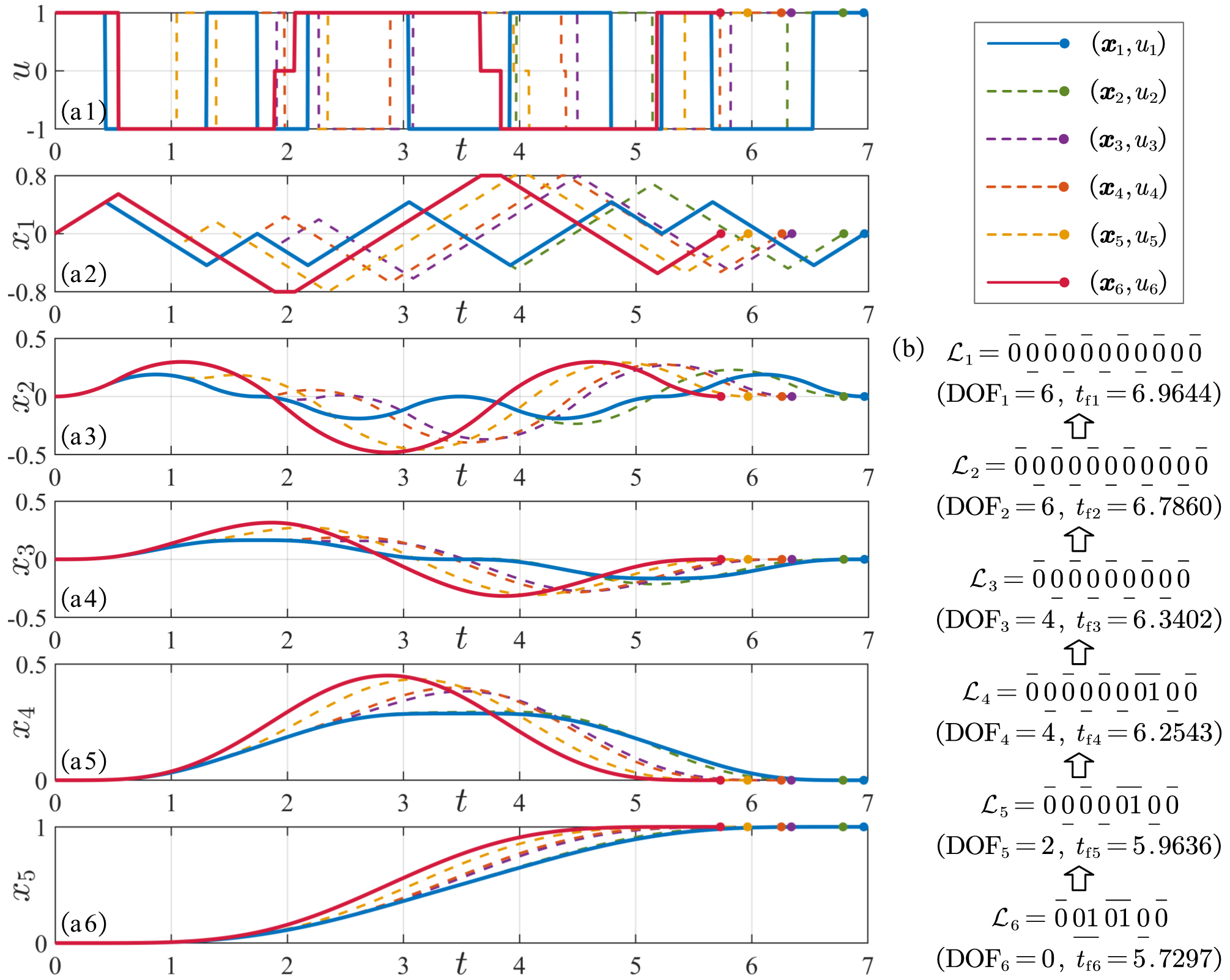}
    \caption{Iteration of a 5th-order feasible BBS trajectory. The initial trajectory $\vx_1$ is planned by \cite{ezair2014planning}. (a) The states and controls. (b) The ASLs.}
    \label{fig:ASL_iter_5order}
\end{figure}

A discretized optimization method \cite{leomanni2022time} is applied to plan the trajectory $\vx_{\text{base}}$ for comparison, where the time domain is discretized into 1000 intervals. An optimal solution is obtained with a 0.016\% relative error of terminal time compared to $\vx_6$, as shown in Fig. \ref{fig:ASL_SOCPs_5order}(b). Therefore, \cite{leomanni2022time} can plan optimal or near-optimal trajectories for the high-order problem. However, $u_{\text{base}}$ exhibits oscillation, as shown in Fig. \ref{fig:ASL_SOCPs_5order}(a1), while $u_i$ is BBS. Defining the total variation of the control as $T_\mathrm{v}\triangleq\int_{0}^{\tf}\abs{u\left(t\right)}\mathrm{d}t$ to describe the trajectory stability, $\vx_5$ achieves an 8.9\% $T_\mathrm{v}$ of $\vx_\text{base}$. For discretized optimization like \cite{leomanni2022time}, the open-loop errors of terminal states are difficult to constrain due to numerical error in each interval. Define $E_\mathrm{s}=\sqrt{\sum_{k=1}^{n}\left(\frac{\hat{x}_{\f k}-x_{\f k}}{x_{\mathrm{m}k}}\right)^2}$ where $\hat\vx_\f=\left(\hat{x}_{\f k}\right)_{k=1}^n$ is directly integrated from the control input. As shown in Fig. \ref{fig:ASL_SOCPs_5order}(d), arc-based trajectories $\vx_i$, $i\in\left[6\right]$, reduce at least 5 orders of magnitude of $E_\mathrm{s}$ compared to $\vx_{\text{base}}$. Hence, an approach combining arc-based methods and the proposed state-centric necessary condition has the potential to plan near-optimal trajectories with higher trajectory stability and computational accuracy compared to discretized optimization.

\section{Conclusion}
    This paper has set out to establish an innovative theoretical framework for time-optimal control of controllable linear systems with a single input. The proposed augmented switching law (ASL) represents the input control and the feasibility of a bang-bang and singular (BBS) trajectory in a compact form. Specifically, the equality and inequality constraints induced by the ASL represent a sufficient and necessary condition for the trajectory' feasibility. Based on the ASL, a given feasible trajectory can be perturbed with feasibility guarantees, resulting in a state-centric necessary condition for time optimality since any perturbed feasible trajectory should have a strictly longer terminal time than the optimal one. The proposed necessary condition states that the Jacobian matrix induced by the ASL must not be of full row rank. In contrast to traditional costate-based conditions, the developed necessary condition requires only states without dependence on costate information.

    The proposed state-centric necessary condition is applied to the optimal control for chain-of-integrator systems with full box constraints from both theoretical and numerical perspectives. An upper bound of the number of arcs is determined in a general sense, and a recursive equation for the junction time in chattering induced by 2nd-order constraints is derived. Note that the above two conclusions are challenging to prove by costate-based conditions. Numerical experiments showed that the proposed state-centric necessary condition can sensitively detect the non-optimality of a given trajectory based on state information, and it has the potential to optimize feasible BBS trajectories with higher trajectory stability and computational accuracy compared to discretized optimization methods.

    \begin{figure}[!t]
        \centering
        \includegraphics[width=\columnwidth]{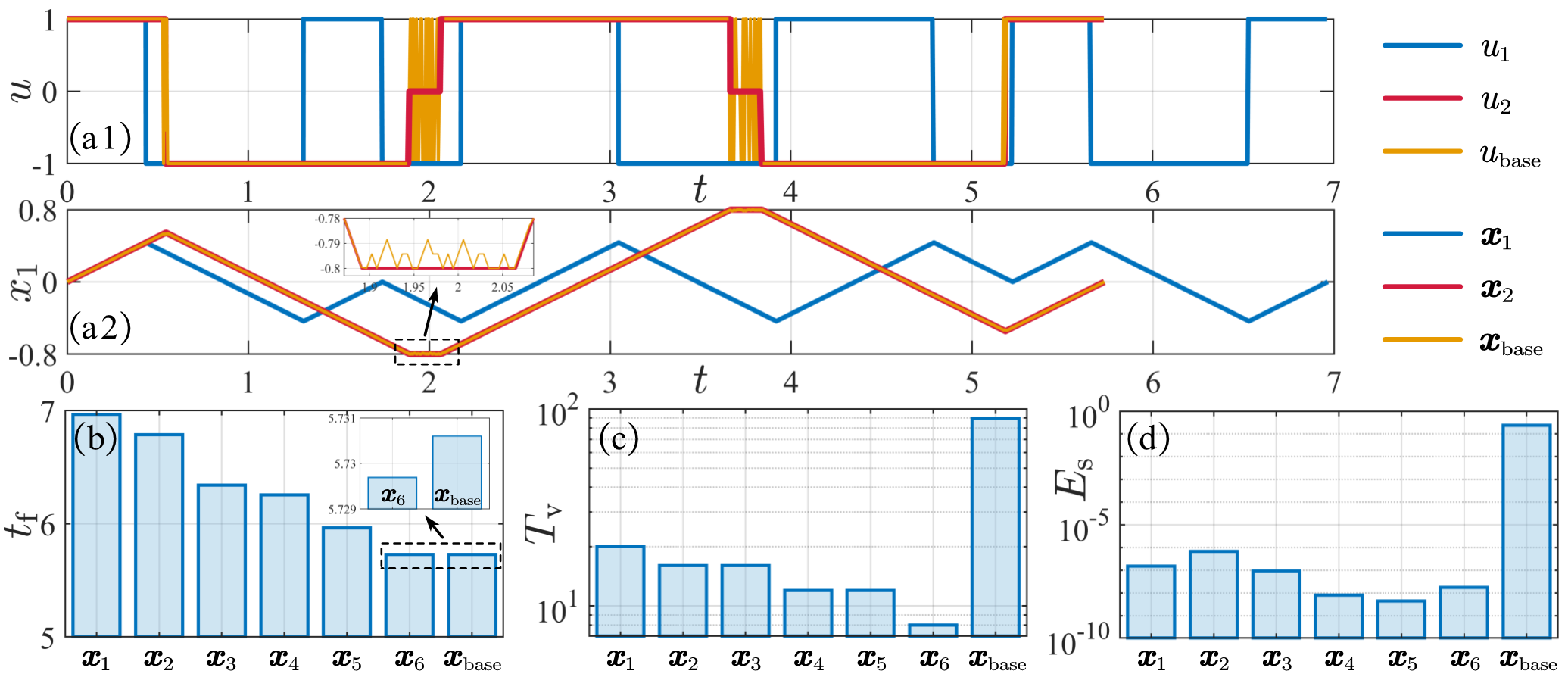}
        \caption{Comparison of the original trajectory $\vx_1$ \cite{ezair2014planning}, the optimized trajectory $\vx_6$ based-on the proposed state-centric necessary condition, and the trajectory $\vx_{\text{base}}$ planned by \cite{leomanni2022time}. (a) The states and controls. (b) The terminal time. (c) The total variation of the control. (d) The open-loop error of terminal states.}
        \label{fig:ASL_SOCPs_5order}
    \end{figure}

\section*{Acknowledgment}
    The authors would like to thank the reviewers for their valuable comments. This work was supported in part by the National Key Research and Development Program of China under Grant 2023YFB4302003, and in part by the National Natural Science Foundation of China under Grant 624B2077.

\ifCLASSOPTIONcaptionsoff
  \newpage
\fi

\bibliographystyle{myIEEEtran}
\bibliography{IEEEabrv,refs/ref}

\newpage

\onecolumn

\appendices

\section{Proofs of Propositions and Theorems}\label{sec:proofs}

\subsection{Proofs in Section \ref{sec:ArcAnalysis}}

\begin{proof}[Proof of Proposition \ref{prop:unconstrainedarc_nosingular}]
    Assume that a singular condition holds in the unconstrained arc, i.e., $\vb^\top\vlambda\equiv0$ for a period. By \eqref{eq:eta_constraint_zero}, $\vC\vx+\vd<\vzero$ implies that $\veta\equiv\vzero$ holds a.e. Hence, \eqref{eq:dotlambda} implies that $\dot\vlambda=-\vA^\top\vlambda$. Therefore,
    \begin{equation}
        \forall k\in\N,\,\vlambda^\top\vA^k\vb=\frac{\mathrm{d}^k}{\mathrm{d}t^k}\left(\vlambda^\top\vb\right)\equiv0,
    \end{equation}
    i.e., $\vlambda^\top\left[\vb,\vA\vb,\vA^2\vb,\dots,\vA^{n-1}\vb\right]\equiv\vzero$. However, by \eqref{eq:controllable}, $\vlambda\equiv0$ holds a.e. Then, \eqref{eq:hamilton_equiv_0} implies that $\lambda_0=0$, which contradicts $\left(\lambda_0,\vlambda\right)\not=0$. Therefore, $\vb^\top\vlambda\not\equiv0$ holds a.e.
\end{proof}

\begin{proof}[Proof of Proposition \ref{prop:constrainedarc_control}]
    Since $\vc_p\not=\vzero$, the controllability condition \eqref{eq:controllable} implies $\vc_p^\top\left[\vb,\vA\vb,\vA^2\vb,\dots,\vA^{n-1}\vb\right]\not=\vzero$; hence, $r_p$ exists and serves as the minimum non-zero column index of a non-zero matrix. For $r\in\left[r_p-1\right]$, considering the $r$th-order derivative of $\vc_p^\top\vx+d_p\equiv0$, \eqref{eq:constrainedarc_highorder_constraint} holds, and $\vc_p^\top\vA^{r_p-1}\left(\vA\vx+\vb u\right)=0$. Hence, \eqref{eq:constrainedarc_control} holds.

    Assume that $\exists t_0\in\left[t_1,t_2\right]$, s.t. $\vx\left(t_{0}\right)$ satisfies \eqref{eq:constrainedarc_highorder_constraint}. Under the driving of \eqref{eq:constrainedarc_control}, note that $\frac{\mathrm{d}^{r_p}}{\mathrm{d}t^{r_p}}\left(\vc_p^\top\vx+d_p\right)=\vc_p^\top\vA^{r_p-1}\left(\vA+\vb\widehat{\va}_p^\top\right)\vx=0$; hence, \eqref{eq:constrainedarc_highorder_constraint} holds in $\left[t_1,t_2\right]$.
\end{proof}

\begin{proof}[Proof of Lemma \ref{lemma:solution_linearODEs}]
    The proof of \eqref{eq:solution_linearODEs} is elementary \cite{stein2009real}. Assume that $\forall r\in\left[n\right]$, $\frac{\mathrm{d}^{r}\left(\vc^\top\vx\right)}{\mathrm{d}t^{r}}\left(t_0\right)=0$. Considering the minimal polynomial of $\vA$, $\exists p^*$, s.t. $p^*$ is a polynomial of degree less than $n$ and $\vA^n=p^*\left(\vA\right)$. Therefore, $\forall r\geq n$, $\exists p_{r}^*$ s.t. $p_{r}^*$ is a polynomial of degree less than $n$ and $\vA^r=p_{r}^*\left(\vA\right)$.

    Note that $\forall r\in\N^*$, $\frac{\mathrm{d}^{r}\left(\vc^\top\vx\right)}{\mathrm{d}t^{r}}=\vc^\top\vA^{r-1}\left(\vA\vx+\vb\right)$. Since $\forall r\in\left[n\right]$, $\frac{\mathrm{d}^{r}\left(\vc^\top\vx\right)}{\mathrm{d}t^{r}}\left(t_0\right)=0$, Note that $\forall r\geq n+1$,
    \begin{equation}\label{eq:cx0_IVP}
        \frac{\mathrm{d}^{r}\left(\vc^\top\vx\right)}{\mathrm{d}t^{r}}\left(t_0\right)=\vc^\top p_{r-1}^*\left(\vA\right)\left(\vA\vx\left(t_0\right)+\vb\right)=0.
    \end{equation}
    It can be proved that $\vc^\top\vx\left(t\right)$ is analytic. By \eqref{eq:cx0_IVP}, $\forall t\in\R$,
    \begin{equation}
        \vc^\top\vx\left(t\right)=\sum_{r=0}^{\infty}\frac{\left(t-t_0\right)^r}{r!}\where{\frac{\mathrm{d}^{r}}{\mathrm{d}t^{r}}\left(\vc^\top\vx\right)}{t_0}\equiv\vc^\top\vx\left(t_0\right).
    \end{equation}
    In other words, $\vc^\top\vx\left(t\right)\equiv\const$ on $\R$.
\end{proof}

\begin{proof}[Proof of Proposition \ref{prop:consistent_constraint}]
    Assume that $\forall p\in\mathcal{P}$, $\vc_p^\top\vx+d_p\equiv0$ holds in the constrained arc. By Proposition \ref{prop:constrainedarc_control}, \eqref{eq:constrainedarc_control} holds $\forall p\in\mathcal{P}$. $\forall q\in\mathcal{P}$, since $\vc_q^\top\vx\equiv\const$, \eqref{eq:constrainedarc_highorder_constraint_multiple} can be derived through considering the $r$th-order derivative of $\vc_q^\top\vx$ for $r\in\left[n\right]$.

    Assume that $\exists t_0\in\left[t_1,t_2\right]$, s.t. $\vx\left(t_0\right)$ satisfies \eqref{eq:constrainedarc_highorder_constraint_multiple}. $\exists p\in\mathcal{P}$, \eqref{eq:constrainedarc_highorder_constraint_multiple} holds in $\left[t_1,t_2\right]$. Then, $\forall q\in\mathcal{P}$, $r\in\left[n\right]$, $\frac{\mathrm{d}^{r}\left(\vc_q^\top\vx\right)}{\mathrm{d}t^{r}}\left(t_0\right)=0$ holds. By Lemma \ref{lemma:solution_linearODEs}, $\vc_q^\top\vx\equiv-d_q$ holds.
\end{proof}

\begin{proof}[Proof of Proposition \ref{prop:connection}]
    Consider the connection between two unconstrained arcs. By Proposition \ref{prop:unconstrainedarc_nosingular}, assume that $\forall i=1,2$, $u\equiv u_i$ on $\systembehavior_i$, where $u_i\in\left\{\pm\um\right\}$. Then, $\widehat{\vb}_i=u_i\vb$. Since $\widehat{\vA}_1=\widehat{\vA}_2=\vA$ and $\mathcal{P}_1=\mathcal{P}_2=\varnothing$, it holds that $\widehat{\vb}_1\not=\widehat{\vb}_2$; hence, $u_1=-u_2$. Therefore, Proposition \ref{prop:connection}.\ref{subprop:connection_unconstrained} holds.

    Consider the connection between two constrained arcs. Assume that $\mathcal{P}_1\cap\mathcal{P}_2\not=\varnothing$. Arbitrarily take a $p^*\in\mathcal{P}_1\cap\mathcal{P}_2$ into consideration. According to Proposition \ref{prop:constrainedarc_control}, $\forall t\in\left[t_0,t_2\right]$, $\dot\vx=\left(\vA-\vb\widehat{\va}_{p^*}^\top\right)\vx$ holds. Evidently, $\forall p\in\mathcal{P}_1\cup\mathcal{P}_2$, $\vc_p^\top\vx\left(t_1\right)+d_p\equiv0$, and $\forall r\in\left[r_p-1\right]$, $\vc_p^\top\vA^{r}\vx\left(t_1\right)\equiv0$. By Proposition \ref{prop:constrainedarc_control}, $\forall t\in\left[t_0,t_2\right]$, $\vc_p^\top\vx+d_p\equiv0$ holds. Therefore, $\mathcal{P}_1=\mathcal{P}_2$, which contradicts $\systembehavior_1\not=\systembehavior_2$. Hence, $\mathcal{P}_1\cap\mathcal{P}_2=\varnothing$.

    Furthermore, consider the monotonicity of $\vc_p^\top\vx+\vd_p$ at $t_1$ to guarantee the feasibility of $\systembehavior_1$ and $\systembehavior_2$. $\exists\varepsilon>0$, s.t. $\forall p\in\mathcal{P}_1$, $t\in\left(t_1,t_1+\varepsilon\right)$, $\vc_p^\top\vx\left(t\right)+\vd_p<0$. Note that on $\systembehavior_2$, $\dot\vx=\widehat{\vA}_2\vx$ and $\forall r\in\N^*$, $\where{\frac{\mathrm{d}^{r}}{\mathrm{d}t^{r}}\left(\vc_p^\top\vx+d_p\right)}{t_1^+}=\vc_p^\top\widehat{\vA}_2^r\vx\left(t_1\right)$. Lemma \ref{lemma:solution_linearODEs} implies that $\exists \hat{r}_p\in\left[n\right]$, s.t. $\where{\frac{\mathrm{d}^{\hat{r}_p}}{\mathrm{d}t^{\hat{r}_p}}\left(\vc_p^\top\vx+d_p\right)}{t_1^+}<0$ and $\forall r\in\left[\hat{r}_p-1\right]$, $\where{\frac{\mathrm{d}^{r}}{\mathrm{d}t^{r}}\left(\vc_p^\top\vx+d_p\right)}{t_1^+}=0$. Since $\forall r\in\left[r_p-1\right]$, $\where{\frac{\mathrm{d}^r}{\mathrm{d}t^r}\left(\vc_p^\top\vx+d_p\right)}{t_1^+}=0$, it holds that $\hat{r}_p\geq r_p$. In other words, \eqref{eq:connection_constrained_s1tos2} holds. For the same reason, \eqref{eq:connection_constrained_s2tos1} holds.

    Conversely, assume that $\dot\vx=\widehat{\vA}_1\vx$ for $\left(t_0,t_1\right)$ and $\dot\vx=\widehat{\vA}_2\vx$ for $\left(t_1,t_2\right)$. Assume that $\vx\left(t_1\right)$ satisfies \eqref{eq:connection_constrained_s1tos2} for $p\in\mathcal{P}_1$, \eqref{eq:connection_constrained_s2tos1} for $p\in\mathcal{P}_2$, and \eqref{eq:constrainedarc_highorder_constraint} for $p\in\mathcal{P}_1\cup\mathcal{P}_2$. Then, $\forall p\in\mathcal{P}_1$, $\vc_p^\top\vx\left(t_1\right)+d_p=0$ and
    \begin{equation}
        \begin{dcases}
            \where{\frac{\mathrm{d}^{\hat{r}_p}}{\mathrm{d}t^{\hat{r}_p}}\left(\vc_p^\top\vx+d_p\right)}{t_1^+}<0,\\
            \forall r\in\left[\hat{r}_p-1\right],\,\where{\frac{\mathrm{d}^r}{\mathrm{d}t^r}\left(\vc_p^\top\vx+d_p\right)}{t_1^+}=0.
        \end{dcases}
    \end{equation}
    Hence, $\exists\varepsilon>0$, $\forall p\in\mathcal{P}_1$, $t\in\left(t_1,t_1+\varepsilon\right)$, $\vc_p^\top\vx\left(t\right)+d_p<0$. For the same reason, $\exists\varepsilon'\in\left(0,\varepsilon\right)$, $\forall p\in\mathcal{P}_2$, $t\in\left(t_1-\varepsilon',t_1\right)$, $\vc_p^\top\vx\left(t\right)+d_p<0$. Hence, Proposition \ref{prop:connection}.\ref{subprop:connection_constrained} holds.

    Consider the connection between an unconstrained arc and a constrained arc. Assume that $\vx$ enters an unconstrained arc $\systembehavior_2$ from a constrained arc $\systembehavior_1$. By Proposition \ref{prop:constrainedarc_control}, $\forall p\in\mathcal{P}_1$, $r\in\N^*$, $\where{\frac{\mathrm{d}^r}{\mathrm{d}t^r}\left(\vc_p^\top\vx+d_p\right)}{t_1^-}=0$; hence,
    \begin{equation}
        \begin{dcases}
            \vc_p^\top\vx\left(t_1\right)+d_p=0,\\
            \forall r\in\left[r_p-1\right],\,\vc_p^\top\vA^{r}\vx\left(t_1\right)=0,\\
            \vc_p^\top\vA^{r_p-1}\left(\vA\vx\left(t_1\right)+u\left(t_1^-\right)\vb\right)=0.
        \end{dcases}
    \end{equation}
    Assume that $u\equiv u_2\in\left\{\pm\um\right\}$ on $\systembehavior_2$. Then, $\forall r\in\left[r_p-1\right]$, $\where{\frac{\mathrm{d}^r}{\mathrm{d}t^r}\left(\vc_p^\top\vx+d_p\right)}{t_1^+}=\vc_p^\top\vA^{r}\vx\left(t_1\right)=0$ since $\vc_p^\top\vA^{r-1}\vb=0$. Since $\exists\varepsilon>0$, s.t. $\forall t\in\left(t_1,t_1+\varepsilon\right)$, $\vc_p^\top\vx\left(t\right)+d_p<0$, it holds that $\where{\frac{\mathrm{d}^{r_p}}{\mathrm{d}t^{r_p}}\left(\vc_p^\top\vx+d_p\right)}{t_1^+}=\vc_p^\top\vA^{r_p-1}\left(\vA\vx\left(t_1\right)+u_2\vb\right)\leq0$. Therefore, $\vc_p^\top\vA^{r_p-1}\vb\left(u_2-u\left(t_1^-\right)\right)\leq0$. Since $\vc_p^\top\vA^{r_p-1}\vb\not=0$ and $\abs{u\left(t_1^-\right)}\leq\abs{u_2}=\um$, \eqref{eq:connection_unconstrained_constrained_control_s2unconstrained} holds.

    Furthermore, consider the monotonicity of $\vc_p^\top\vx+\vd_p$ at $t_1$ to guarantee the feasibility of $\systembehavior_2$ where $\vc_p^\top\vx+\vd_p<0$ for $\left(t_1,t_1+\varepsilon\right)$. Note that on $\systembehavior_2$, $\dot\vx=\vA\vx+u_2\vb$; hence, $\forall r\in\N^*$, $\where{\frac{\mathrm{d}^{r}}{\mathrm{d}t^{r}}\left(\vc_p^\top\vx+d_p\right)}{t_1^+}=\vc_p^\top\vA^{r-1}\left(\vA\vx\left(t_1\right)+u_2\vb\right)$. Since $\forall r\in\left[r_p-1\right]$, $\where{\frac{\mathrm{d}^{r}}{\mathrm{d}t^{r}}\left(\vc_p^\top\vx+d_p\right)}{t_1^+}=0$, Lemma \ref{lemma:solution_linearODEs} implies that $\exists \hat{r}_p\in\left[r_p,n\right]\cap\N$, s.t. $\where{\frac{\mathrm{d}^{\hat{r}_p}}{\mathrm{d}t^{\hat{r}_p}}\left(\vc_p^\top\vx+d_p\right)}{t_1^+}<0$ and $\forall r\in\left[\hat{r}_p-1\right]$, $\where{\frac{\mathrm{d}^{r}}{\mathrm{d}t^{r}}\left(\vc_p^\top\vx+d_p\right)}{t_1^+}=0$.

    Conversely, assume that $\dot\vx=\widehat{\vA}_1\vx$ for $\left(t_0,t_1\right)$ and $\dot\vx=\vA\vx+u_2\vb$ for $\left(t_1,t_2\right)$. Assume that $\vx\left(t_1\right)$ satisfies \eqref{eq:connection_unconstrained_constrained_s1tos2} and \eqref{eq:constrainedarc_highorder_constraint} for $p\in\mathcal{P}_1$. Then, $\vc_p^\top\vx\left(t_1\right)+d_p=0$ and
    \begin{equation}
        \begin{dcases}
            \where{\frac{\mathrm{d}^{\hat{r}_p}}{\mathrm{d}t^{\hat{r}_p}}\left(\vc_p^\top\vx+d_p\right)}{t_1^+}<0,\\
            \forall r\in\left[\hat{r}_p-1\right],\,\where{\frac{\mathrm{d}^r}{\mathrm{d}t^r}\left(\vc_p^\top\vx+d_p\right)}{t_1^+}=0.
        \end{dcases}
    \end{equation}
    Hence, $\exists\varepsilon>0$, $\forall p\in\mathcal{P}_1$, $t\in\left(t_1,t_1+\varepsilon\right)$, $\vc_p^\top\vx\left(t\right)+d_p<0$. By Proposition \ref{prop:constrainedarc_control}, $\forall p\in\mathcal{P}_1$, $t\in\left[t_0,t_1\right]$, $\vc_p^\top\vx\left(t\right)+d_p\equiv0$. So Proposition \ref{prop:connection}.\ref{subprop:connection_constrained_unconstrained} holds. Similarly, Proposition \ref{prop:connection}.\ref{subprop:connection_unconstrained_constrained} holds.
\end{proof}

\subsection{Proofs in Section \ref{sec:KeypointsFeasibilityOptimalTrajectory}}

\begin{proof}[Proof of Proposition \ref{prop:constrainedarc_end_feasibility}]
    Proposition \ref{prop:constrainedarc_end_feasibility}.\ref{prop:constrainedarc_end_feasibility_state} holds for the same reason as that of Proposition \ref{prop:connection}.

    Consider Proposition \ref{prop:constrainedarc_end_feasibility}.\ref{prop:constrainedarc_end_feasibility_control}. Proposition \ref{prop:constrainedarc_control} implies that $u=-\frac{\vc_p^\top\vA^{r_p}}{\vc_p^\top\vA^{r_p-1}\vb}\vx$. $\exists\varepsilon>0$, s.t. $u\left(t\right)\leq\um$ on $\left[t_0,t_0+\varepsilon\right]$ if and only if (1) $\exists\varepsilon>0$, s.t. $u<\um$ for $\left(t_0,t_0+\varepsilon\right)$, or (2) $\exists\varepsilon>0$, s.t. $u\left(t\right)\equiv\um$ on $\left[t_0,t_0+\varepsilon\right]$. By Proposition \ref{prop:constrainedarc_control}, (1) is equivalent to $\exists\hat{r}_p\in\left[0,n\right]\cap\N$, s.t. $\forall  r\in\left[0,\hat{r}_p-1\right]\cap\N$, $\frac{\mathrm{d}^r\left(u-\um\right)}{\mathrm{d}t^r}\left(t_0\right)=0$, and $\frac{\mathrm{d}^{\hat{r}_p}\left(u-\um\right)}{\mathrm{d}t^{\hat{r}_p}}\left(t_0\right)<0$. (2) is equivalent to $\forall r\in\left[0,n\right]\cap\N$, $\frac{\mathrm{d}^r\left(u-\um\right)}{\mathrm{d}t^r}\left(t_0\right)=0$.
\end{proof}

\begin{proof}[Proof of Proposition \ref{prop:tangentcondition}]
    $\vc^\top\vx\left(t_1\right)+d=0$. $\exists\varepsilon>0$, s.t. $B_\varepsilon\left(t_1\right)\subset\left(t_0,t_2\right)$ and $\forall t\in\Setminus{B_\varepsilon\left(t_1\right)}{\left\{t_1\right\}}$, $\vc^\top\vx\left(t\right)+d<0$. Note that $\forall r\in\N^*$, $\frac{\mathrm{d}^r}{\mathrm{d}t^r}\left(\vc^\top\vx+d\right)=\vc^\top\widehat{\vA}^{r-1}\left(\widehat{\vA}\vx+\widehat{\vb}\right)$.

    Assume that $\forall r\in\left[n\right]$, $\where{\frac{\mathrm{d}^r}{\mathrm{d}t^r}\left(\vc^\top\vx+d\right)}{t_1}=0$. By Lemma \ref{lemma:solution_linearODEs}, $\forall t\in B_\varepsilon\left(t_1\right)$, $\vc^\top\vx\left(t\right)+d\equiv0$, which contradicts the tangent condition. Hence, $\exists \hat{r}\in\left[n\right]$, s.t. $\forall r\in\left[\hat{r}-1\right]$, $\where{\frac{\mathrm{d}^r}{\mathrm{d}t^r}\left(\vc^\top\vx+d\right)}{t_1}=0$, and $\where{\frac{\mathrm{d}^{\hat{r}}}{\mathrm{d}t^{\hat{r}}}\left(\vc^\top\vx+d\right)}{t_1}\not=0$.

    If $\hat{r}$ is odd, then $\vc^\top\vx+d$ strictly crosses $0$ at $t_1$, which contradicts the tangent condition. Hence, $\hat{r}$ is even. If $\where{\frac{\mathrm{d}^{\hat{r}}}{\mathrm{d}t^{\hat{r}}}\left(\vc^\top\vx+d\right)}{t_1}>0$, then $\vc^\top\vx+d$ achieves a strict local minimum at $t_1$, which contradicts the maximum condition. Hence, $\where{\frac{\mathrm{d}^{\hat{r}}}{\mathrm{d}t^{\hat{r}}}\left(\vc^\top\vx+d\right)}{t_1}<0$. Therefore, \eqref{eq:tangentcondition} holds.

    Conversely, assume that \eqref{eq:tangentcondition} holds. Considering the Taylor of $\vc^\top\vx+d$ at $t_1$, it holds that $\vc^\top\vx+d<0$ at a deleted neighborhood of $t_1$. Hence, $\systembehavior$ is tangent to the boundary of the constraint $\vc^\top\vx+d$ at $t_1$.
\end{proof}

\subsection{Proofs in Section \ref{sec:StateCentricNecessaryCondition}}

\begin{proof}[Proof of Theorem \ref{thm:uniqueness_nochattering}]
    If $\forall t\in\left[0,t_\f^*\right]$, $\vx_1^*\left(t\right)=\vx_2^*\left(t\right)$, then it holds a.e. that $u_1^*\left(t\right)=u_2^*\left(t\right)$ since $\vx_1^*\left(t_0\right)=\vx_2^*\left(t_0\right)=\vx_0$.

    Assume that $\exists t'\in\left(0,t_\f^*\right)$, s.t. $\vx_1^*\left(t'\right)\not=\vx_2^*\left(t'\right)$. Let $t_0=\arg\min\left\{t\in\left(0,t_\f^*\right):\vx_1^*\left(t\right)\not=\vx_2^*\left(t\right)\right\}$. Evidently, $0\leq t_0<t_\f^*$. Assume that $\vx_1^*\left(t_0\right)\not=\vx_2^*\left(t_0\right)$. Then, $t_0>0$. Due to the continuity of $\vx$, $\forall t\in\left(0,t_0\right)$, $\vx_1^*\left(t\right)\equiv\vx_2^*\left(t\right)$ implies that $\vx_1^*\left(t_0\right)=\vx_2^*\left(t_0\right)$, which contradicts the assumption. Therefore, $\vx_1^*\left(t_0\right)=\vx_2^*\left(t_0\right)$ holds. Furthermore, $\forall t\in\left[0,t_0\right]$, $\vx_1^*\left(t\right)=\vx_2^*\left(t\right)$ implies that $u_1^*\left(t\right)=u_2^*\left(t\right)$ for $t\in\left[0,t_0\right]$ a.e.

    According to Bellman's principle of optimality \cite{bellman1952theory}, $\vx=\vx_1^*\left(t\right)$ and $\vx=\vx_2^*\left(t\right)$, $t\in\left[t_0,t_\f^*\right]$ are still the time-optimal trajectories between $\vx_1^*\left(t_0\right)=\vx_2^*\left(t_0\right)$ and $\vx_1^*\left(t_\f^*\right)=\vx_2^*\left(t_\f^*\right)=\vxf$. Let $u_3^*\left(t\right)=\frac12\left(u_1^*\left(t\right)+u_2^*\left(t\right)\right)$ and $\vx_3^*\left(t\right)=\frac12\left(\vx_1^*\left(t\right)+\vx_2^*\left(t\right)\right)$. Evidently, $u_3^*\left(t\right)$ is feasible since the feasible set of $\left(u,\vx\right)$ is convex; hence, $u_3^*\left(t\right)$ is optimal.

    Denote that $\forall k=1,2$, the switching law of $\vx=\vx_i^*\left(t\right)$, $t\in\left[t_0,t_\f^*\right]$ is $\switchinglaw_i=\systembehavior_1^{(i)}\systembehavior_2^{(i)}\dots\systembehavior_{N_i}^{(i)}$, where $\forall j\in\left[N_i\right]$, $\systembehavior_j^{(i)}=\left(\widehat{\vA}_j^{(i)},\widehat{\vb}_j^{(i)},\vF_j^{(i)},\vg_j^{(i)},\mathcal{P}_j^{(i)}\right)$ lasts for $t\in\left(t_{j-1}^{(i)},t_j^{(i)}\right)$. Among them, $t_0^{(i)}=t_0$ and $t_{N_i}^{(i)}=t_\f^*$. Evidently, if $\systembehavior_1^{(1)}=\systembehavior_1^{(2)}$, then $\forall t\in\left(t_0,\min\left\{t_1^{(1)},t_1^{(2)}\right\}\right)$, $u_1^*\left(t\right)=u_2^*\left(t\right)$ and $\vx_1^*\left(t\right)=\vx_2^*\left(t\right)$, which contradicts the definition of $t_0$. So $\systembehavior_1^{(1)}\not=\systembehavior_1^{(2)}$.

    Assume that both $\systembehavior_1^{(1)}$ and $\systembehavior_1^{(2)}$ are unconstrained arcs. Without loss of generality, assume that $u_1^*\left(t\right)\equiv\um$ in $\systembehavior_1^{(1)}$; hence, $u_2^*\left(t\right)\equiv-\um$ in $\systembehavior_1^{(2)}$. Then, $\exists\varepsilon>0$, s.t. $\forall t\in\left(t_0,t_0+\varepsilon\right)$, $\vC\vx_1^*\left(t\right)+\vd<\vzero$ and $\vC\vx_2^*\left(t\right)+\vd<\vzero$. Therefore, $\vC\vx_3^*\left(t\right)+\vd=\vC\frac{\vx_1^*\left(t\right)+\vx_2^*\left(t\right)}{2}+\vd<\vzero$. In other words, $\vx_3^*$ is strictly unconstrained in $\in\left(t_0,t_0+\varepsilon\right)$. According to Proposition \ref{prop:constrainedarc_control}, $\forall t\in\left(t_0,t_0+\varepsilon\right)$, $u_3^*\left(t\right)\in\left\{\pm\um\right\}$, which contradicts the fact that $u_3^*\left(t\right)=\frac{u_1^*\left(t\right)+u_2^*\left(t\right)}{2}\equiv0$. Therefore, one of $\systembehavior_1^{(1)}$ and $\systembehavior_1^{(2)}$ is constrained. Without loss of generality, assume that $\systembehavior_1^{(1)}$ is constrained.

    Assume that $\systembehavior_1^{(2)}$ is an unconstrained arc. Then, $\exists\varepsilon\in\left(0,\min\left\{t_1^{(1)},t_1^{(2)}\right\}-t_0\right)$, s.t. $\forall t\in\left(t_0,t_0+\varepsilon\right)$, $\vC\vx_2^*\left(t\right)+\vd<\vzero$. Since $\vC\vx_1^*\left(t\right)+\vd\leq\vzero$, it holds that $\vC\vx_3^*\left(t\right)+\vd=\vC\frac{\vx_1^*\left(t\right)+\vx_2^*\left(t\right)}{2}+\vd<\vzero$. In other words, $\vx_3^*$ is strictly unconstrained in $\in\left(t_0,t_0+\varepsilon\right)$. According to Proposition \ref{prop:constrainedarc_control}, $\forall t\in\left(t_0,t_0+\varepsilon\right)$, $u_3^*\left(t\right)=\frac{u_1^*\left(t\right)+u_2^*\left(t\right)}{2}\in\left\{\pm\um\right\}$. Note that $u_1^*\left(t\right)\in\left[-\um,\um\right]$ and $u_2^*\left(t\right)\equiv u_0\in\left\{\pm\um\right\}$; hence, $\forall t\in\left(t_0,t_0+\varepsilon\right)$, $u_1^*\left(t\right)\equiv u_0$. Therefore, $\vx_1^*\left(t\right)=\vx_2^*\left(t\right)$ for $\left[t_0,t_0+\varepsilon\right]$, which contradicts the definition of $t_0$.

    Therefore, both $\systembehavior_1^{(1)}$ and $\systembehavior_1^{(2)}$ are constrained arcs. Assume that $\mathcal{P}_1^{(1)}\cap\mathcal{P}_1^{(2)}\not=\varnothing$. Arbitrarily consider $p\in\mathcal{P}_1^{(1)}\cap\mathcal{P}_1^{(2)}$. $\forall t\in\left(t_0,\min\left\{t_1^{(1)},t_1^{(2)}\right\}\right)$, $i=1,2$, $\dot\vx_i^*\left(t\right)=\left(\vA-\vb\widehat{\va}_p^\top\right)\vx_i^*\left(t\right)$. $\vx_1^*\left(t_0\right)=\vx_2^*\left(t_0\right)$ implies that $\vx_1^*\left(t\right)\equiv\vx_2^*\left(t\right)$, which contradicts the definition of $t_0$. Therefore, $\mathcal{P}_1^{(1)}\cap\mathcal{P}_1^{(2)}=\varnothing$. Then, $\exists\varepsilon\in\left(0,\min\left\{t_1^{(1)},t_1^{(2)}\right\}-t_0\right)$, s.t. $\forall t\in\left(t_0,t_0+\varepsilon\right)$, $i=1,2$, $p\not\in\mathcal{P}_1^{(i)}$, $\vc_p^\top\vx_i^*\left(t\right)+d_p<0$. Hence, $\forall p\in\left[P\right]$, $\vc_p^\top\vx_3^*\left(t\right)+d_p=\vc_p^\top\frac{\vx_1^*\left(t\right)+u\vx_2^*\left(t\right)}{2}+d_p<0$. By Proposition \ref{prop:constrainedarc_control}, $u_3^*\left(t\right)\in\left\{\pm\um\right\}$; hence, $u_1^*\left(t\right)=u_2^*\left(t\right)\in\left\{\pm\um\right\}$, which contradicts the definition of $t_0$. Therefore, $\forall t\in\left[t_0,t_\f^*\right]$, $\vx_1^*\left(t\right)=\vx_2^*\left(t\right)$; hence, $u_1^*\left(t\right)=u_2^*\left(t\right)$ a.e.
\end{proof}

\begin{proof}[Proof of Proposition \ref{prop:disturbedtrajectory_error}]
    Since $\norm[]{\vt-\vt'}<\varepsilon<\frac12\min_{i\in\left[M-1\right]}\left\{\abs{t_{i+1}-t_i}\right\}$, it holds that $\forall i\in\left[M-1\right]$, $\max\left\{t_i,t_i'\right\}<\min\left\{t_{i+1},t_{i+1}'\right\}$. $\forall i\in\left[M\right]$, let $\hat{t}_{2i-1}=\min\left\{t_i,t_i'\right\}$, $\hat{t}_{2i}=\max\left\{t_i,t_i'\right\}$, and $\hat{t}_0=t_0$. Then $\left\{\hat{t}_i\right\}_{i=0}^{2M}$ increases monotonically.

    Note that $\forall\vA\in\R^{n\times n}$, $\e^{\vA t}$ is locally Lipschitz continuous w.r.t. $t$. Therefore, $\exists C_1>0$, s.t. $\forall t\in\left[0,t_M-t_0\right]$, $\max_{i\in\left[M\right]}\norm[]{\e^{\vA_i t}-\vI}\leq C_1t$. Let $C_2=\max_{i\in\left[M\right]}\norm[]{\vb_i}$ and $C_3=\max_{i\in\left[M\right]}\sup_{t\in\left[0,t_M-t_0\right]}\norm[]{\ve^{\vA_it}}\geq1$.

    By Lemma \ref{lemma:solution_linearODEs}, $\forall i\in\left[M-1\right]\cup\left\{0\right\}$, $t\in\left[\hat{t}_{2i},\hat{t}_{2i+1}\right]$,
    \begin{equation}
        \begin{dcases}
            \vx\left(t\right)=\e^{\vA_i\left(t-\hat{t}_{2i}\right)}\vx\left(\hat{t}_{2i}\right)+\int_{\hat{t}_{2i}}^{t}\e^{\vA_i\left(t-\tau\right)}\vb_i\mathrm{d}\tau,\\
            \vx'\left(t\right)=\e^{\vA_i\left(t-\hat{t}_{2i}\right)}\vx'\left(\hat{t}_{2i}\right)+\int_{\hat{t}_{2i}}^{t}\e^{\vA_i\left(t-\tau\right)}\vb_i\mathrm{d}\tau.
        \end{dcases}
    \end{equation}
    Therefore, $\forall t\in\left[\hat{t}_{2i},\hat{t}_{2i+1}\right]$, $\norm[]{\e^{\vA_i\left(t-\hat{t}_{2i}\right)}}\leq C_3$ implies that
    \begin{equation}\label{prop:disturbedtrajectory_error_0to1}
        \begin{aligned}
            &\norm[]{\vx\left(t\right)-\vx'\left(t\right)}\leq C_3\norm[]{\vx\left(\hat{t}_{2i}\right)-\vx'\left(\hat{t}_{2i}\right)}.
        \end{aligned}
    \end{equation}

    $\forall i\in\left[M\right]$, let (a) $j=i+1$, $j'=i$ if $t_i\leq t_i'$; (b) $j=i$, $j'=i+1$ if $t_i>t_i'$. Then, $\forall t\in\left[\hat{t}_{2i-1},\hat{t}_{2i}\right]$, $\dot\vx=\vA_j\vx+\vb_j$ and $\dot\vx'=\vA_{j'}\vx+\vb_{j'}$. Lemma \ref{lemma:solution_linearODEs} implies that $\forall t\in\left[\hat{t}_{2i-1},\hat{t}_{2i}\right]$,
    \begin{equation}
        \begin{dcases}
            \vx\left(t\right)=\e^{\vA_{j}\left(t-\hat{t}_{2i-1}\right)}\vx\left(\hat{t}_{2i-1}\right)+\int_{\hat{t}_{2i-1}}^{t}\e^{\vA_{j}\left(t-\tau\right)}\vb_{j}\mathrm{d}\tau,\\
            \vx'\left(t\right)=\e^{\vA_{j'}\left(t-\hat{t}_{2i-1}\right)}\vx'\left(\hat{t}_{2i-1}\right)+\int_{\hat{t}_{2i-1}}^{t}\e^{\vA_{j'}\left(t-\tau\right)}\vb_{j'}\mathrm{d}\tau.
        \end{dcases}
    \end{equation}
    Therefore, $\forall t\in\left[\hat{t}_{2i-1},\hat{t}_{2i}\right]$, it holds that
    \begin{equation}\label{prop:disturbedtrajectory_error_1to2}
        \begin{aligned}
            &\norm[]{\vx\left(t\right)-\vx'\left(t\right)}\\
            \leq&\norm[]{\e^{\vA_{j}\left(t-\hat{t}_{2i-1}\right)}-\e^{\vA_{j'}\left(t-\hat{t}_{2i-1}\right)}}\norm[]{\vx\left(\hat{t}_{2i-1}\right)}\\
            +&\norm[]{\e^{\vA_{j'}\left(t-\hat{t}_{2i-1}\right)}}\norm[]{\vx\left(\hat{t}_{2i-1}\right)-\vx'\left(\hat{t}_{2i-1}\right)}\\
            +&\varepsilon\norm[]{\e^{\vA_{j}\left(t-\hat{t}_{2i-1}\right)}\vb_j-\e^{\vA_{j'}\left(t-\hat{t}_{2i-1}\right)}\vb_{j'}}\\
            \leq&2\left(C_1\norm[]{\vx}+C_2C_3\right)\varepsilon+C_3\norm[]{\vx\left(\hat{t}_{2i-1}\right)-\vx'\left(\hat{t}_{2i-1}\right)}.
        \end{aligned}
    \end{equation}

    Since $\vx\left(\hat{t}_0\right)=\vx'\left(\hat{t}_0\right)$, \eqref{prop:disturbedtrajectory_error_0to1} and \eqref{prop:disturbedtrajectory_error_1to2} implies that
    \begin{equation}
        \norm[\infty]{\vx-\vx'}\leq 2\left(C_3^{2M-2}-1\right)\left(C_1\norm[]{\vx}+C_2C_3\right)\varepsilon.
    \end{equation}
    Therefore, \eqref{eq:disturbedtrajectory_error_sametime} holds.
\end{proof}

\begin{proof}[Proof of Theorem \ref{thm:disturbedtrajectory_feasibility}]
    Fix $\vx=\vx\left(t\right)$. Firstly, let
    \begin{equation}
        \varepsilon<\frac12\min_{i\in\left[M-1\right]}\left\{\abs{t_{i+1}-t_i}\right\}.
    \end{equation}
    According to Proposition \ref{prop:disturbedtrajectory_error}, $\exists C>0$ only dependent on $\vx=\vx\left(t\right)$, s.t. if $\norm[]{\vt-\vt'}<\varepsilon$, then \eqref{eq:disturbedtrajectory_error_sametime} holds.

    Secondly, consider the inequality constraints $\vc_p^\top\vx'+d_p\leq0$. Denote $\left\{i^{(p)}_j\right\}_{j=0}^{N_p}\subset\left[M\right]\cup\left\{0\right\}$ increasing strictly monotonically, s.t. $i^{(p)}_0=0$, $i^{(p)}_{N_p}=M$, and $\forall j\in\left[N_p-1\right]$, $\vc_p^\top\vx+d_p\leq0$ switches between active and inactive at $t_{i^{(p)}_{j}}$. In other words, $\forall j\in\left[N_p-1\right]$, $\vc_p^\top\vx\left(t_{i_j^{(p)}}\right)+d_p=0$ holds. Furthermore, $\forall j\in\left[N_p\right]$, $t\in\left(t_{i^{(p)}_{j-1}},t_{i^{(p)}_{j}}\right)$, either $\vc_p^\top\vx\left(t\right)+d_p\equiv0$ or $\vc_p^\top\vx\left(t\right)+d_p<0$ holds. Since $\switchinglaw=\switchinglaw'$, the above property on $\left(\vx\left(t\right),\vt\right)$ holds for $\left(\vx'\left(t\right),\vt'\right)$ as well.

    Consider the case where $\vc_p^\top\vx+d_p\equiv0$ in $\left(t_{i^{(p)}_{j-1}},t_{i^{(p)}_{j}}\right)$, as shown in Fig. \ref{fig:proof_theorem2}(a). $\switchinglaw=\switchinglaw'$ implies that $\vc_p^\top\vx'+d_p\equiv0$ holds. Therefore, $\vc_p^\top\vx'\left(t\right)+d_p\leq0$ is feasible in $\left(t_{i^{(p)}_{j-1}}',t_{i^{(p)}_{j}}'\right)$.

    Consider the case where $\vc_p^\top\vx+d_p<0$ in $\left(t_{i^{(p)}_{j-1}},t_{i^{(p)}_{j}}\right)$, as shown in Fig. \ref{fig:proof_theorem2}(b). The feasibility of $\vx'\left(t\right)$ in $\left[t_{i^{(p)}_{j-1}}',t_{i^{(p)}_{j}}'\right]$ is discussed in two parts, i.e., far from and near the keypoints.

    Denote $f_p\left(t\right)=\vc_p^\top\vx\left(t\right)+d_p$ and $f_p'\left(t\right)=\vc_p^\top\vx\left(t\right)+d_p$. Then, $\forall r\in\N^*$, $\frac{\mathrm{d}^rf_p}{\mathrm{d}t^r}=\vc_p^\top\vA_{i_j^{(p)}}^{r-1}\left(\vA_{i_j^{(p)}}\vx+\vb_{i_j^{(p)}}\right)$ is continuous in $\left(t_{i^{(p)}_{j-1}},t_{i^{(p)}_{j-1}+1}\right)$. If $f_p\left(t_{i^{(p)}_{j-1}}\right)<0$ holds, then let $r_j^{(p)}=0$ and $t_{1,j}^{(p)}=t_{i_{j-1}^{(p)}}$; otherwise, by Propositions \ref{prop:connection}, \ref{prop:constrainedarc_end_feasibility}, and \ref{prop:tangentcondition}, $\exists r_j^{(p)}\in\left[n\right]$, s.t. $f_{p,j}\triangleq\frac{\mathrm{d}^{r_j^{(p)}}f_p}{\mathrm{d}t^{r_j^{(p)}}}\left(t_{i^{(p)}_{j-1}}^+\right)<0$ and $\forall r\in\left[r_j^{(p)}-1\right]$, $\frac{\mathrm{d}^{r}f_p}{\mathrm{d}t^r}\left(t_{i^{(p)}_{j-1}}^+\right)=0$. Let $t_{1,j}^{(p)}\in\left(t_{i^{(p)}_{j-1}},t_{i^{(p)}_{j-1}+1}\right)$ is sufficiently small, s.t. $\forall t\in\left(t_{i^{(p)}_{j-1}},t_{1,j}^{(p)}\right)$, $\frac{\mathrm{d}^{r_j^{(p)}}f_p}{\mathrm{d}t^{r_j^{(p)}}}\left(t\right)<\frac12f_{p,j}<0$ holds. $t_{2,j}^{(p)}\leq t_{i^{(p)}_j}$ is constructed similarly. Then, $f_p\left(t\right)$ decreases strictly monotonically in $\left(t_{i^{(p)}_{j-1}},t_{1,j}^{(p)}\right)$ and increases strictly monotonically in $\left(t_{2,j}^{(p)},t_{i^{(p)}_j}\right)$.

    \begin{figure}[!t]
        \centering
        \includegraphics[width=0.5\columnwidth]{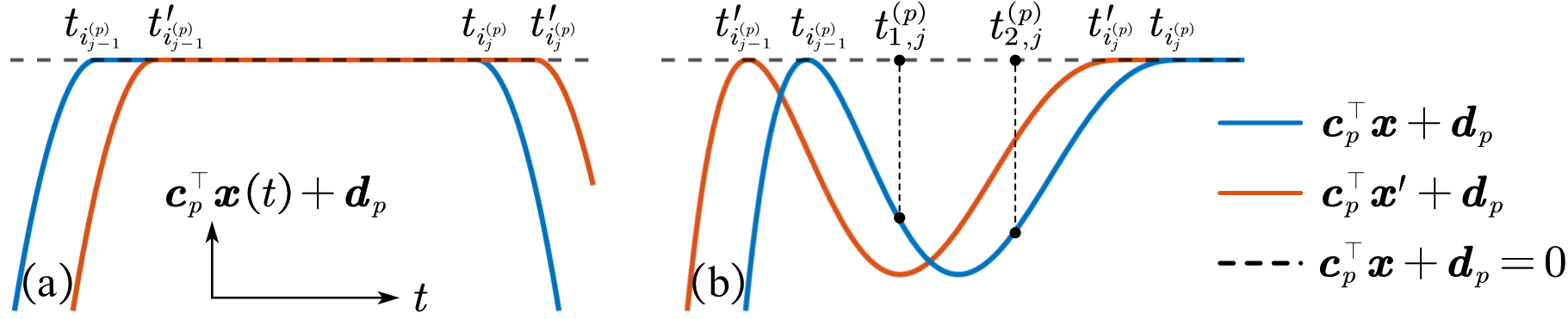}
        \caption{Feasibility of the perturbed trajectory near the keypoints. (a) Constrained arc. (b) Unconstrained arc.}
        \label{fig:proof_theorem2}
    \end{figure}

    Due to the continuity of $f_p$, $\sup_{t\in\left[t_{1,j}^{(p)},t_{2,j}^{(p)}\right]}f_p\left(t\right)<0$ holds. For the part far from keypoints, i.e., $\left[t_{1,j}^{(p)},t_{2,j}^{(p)}\right]$, let
    \begin{equation}\label{eq:disturbedtrajectory_feasibility_far}
        \varepsilon<-\left(4C\norm[]{\vc_p}\right)^{-1}\sup\nolimits_{t\in\left[t_{1,j}^{(p)},t_{2,j}^{(p)}\right]}f_p\left(t\right).
    \end{equation}
    Then, according to \eqref{eq:disturbedtrajectory_error_sametime}, $\forall t\in\left[t_{1,j}^{(p)},t_{2,j}^{(p)}\right]$, it holds that
    \begin{equation}
        f_p'\left(t\right)\leq f_p\left(t\right)+\norm[]{\vc_p}\norm[\infty]{\vx-\vx'}<0.
    \end{equation}
    In other words, $\vc_p^\top\vx'\left(t\right)+d_p\leq0$ is feasible in $\left[t_{1,j}^{(p)},t_{2,j}^{(p)}\right]$.

    Consider the part near the left keypoint, i.e., $\left[t_{i^{(p)}_{j-1}}',t_{1,j}^{(p)}\right]$. Assume that $t_{i^{(p)}_{j-1}}'<t_{1,j}^{(p)}$; otherwise, $\left[t_{i^{(p)}_{j-1}}',t_{1,j}^{(p)}\right]=\varnothing$.

    For the case where $f_p\left(t_{i^{(p)}_{j-1}}\right)<0$ and $t_{1,j}^{(p)}=t_{i_{j-1}^{(p)}}$, let $\varepsilon$ be sufficiently small, s.t. $f_p\left(t\right)\leq\frac12f_p\left(t_{i^{(p)}_{j-1}}\right)$ in $\left[t_{i^{(p)}_{j-1}}',t_{i^{(p)}_{j-1}}\right]$. By \eqref{eq:disturbedtrajectory_error_sametime} and \eqref{eq:disturbedtrajectory_feasibility_far}, $\forall t\in\left[t_{i^{(p)}_{j-1}}',t_{i^{(p)}_{j-1}}\right]$, $f_p'\left(t\right)\leq0$ is feasible since $f_p'\left(t\right)\leq f_p\left(t\right)+\norm[]{\vc_p}\norm[\infty]{\vx-\vx'}\leq\frac14f_p\left(t_{i^{(p)}_{j-1}}\right)<0$.

    For the case where $f_p\left(t_{i^{(p)}_{j-1}}\right)=0$ and $t_{i_{j-1}^{(p)}}<t_{1,j}^{(p)}$, let $\varepsilon<\hat\varepsilon\triangleq-f_{p,j}\left(4C\norm[]{\vc_p}\norm[]{\vA_{i_j^{(p)}}}^{r_j^{(p)}}+1\right)^{-1}$.

    Hence, $\forall t\in\left[t_{i^{(p)}_{j-1}},t_{1,j}^{(p)}\right]$, it holds that
    \begin{align}
            &\frac{\mathrm{d}^{r_j^{(p)}}f_p'}{\mathrm{d}t^{r_j^{(p)}}}\left(t\right)\leq\frac{\mathrm{d}^{r_j^{(p)}}f_p}{\mathrm{d}t^{r_j^{(p)}}}\left(t\right)+\norm[]{\vc_p}\norm[]{\vA_{i_j^{(p)}}}^{r_j^{(p)}}\norm[\infty]{\vx-\vx'}\notag\\
            \leq&\frac12f_{p,j}-\frac14f_{p,j}<0.
    \end{align}
    If $t_{i^{(p)}_{j-1}}'<t_{i^{(p)}_{j-1}}$, let $\varepsilon$ be sufficiently small, s.t.
    \begin{equation}
        \sup\nolimits_{\delta\in\left[0,\varepsilon\right]}\norm[]{\vx\left(t_{i^{(p)}_{j-1}}-\delta\right)-\vx\left(t_{i^{(p)}_{j-1}}\right)}<\hat\varepsilon.
    \end{equation}
    Then, $\forall t\in\left[t_{i^{(p)}_{j-1}}',t_{i^{(p)}_{j-1}}\right]\subset\left[t_{i^{(p)}_{j-1}}-\varepsilon,t_{i^{(p)}_{j-1}}\right]$, it holds that
    \begin{align}
        &\frac{\mathrm{d}^{r_j^{(p)}}f_p'}{\mathrm{d}t^{r_j^{(p)}}}\left(t\right)\leq
        \vc_p^\top\vA_{i_j^{(p)}}^{r_j^{(p)}-1}\left(\vA_{i_j^{(p)}}\vx\left(t_{i^{(p)}_{j-1}}\right)+\vb_{i_j^{(p)}}\right)\notag\\
        +&\norm[]{\vc_p}\norm[]{\vA_{i_j^{(p)}}}^{r_j^{(p)}}\left(\norm[\infty]{\vx-\vx'}+\norm[]{\vx\left(t\right)-\vx\left(t_{i^{(p)}_{j-1}}\right)}\right)\notag\\
        \leq&f_{p,j}-\frac12f_{p,j}<0.
    \end{align}
    Therefore, once $\varepsilon$ is sufficiently small, then $\frac{\mathrm{d}^{r_j^{(p)}}f_p'}{\mathrm{d}t^{r_j^{(p)}}}<0$ holds in $\left[t_{i^{(p)}_{j-1}}',t_{1,j}^{(p)}\right]$. $\switchinglaw'=\switchinglaw$ implies that $\forall r\in\left[r_j^{(p)}-1\right]$, $\frac{\mathrm{d}^{r}f_p'}{\mathrm{d}t^r}\left(t_{i^{(p)}_{j-1}}^+\right)=0$; hence, $f_p'$ decreases strictly monotonically in $\left(t_{i^{(p)}_{j-1}}',t_{1,j}^{(p)}\right)$. Therefore, $\forall t\in\left[t_{i^{(p)}_{j-1}}',t_{1,j}^{(p)}\right]$, $\vc_p^\top\vx'\left(t\right)+d_p<0$ is feasible. Similarly, once $\varepsilon$ is sufficiently small, then $\vc_p^\top\vx'\left(t\right)+d_p<0$ is feasible in $\left[t_{2,j}^{(p)},t_{i^{(p)}_j}'\right]$.

    Note that the above conditions on $\varepsilon$ only depends on $\vx=\vx\left(t\right)$. In summary, once $\varepsilon$ is sufficiently small, then $\forall p\in\left[P\right]$, $t\in\left[t_0',t_M'\right]$, $\vc^\top\vx'\left(t\right)+d_p\leq0$ holds.

    Thirdly, the control constraints $\abs{u'}\leq\um$ can be proved by applying a similar analysis when $\varepsilon$ is sufficiently small.
\end{proof}

\begin{proof}[Proof of Proposition \ref{prop:Jacobian}]
    $\forall j>i$, $\frac{\partial\vx_i}{\partial t_j}=\vzero$ holds evidently. According to Lemma \ref{lemma:solution_linearODEs}, $\vx_i=\e^{\vA_i\left(t_i-t_{i-1}\right)}\vx_{i-1}+\int_{t_{i-1}}^{t_i}\e^{\vA_i\left(t_i-\tau\right)}\vb_i\mathrm{d}\tau$; hence, $\forall i\in\left[M\right]$, $\frac{\partial\vx_i}{\partial t_i}=\vA_i\vx_i+\vb_i$ holds. Then, $\forall i>1$, it holds that
    \begin{equation}
        \begin{aligned}
            &\frac{\partial\vx_i}{\partial t_{i-1}}=\e^{\vA_i\left(t_i-t_{i-1}\right)}\left(\frac{\partial\vx_{i-1}}{\partial t_{i-1}}-\vA_{i}\vx_{i-1}-\vb_{i}\right)\\
            =&\e^{\vA_i\left(t_i-t_{i-1}\right)}\left(\left(\vA_{i-1}-\vA_{i}\right)\vx_{i-1}+\vb_{i-1}-\vb_{i}\right).
        \end{aligned}
    \end{equation}
    So $\forall i>1$, $j<i-1$, $\frac{\partial\vx_i}{\partial t_j}=\prod_{k=i}^{j+2}\e^{\vA_k\left(t_k-t_{k-1}\right)}\frac{\partial\vx_{j+1}}{\partial t_j}$.
\end{proof}


\begin{proof}[Proof of Theorem \ref{thm:Jacobian_necessary_condition}]
    Assume that $\vH:\R^{M}\to\R^{M'}$. $M'$ is the number of linearly independent equality constraints on keypoints, and $M$ is the number of keypoints. Assume that $\frac{\partial\vH}{\partial\vt_{1:\left(M-1\right)}}\left(\vt^*\right)$ has full row rank, i.e.,
    \begin{equation}\label{eq:Jacobian_necessary_condition}
        \rank\frac{\partial\vH}{\partial\vt_{1:\left(M-1\right)}}\left(\vt^*\right)=M'\leq M-1.
    \end{equation}
    Assume that the $i$-th row of $\frac{\partial\vH}{\partial\vt_{1:\left(M-1\right)}}\left(\vt^*\right)$ is linearly independent where $i\in\mathcal{I}\subset\left[M-1\right]$ and $\abs{\mathcal{I}}=M'$. Let $\varepsilon>0$ be sufficiently small to satisfy the condition of Theorem \ref{thm:disturbedtrajectory_feasibility} and the implicit function theorem \cite{stein2009real}. The implicit function $\vf:\left(t_i\right)_{i\in\Setminus{\left[M\right]}{\mathcal{I}}}\mapsto\left(t_i\right)_{i\in\mathcal{I}}$ is induced by $\vH\left(\vt\right)=\vzero$. Based on $\vf$, $\exists\vt'$ satisfies $\vH\left(\vt'\right)=\vzero$, $t_M'<t_M^*$, and $\norm[]{\vt'-\vt^*}<\varepsilon$. According to Theorem \ref{thm:disturbedtrajectory_feasibility}, the perturbed trajectory $\vx=\vx'\left(t\right)$ is feasible since $\vH\left(\vt'\right)=\vzero$ implies that $\vx'$ and $\vx^*$ have the same ASL. However, $\vx'$ achieves a shorter terminal time $t_M'<t_M^*$, which contradicts the optimality of $\vx^*$.
\end{proof}

\subsection{Proofs in Section \ref{sec:Applications}}

\begin{proof}[Proof of Corollary \ref{cor:chainofintegrators_purecontrolconstriant}]
    Chattering does not occur since no state inequality constraints exist. Assume that the optimal control switches $\left(N-1\right)$ times. In other words, $\forall i\in\left[N\right]$, $t\in\left(t_{i-1},t_i\right)$, $u\left(t\right)\equiv u_i\in\left\{\pm\um\right\}$. Note that $\vJ=\left[-\Delta u_{i+1}\vphi\left(t_N-t_i\right)\right]_{i=1}^{N-1}\in\R^{n\times\left(N-1\right)}$. The Vandermonde form of $\vJ$ implies that $\rank\vJ=\min\left\{n,N-1\right\}$ since $\Delta u_i\not=0$. By Theorem \ref{thm:Jacobian_necessary_condition}, $\rank\vJ<n$; hence, $N-1<n$.
\end{proof}

\begin{proof}[Proof of Corollary \ref{cor:chainofintegrators_onlysystembehaviors}]
    For a constrained arc $\systembehavior_i$, add the equality constraint at $t_i$, i.e., $\vx_{i,1:\abs{\systembehavior_i}}=\sgn\left(\systembehavior_i\right)x_{\m\abs{\systembehavior_i}}\ve_{\abs{\systembehavior_i}}$. Note that if $\systembehavior_i$ is constrained, i.e., $\abs{\systembehavior_i}\not=0$, then $u\equiv0$ in $\systembehavior_i$. Therefore, the constraints induced by $\switchinglaw$ are $\vx_{N}=\vxf$ and $\vx_{i,1:\abs{\systembehavior_i}}=\sgn\left(\systembehavior_i\right)x_{\m\abs{\systembehavior_i}}\ve_{\abs{\systembehavior_i}}$, $\forall i\in\left[N\right]$, $\abs{\systembehavior_i}\not=0$.

    By Condition \ref{condition:3_2_switchinglaw_ijN_}, $\abs{\systembehavior_N}=0$. Denote $\abs{\systembehavior_{N+1}}=n$, and $\mathcal{I}=\left\{i\in\left[N+1\right]:\abs{\systembehavior_i}\not=0\right\}$. According to Proposition \ref{prop:Jacobian}, $\vJ=\left(-\Delta u_{j+1}\vphi_{1:\abs{\systembehavior_i}}\left(t_i-t_j\right)\delta_{j<i}\right)_{i\in\mathcal{I},\,j\in\left[N-1\right]}$, where $\delta_{A}$ is the indicator function of condition $A$.

    Let $\vA\leftrightarrow\vB$ mean that $\vA$ has full row rank if and only if $\vB$ has full row rank. By $\forall j\in\left[N-1\right]$, $\Delta u_{j+1}\not=0$, it holds that $\vJ\leftrightarrow\left(\vphi_{1:\abs{\systembehavior_i}}\left(t_i-t_j\right)\delta_{j<i}\right)_{i\in\mathcal{I},\,j\in\left[N-1\right]}$.

    By iteratively taking differences between adjacent rows of the right-hand side, dividing by the time difference, and eliminating, it can be proved that $\vJ$ has full rank based on Conditions \ref{condition:3_2_switchinglaw_1ji}, \ref{condition:3_2_switchinglaw_ijN}, and \ref{condition:3_2_switchinglaw_ijN_}. The proof is similar to our previous work \cite{wang2025time} and is omitted due to space limitation. According to Theorem \ref{thm:Jacobian_necessary_condition}, the number of rows should be greater than the number of columns, i.e., $\sum_{i\in\mathcal{I}}\abs{\systembehavior_i}=\sum_{i=1}^{N}\abs{\systembehavior_i}+n>N-1$.
\end{proof}

\begin{proof}[Proof of Corollary \ref{cor:chainofintegrators_chattering_s2}]
    Since $\forall k\in\N^*$, $\vx_{1:2}\left(t_{3k-3}\right)=\vx_{1:2}\left(t_{3k}\right)=x_{\m2}\ve_2$, it can be proved that
    \begin{equation}
        \begin{dcases}
            u\left(t\right)\equiv u_{3k-2}=-\um,&t\in\left(t_{3k-3},t_{3k-2}\right),\\
            u\left(t\right)\equiv u_{3k-1}=\um,&t\in\left(t_{3k-2},t_{3k-1}\right),\\
            u\left(t\right)\equiv u_{3k}=-\um,&t\in\left(t_{3k-1},t_{3k}\right),
        \end{dcases}
    \end{equation}
    where $t_{3k-1}=\frac{3t_{3k}+t_{3k-3}}{4}$ and $t_{3k-2}=\frac{t_{3k}+3t_{3k-3}}{4}$.

    $\forall N\in\N^*$, the ASL between $\vx\left(t_0\right)$ and $\vx\left(t_{3N}\right)$ is denoted as $\switchinglaw_N$, which requires that $\forall k\in\left[N-1\right]$, $\vx_{3k,1:2}=x_{\m2}\ve_2$, and $\vx_{3N}=\vx\left(t_{3N}\right)$. By Proposition \ref{prop:Jacobian}, the Jacobian matrix $\vJ$, except the last column, is 
    \begin{equation}\label{eq:Jacobian_chattering_s2}
        \hspace{-0.1cm}\resizebox{0.95\textwidth}{!}{$
            \left[\begin{array}{ccccccccc}
            -2\um&2\um&-\um&0&0&0&\dots&0&0\\
            -2\um\left(t_3-t_1\right)&2\um\left(t_3-t_2\right)&0&0&0&0&\dots&0&0\\
            -2\um&2\um&0&-2\um&2\um&-\um&\dots&0&0\\
            -2\um\left(t_6-t_1\right)&2\um\left(t_6-t_2\right)&0&-2\um\left(t_6-t_4\right)&2\um\left(t_6-t_5\right)&0&\dots&0&0\\
            \vdots&\vdots&\vdots&\vdots&\vdots&\vdots&\ddots&\vdots&\vdots\\
            -2\um&2\um&0&-2\um&2\um&0&\dots&-2\um&2\um\\
            -2\um\left(t_{3N}-t_1\right)&2\um\left(t_{3N}-t_2\right)&0&-2\um\left(t_{3N}-t_4\right)&2\um\left(t_{3N}-t_5\right)&0&\dots&-2\um\left(t_{3N}-t_{3N-2}\right)&2\um\left(t_{3N}-t_{3N-1}\right)\\
            -2\um\vphi_{3:n}\left(t_{3N}-t_1\right)&2\um\vphi_{3:n}\left(t_{3N}-t_2\right)&\vzero&-2\um\vphi_{3:n}\left(t_{3N}-t_4\right)&2\um\vphi_{3:n}\left(t_{3N}-t_5\right)&\vzero&\dots&-2\um\vphi_{3:n}\left(t_{3N}-t_{3N-2}\right)&2\um\vphi_{3:n}\left(t_{3N}-t_{3N-1}\right)\\
        \end{array}\right].
        $}
    \end{equation}
    Through basic row and column transformations, it holds that $\vJ\leftrightarrow\vJ'$. Considering the rank of columns $\left(k-n+3\right)$ through $k$ of $\vJ'$, \eqref{eq:recursive_chattering_s2} holds.
\end{proof}

\section{An Example of Application of the Proposed State-Centric Necessary Condition}\label{app:comparison}

 Let's introduce Corollary 4 of order $n$ as follows. We will show why it is challenging to derive the existence/non-existence of chattering based on PMP-based necessary conditions, and how to obtain the result based on the proposed state-centric necessary condition. Specifically, we will try to derive the result based on PMP-based/state-centric necessary conditions as examples. Note that the fundamental idea of the state-centric approach has already been applied in our previous work \cite[Section \RomanNum{6}.B]{wang2025chattering}.

\subsection{Problem Statement}

 Consider the following simplified 4th-order chain-of-integrator system:
\begin{equation}\label{eq:problem_n4s2_A}
    \begin{aligned}
        \min\quad&\tf\\
        \text{s.t.}\quad&\dot{x}_1=u,\,\dot{x}_k=x_{k-1},\,k\in\left\{2,3,\dots,n\right\},\\
        &\abs{u}\leq1, x_2\leq 1,\\
        &\vx\left(0\right)=\vx_0,\,\vx\left(\tf\right)=\vx_\f.
    \end{aligned}
\end{equation}

Now, the question is: Is there a $\left(\vx_0,\vx_\f\right)$, s.t. the optimal control is chattering? In other words, can the optimal control switch infinitely many times in a finite interval $\left[0,\tf\right]$?

(For the case where $n=4$, the answer is NO according to Corollary 4 in the manuscript.)

\subsection{Preliminaries} 

 To compare fairly the proposed state-centric necessary condition and PMP-based necessary conditions, we provide some preliminaries for both approaches.

Assume that chattering occurs in the optimal control. Then, $\exists\left\{t_i\right\}_{i\in\N}\subset\left[0,\tf\right]$ increasing strictly monotonically, s.t.
\begin{itemize}
    \item $\left\{t_i\right\}_{i\in\N}$ has a finite limit point, i.e.,
    \begin{equation}\label{eq:chattering_limit_time_A}
        t_\infty\triangleq\lim_{i\to\infty}t_i\in\left[0,\tf\right].
    \end{equation}
    Furthermore, the chattering (unconstrained) arcs converge to a constrained arc at $t_\infty$, i.e.,
    \begin{equation}
        \lim_{t\to t_\infty}x_1=0,\,\lim_{t\to t_\infty}x_2=1.
    \end{equation}
    \item $\forall i\in\N$, $x_2\left(t_i\right)=1$ and $x_1\left(t_i\right)=0$ hold.
    \item $\forall i\in\N$, the control in $\left(t_i,t_{i+1}\right)$ satisfies
    \begin{equation}\label{eq:control_chattering_A}
        u\left(t\right)=\begin{dcases}
            -1,&t\in\left(t_{i},\frac34t_i+\frac14t_{i+1}\right),\\
            1,&t\in\left(\frac34t_i+\frac14t_{i+1},\frac14t_i+\frac34t_{i+1}\right),\\
            -1,&t\in\left(\frac14t_i+\frac34t_{i+1},t_{i+1}\right).
        \end{dcases}
    \end{equation}
\end{itemize}
The proof of the above conclusions can be found in our related work \cite[Section \RomanNum{6}.B]{wang2025chattering}.

The above conclusions are visualized in Figure \ref{fig:n4s2_chattering_reason_A}.

\begin{figure}[!h]
    \centering
    \includegraphics[width=0.3\textwidth]{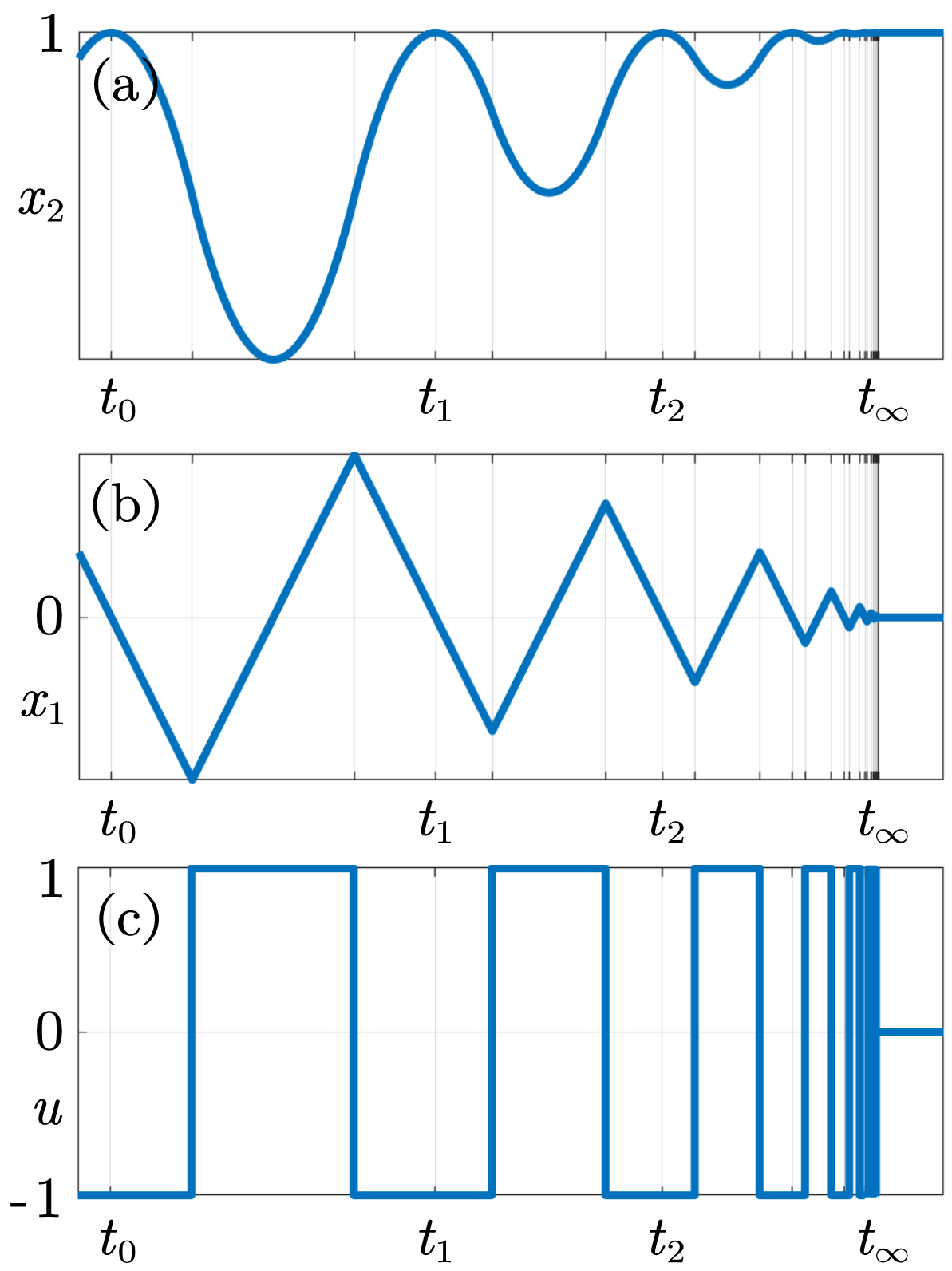}
    \caption{Chattering phenomena in time-optimal control for 4th-order chain-of-integrator systems with constraints on $u$ and $x_2$.}
    \label{fig:n4s2_chattering_reason_A}
\end{figure}

\subsection{Why Deriving the Existence/Non-Existence of Chattering through PMP-Based Necessary Conditions is Challenging?}

 Now, we try to derive the existence/non-existence of chattering based on PMP-based necessary conditions, which will be shown challenging.

The Hamiltonian of problem \eqref{eq:problem_n4s2_A} is
\begin{equation}
    \hamilton\left(\vx\left(t\right),u\left(t\right),\lambda_0,\vlambda\left(t\right),\veta\left(t\right),t\right)
    =\lambda_0+\lambda_1u+\sum_{k=2}^{n}\lambda_k x_{k-1}+\eta\left(x_2-1\right).
\end{equation}
The costate vector $\vlambda$ satisfies the Hamilton's equations, i.e.,
\begin{equation}
    \begin{dcases}
        \dot{\lambda}_2=-\lambda_3+\eta,\\
        \dot{\lambda}_n=0,\\
        \dot{\lambda}_k=-\lambda_{k+1},\,k\in\left\{1,3,4,5,\dots,n-1\right\}.\\
    \end{dcases}
\end{equation}
Furthermore, $\forall t\in\left[0,\tf\right]$, it holds that
\begin{equation}
    \begin{dcases}
        \lambda_0\geq0,\,\left(\lambda_0,\vlambda\left(t\right)\right)\not=\vzero,\\
        \eta\left(t\right)\geq0,\,\eta\left(t\right)\left(x_2\left(t\right)-1\right)=0,\\
        \hamilton\left(t\right)\equiv0.
    \end{dcases}
\end{equation}

PMP implies that it holds a.e. that
\begin{equation}\label{eq:control_PMP_B}
    u\left(t\right)=\begin{dcases}
        -1,&\lambda_1\left(t\right)>0,\\
        0,&\lambda_1\left(t\right)\equiv0 \text{ holds for a period},\\
        1,&\lambda_1\left(t\right)<0.\\
    \end{dcases}
\end{equation}
Specifically, in a period, $u\equiv0$ if and only if a constrained arc occurs where $x_2\equiv1$.

Furthermore, a junction of the costate vector $\vlambda$ occurs when the state constraint $x_2\leq1$ switches between active and inactive. Specifically, at junction time $t_i$, $i\in\N$, it holds that
\begin{equation}\label{eq:costate_junction_A}
    \begin{dcases}
        \lambda_2\left(t_i^+\right)\leq\lambda_2\left(t_i^-\right),\\
        \lambda_k\left(t_i^+\right)=\lambda_k\left(t_i^-\right).
    \end{dcases}
\end{equation}

Consider the trajectory in $\left[t_0,t_\infty\right]$. Denote
\begin{equation}
    \begin{dcases}
        \lambda_{0k}\triangleq\lambda_k\left(t_0^+\right),\,\forall k\in\left\{1,2,\dots,n\right\},\\
        \mu_i=\lambda_2\left(t_i^-\right)-\lambda_2\left(t_i^+\right)\geq0,\,\forall i\in\N^*.
    \end{dcases}
\end{equation}
Then costates in $\left(t_0,t_\infty\right)$ are
\begin{equation}
    \begin{dcases}
        \lambda_k\left(t\right)=\sum_{j=k}^{n}\frac{\lambda_{0j}}{\left(j-k\right)!}\left(t_0-t\right)^{j-k},\,k\in\left\{3,4,\dots,n\right\},\\
        \lambda_2\left(t\right)=\sum_{j=2}^{n}\frac{\lambda_{0j}}{\left(j-2\right)!}\left(t_0-t\right)^{j-2}-\sum_{i=1}^{\infty}\mu_i\delta_{\left(t_i,t_\infty\right)}\left(t\right),\\
        \lambda_1\left(t\right)=\sum_{j=1}^{n}\frac{\lambda_{0j}}{\left(j-1\right)!}\left(t_0-t\right)^{j-1}-\sum_{i=1}^{\infty}\mu_i\left(t_i-t\right)\delta_{\left(t_i,t_\infty\right)}\left(t\right),
    \end{dcases}
\end{equation}
where
\begin{equation}
    \delta_{\left(t_i,t_\infty\right)}\left(t\right)=\begin{dcases}
        1,&t\in\left(t_i,t_\infty\right),\\
        0,&\text{otherwise}.
    \end{dcases}
\end{equation}
In other words, $\forall k\geq3$, $\lambda_k\left(t\right)$ is a polynomial of order $\left(n-k-1\right)$. $\forall k\leq2$, $\lambda_k\left(t\right)$ is a piecewise polynomial of order $\left(n-k-1\right)$ due to the junction condition \eqref{eq:costate_junction_A}.

For convenience, let
\begin{equation}
    \lambda_{1i}\left(t\right)=\sum_{j=1}^{n}\frac{\lambda_{0j}}{\left(j-1\right)!}\left(t_0-t\right)^{j-1}-\sum_{p=1}^{i-1}\mu_p\left(t_p-t\right),\,i\in\N^*.
\end{equation}
Then,
\begin{equation}
    \lambda_{1}\left(t\right)=\lambda_{1i}\left(t\right),\,\forall i\in\N^*,\,t\in\left(t_{i-1},t_{i}\right).
\end{equation}

According to \eqref{eq:control_chattering_A} and \eqref{eq:control_PMP_B}, it holds that
\begin{equation}
    \lambda_{1i}\left(\frac14t_i+\frac34t_{i-1}\right)=\lambda_{1i}\left(\frac34t_i+\frac14t_{i-1}\right)=0,\,\forall i\in\N.
\end{equation}
In other words, $\forall i\in\N$, $\exists\hat{t}_i\not\in\left[t_{i-1},t_{i}\right]$, s.t.
\begin{equation}
    \lambda_{1i}\left(t\right)=\left(t-\frac14t_i-\frac34t_{i-1}\right)\left(t-\frac34t_i-\frac14t_{i-1}\right)p_i\left(t\right),
\end{equation}
where $p_i$ is a polynomial of order no higher than $n-3$. Furthermore, 
\begin{equation}
    \forall t\in\left(t_{i-1},t_i\right),\,p_i\left(t\right)>0.
\end{equation}

Therefore, a sufficient condition for the non-existence of chattering is $\mathcal{C}=\varnothing$, where
\begin{equation}\label{eq:problem_n4s2_A_C}
    \mathcal{C}=\left\{
    \begin{aligned}
        &\left(\left(\mu_i\right)_{i\in\N^*},\left(t_i\right)_{i\in\N},\left(p_i\left(\bullet\right)\right)_{i\in\N},\left(\lambda_{1i}\left(\bullet\right)\right)_{i\in\N}\right):\\
        &\qquad t_{i}<t_{i+1},\,\lim_{i\to\infty}t_i<\infty,\,\lambda_{0n}\not=0,\,\mu_i\geq0,\\
        &\qquad p_i\left(\bullet\right)\text{ is a polynomial of order no higher than }n-3,\\
        &\qquad\forall t\in\left(t_{i-1},t_i\right),\,p_i\left(t\right)>0,\\
        &\qquad\forall t\in\R,\,\lambda_{1i}\left(t\right)=\lambda_{0n}\frac{\left(-1\right)^{n-1}}{\left(n-1\right)!}\left(t-\frac14t_i-\frac34t_{i-1}\right)\left(t-\frac34t_i-\frac14t_{i-1}\right)p_i\left(t\right),\\
        &\qquad\forall t\in\R,\,\lambda_{1\,i+1}\left(t\right)=\lambda_{1i}\left(t\right)+\mu_i\left(t-t_i\right)
    \end{aligned}\right\}.
\end{equation}

It is significantly challenging to prove that $\mathcal{C}_\pm=\varnothing$. A key step is deriving a recursive expression for $\left(t_i\right)_{i\in\N}$. However, the above key step requires eliminating the polynomial $p_i$, which involves more complex computations. As far as we have tried, it is too complex to obtain the recursive expression for $\left(t_i\right)_{i\in\N}$ based on \eqref{eq:problem_n4s2_A_C}. In contrast, our proposed Corollary 4 can directly derive the recursive expression for $\left(t_i\right)_{i\in\N}$ in a compact and simple form.

\subsection{Why Deriving the Existence/Non-Existence of Chattering through the Proposed State-Centric Necessary Condition is Possible?}

 Denote
\begin{equation}
    \tau_i=t_\infty-t_{i}.
\end{equation}
Then, chattering is allowed only if $\tau_0\in\R_{++}$ and $\lim_{i\to\infty}\tau_i=0$.

According to Corollary 4, the recursive expression for $\left(\tau_i\right)_{i\in\N^*}$ is
\begin{equation}\label{eq:problem_n4s2_A_tau}
    \det\left(\left(\tau_{i-k}+3\tau_{i-k-1}\right)^{j+1}-3\left(3\tau_{i-k}+\tau_{i-k-1}\right)^{j+1}+3\left(\tau_{i-k+1}+3\tau_{i-k}\right)^{j+1}-\left(3\tau_{i-k+1}+\tau_{i-k}\right)^{j+1}\right)_{j\in\left[n-2\right],k\in\left[n-2\right]}=0.
\end{equation}
Once the initial value $\left(\tau_i\right)_{i=1}^{n-2}$ is given, the whole sequence $\left(\tau_i\right)_{i\in\infty}$ is determined. In other words, a sufficient condition for the non-existence of chattering is $\widehat{\mathcal{C}}=\varnothing$, where
\begin{equation}\label{eq:problem_n4s2_A_hatC}
    \widehat{\mathcal{C}}=\left\{\left(\tau_i\right)_{i\in\N^*}:\tau_i>\tau_{i+1}>0,\,\lim_{i\to\infty}\tau_i=0,\,\text{\eqref{eq:problem_n4s2_A_tau} holds}\right\}.
\end{equation}

Evidently, the condition on \eqref{eq:problem_n4s2_A_hatC} is much simpler than the condition on \eqref{eq:problem_n4s2_A_C}. Specifically, the condition on \eqref{eq:problem_n4s2_A_hatC} requires analyzing the convergence of a series for the recursive expressions \eqref{eq:problem_n4s2_A_tau} in the form of a system of polynomial equations.

\subsection{Deriving the Non-Existence of Chattering through the Proposed State-Centric Necessary Condition in 4th-Order Problem}

 Based on the proposed state-centric necessary condition, we can even directly prove that $\widehat{\mathcal{C}}=\varnothing$ when $n=4$; hence, chattering does not exist in the optimal control of problem \eqref{eq:problem_n4s2_A}. Note that the above claim is a novel theoretical result.

Let $n=4$. Then, the recursive expression \eqref{eq:problem_n4s2_A_tau} for $\left(\tau_i\right)_{i\in\N^*}$ is
\begin{equation}
    \begin{aligned}
        &\det\left(\left(\tau_{i-k}+3\tau_{i-k-1}\right)^{j+1}-3\left(3\tau_{i-k}+\tau_{i-k-1}\right)^{j+1}+3\left(\tau_{i-k+1}+3\tau_{i-k}\right)^{j+1}-\left(3\tau_{i-k+1}+\tau_{i-k}\right)^{j+1}\right)_{j\in\left[2\right],k\in\left[2\right]}\\
        =&144\left(\tau_{i-1}-\tau_{i-3}\right)\left(\tau_{i}-\tau_{i-2}\right)f\left(\tau_{i-1}-\tau_{i};\tau_{i-2}-\tau_{i-1},\tau_{i-3}-\tau_{i-2}\right)=0,
    \end{aligned}
\end{equation}
where
\begin{equation}
    f\left(z;z_1,z_2\right)=\left(z_1-z_2\right)z^2+\left(z_1^2+z_1z_2-z_2^2\right)z-z_2\left(z_1^2+z_1z_2-z_2^2\right).
\end{equation}
Note that $\tau_{i}<\tau_{i-1}<\tau_{i-2}<\tau_{i-3}$. Let $z_i=\tau_{i-1}-\tau_{i}>0$. Then,
\begin{equation}\label{eq:problem_n4s2_A_fz0}
    f\left(z_i;z_{i-1},z_{i-2}\right)=0.
\end{equation}

$\lim_{i\to\infty}\tau_i=0$ implies that $\lim_{i\to\infty}z_i=0$. Without loss of generality, let $z_1>z_2$. Then,
\begin{equation}
    \begin{dcases}
        f\left(z_2;z_1,z_2\right)=\left(z_1-z_2\right)^2z_2>0,\\
        f\left(0;z_1,z_2\right)=-z_2\left(z_1^2+z_1z_2-z_2^2\right)<0.
    \end{dcases}
\end{equation}
Note that $f\left(z;z_1,z_2\right)$ is a quadratic function of $z$. Therefore, $f\left(z;z_1,z_2\right)=0$ has two roots in $\left(-\infty,0\right)$ and $\left(0,z_2\right)$, respectively. Since $z_3>0$ and $f\left(z_3;z_1,z_2\right)=0$, it holds that $0<z_3<z_2<z_1$. For the same reason recursively, it holds that
\begin{equation}
    0<z_i<z_{i-1},\,\forall i\geq2.
\end{equation}

Let $r_i=1-\frac{z_{i+1}}{z_i}\in\left(0,1\right)$. Then, \eqref{eq:problem_n4s2_A_fz0} implies that
\begin{equation}\label{eq:problem_n4s2_A_ri}
    g\left(r_{i+1},r_i\right)\triangleq r_{i+1}^2-\left(3+\frac{1}{r_i\left(1-r_i\right)}\right)r_{i+1}+1=0.
\end{equation}
Note that
\begin{equation}
    \begin{dcases}
        g\left(0,r_i\right)=1>0,\\
        g\left(r_i,r_i\right)=-\frac{r_i\left(2-r_i\right)^2}{1-r_i}<0,\\
        g\left(1,r_i\right)=-1-\frac{1}{r_i\left(1-r_i\right)}<0.
    \end{dcases}
\end{equation}
\eqref{eq:problem_n4s2_A_ri} has two roots in $\left(0,r_i\right)$ and $\left(1,+\infty\right)$, respectively. Hence, $0<r_{i+1}<r_i<1$. Denote $r_\infty\triangleq\lim_{i\to\infty}r_i\in\left[0,1\right]$. Then, $g\left(r_\infty,r_\infty\right)$ implies that $r_\infty=0$. Note that
\begin{equation}
    r_{i+1}=\frac{1}{2}\left[-\left(3+\frac{1}{r_i\left(1-r_i\right)}\right)+\sqrt{\left(3+\frac{1}{r_i\left(1-r_i\right)}\right)^2-4}\right]=r_i-4r_i^2+\mathcal{O}\left(r_i^3\right),\,i\to\infty.
\end{equation}
By Stolz-Ces\`aro theorem, it holds that
\begin{equation}
    \lim_{i\to\infty}i\left(\frac{z_i}{z_{i+1}}-1\right)=\lim_{i\to\infty}\frac{ir_i}{1-r_i}=\lim_{i\to\infty}\frac{i}{\frac{1}{r_i}}=\lim_{i\to\infty}\frac{i+1-i}{\frac{1}{r_{i+1}}-\frac{1}{r_i}}=\lim_{i\to\infty}\frac{r_i^2\left(1+\mathcal{O}\left(r_i\right)\right)}{4r_i^2+\mathcal{O}\left(r_i^3\right)}=\frac{1}{4}\in\left(0,1\right).
\end{equation}
By Raabe-Duhamel's test, it holds that
\begin{equation}
    \infty=\sum_{i=1}^{\infty}z_i=\tau_0-\tau_\infty=t_\infty-t_0<\infty,
\end{equation}
which is a contradiction.

Therefore, the set $\widehat{\mathcal{C}}$ defined in \eqref{eq:problem_n4s2_A_hatC} is empty. Therefore, chattering does not exist in the optimal control of problem \eqref{eq:problem_n4s2_A} when $n=4$.

\subsection{Summary}

 In this appendix, we provide an example to compare the proposed state-centric necessary condition and PMP-based necessary conditions.
\begin{enumerate}
    \item We show that it is challenging to derive the existence/non-existence of chattering based on PMP-based necessary conditions in the considered optimal control problem. The key step is to derive the recursive expression for junction times, which involves complex computations.
    \item We show that the proposed state-centric necessary condition can directly derive the recursive expression for junction times in a compact and simple form, which is much simpler than the condition based on PMP-based necessary conditions. Specifically, we prove that chattering does not exist in the considered optimal control problem of order 4.
\end{enumerate}

\end{document}